\documentclass[12pt]{amsart} 
\usepackage{amsmath,amsthm, amsxtra, amsfonts, amssymb,color}
\usepackage{tikz}
\usetikzlibrary{calc,arrows.meta,decorations.pathreplacing}
	\usetikzlibrary{decorations.markings}

\usepackage{graphicx}
\usepackage[hidelinks,backref=page]{hyperref} 
\hypersetup{
	colorlinks,
	citecolor=magenta,
	filecolor=magenta,
	linkcolor=blue,
	urlcolor=black
}
\usepackage[top=2cm, bottom=2cm, right=2cm, left=2cm]{geometry}
\usepackage{thmtools}
\usepackage[capitalize]{cleveref}

\def\J{{\mathcal{J}}}
\def\K{{\mathcal{K}}}
\def\O{{\mathcal{O}}}

\def\std{{\rm std}}

\def\m{{\bf{m}}}
\def\cR{{\mathcal{R}}}
\def\QQ{{\mathbb Q}}
\def\R{{\mathbb R}}
\def\C{{\mathbb C}}
\def\Z{{\mathbb Z}}
\def\T{{\mathbb{T}}}
\def\A{{\mathcal{A}}}
\def\B{{\mathcal{B}}}

\def\a{{\mathbf{a}}}
\def\b{{\mathbf{b}}}
\def\bb{{\bar \b}}
\def\r{{\mathbf{r}}}
\def\v{{\mathbf{v}}}
\def\x{{\mathbf{x}}}
\def\ty{{\tilde \y}}
\def\y{{\mathbf{y}}}
\def\z{{\mathbf{z}}}
\def\h{{\mathfrak{h}}}
\def\t{{\mathfrak{t}}}
\def\bp{{\bar{\tropP}}}

\def\GL{{\rm GL}}

\def\conv{{\rm conv}}

\def\Trop{{\rm Trop}}
\def\P{{\mathbb{P}}}
\def\Gr{{\rm Gr}}
\def\Dr{{\rm Dr}}
\def\Q{{\vec{Q}}}
\def\SL{{\rm SL}}

\def\PGL{{\rm PGL}}
\def\sp{{\rm span}}

\def\minconv{{\rm minconv}}
\def\maxconv{{\rm maxconv}}
\def\Lin{{L}}
\def\Le{{\Sigma}}
\def\0{{\bf 0}}
\def\val{{\rm val}}
\def\vn{{\t}}

\def\pluck{p}
\def\dist{\delta}

\def\PK{{\rm PK}}
\def\ncyc{{\rm ncyc}}

\def\one{{\bf 1}}

\def\str{{\eta}}

\newcommand{\iotaFInc}{{\iota}}  
\newcommand{\iotaStd}{{\iota}}   
\newcommand{\tropP}{{\pi}}      

\newtheorem{theorem}{Theorem}[section]
\newtheorem{lemma}[theorem]{Lemma}
\newtheorem{proposition}[theorem]{Proposition}
\newtheorem{thm}[theorem]{Theorem}

\newtheorem{corollary}[theorem]{Corollary}

\newtheorem{example}[theorem]{Example}
\newtheorem{definition}[theorem]{Definition}

\newtheorem{rem}[theorem]{Remark}
\newtheorem{remark}[theorem]{Remark}
\newtheorem{convention}[theorem]{Convention}

\newenvironment{nouppercase}{%
	\renewcommand{\uppercasenonmath}[1]{}}{}
\newcommand\defn[1]{\emph{#1}}

\author{Nick Early}
\author{Thomas Lam}
\address{School of Natural Sciences, Institute for Advanced Study, 1 Einstein Dr. Princeton, NJ 08540, USA.}
\email{\href{mailto:earlnick@ias.edu}{earlnick@ias.edu}}
\address{Department of Mathematics, University of Michigan, 2074 East Hall, 530 Church Street, Ann Arbor, MI 48109-1043, USA}
\email{\href{mailto:tfylam@umich.edu}{tfylam@umich.edu}}

\numberwithin{equation}{section}

\begin{document}
	\title{Scaffolds for Higher Tropical Grassmannians: Foundations}
	\begin{nouppercase}
		\maketitle
	\end{nouppercase}
	
\begin{abstract}
\emph{Scaffolds} are the one-dimensional skeleta of high-dimensional flag simplicial complexes of nonpositive curvature. They generalize the phylogenetic trees of $\Trop\,\Gr(2,n)$  to arbitrary $k$, drawing together $\SL(k)$-web bases, affine buildings, the combinatorics of the positive tropical Grassmannian and low-dimensional topology. We prove that scaffolds model points in all tropical Grassmannians via a $k$-point distance function.

In this paper, we study in detail CAT(0) planar graphs, which are positive scaffolds for the tropical Grassmannian of three-planes.  CAT(0) planar graphs are directed versions of the diskoids of Fontaine-Kamnitzer-Kuperberg, planar dual to $\SL(3)$-webs.  Our main result is the construction of a unique representation of any given integer positive tropical Pl\"ucker vector by a normal CAT(0) planar graph.  We show that any normal CAT(0) planar graph embeds into the tropical linear space as a Lam-Postnikov membrane, and embeds into the Keel-Tevelev membrane within the affine building.  We show that Early's planar basis expansion can be computed directly from the strand combinatorics of the dual web, and connect this expansion to Petersen-Pylyavskyy-Speyer's noncrossing tableaux, explored further in our companion paper.  

\end{abstract}

	\tableofcontents

	\section{Introduction}
	In 2004, Speyer and Sturmfels \cite{SS} introduced the tropical Grassmannian, a $k(n-k)+1$ dimensional polyhedral fan in $\R^{\binom{[n]}{k}}$.  Just as the Grassmannian is ubiquitous in mathematics, the tropical Grassmannian has emerged as a fundamental object in tropical geometry, and more broadly in combinatorial algebraic geometry with applications to mathematical biology and theoretical physics.
		
	The tropical Grassmannian $\Trop\, \Gr(2,n)$ of two-planes can be identified with the moduli space of phylogenetic trees, and plays an important role in mathematical biology.  This moduli space can also be viewed as the tropicalization $\Trop \, M_{0,n}$ of the moduli space $M_{0,n}$ of $n$-pointed rational curves.  Whereas the combinatorial study of the tropical moduli spaces $\Trop\,M_{g,n}$ of higher genus curves has flourished in the last decade, the study of higher tropical Grassmannians $\Trop\, \Gr(k,n)$ in recent years has been relatively subdued, especially when it comes to explicit combinatorial models; see \cite{HJJS,CGUZ,FL}.  The aim of this work is to propose a new combinatorial model for higher tropical Grassmannians: \emph{scaffolds}, built from the bones of tropical linear spaces.
	
\medskip
Scaffolds are pairs $(Q,\z)$ consisting of a directed graph $Q$ equipped with $n$ distinguished vertices $\z = (z_1,\ldots,z_n)$, satisfying conditions spelled out in \cref{def:normalmodel,def:scaffold}.  The tropical Pl\"ucker vector $\tropP_\bullet$ of a point in $\Trop \, \Gr(k,n)$ is obtained from a scaffold $(Q,\z)$ by using a $k$-point distance function that we call the \emph{Fermat-Le distance function}:
$$
\Le_Q(i_1,\ldots,i_k) =\Le_Q(I) := \min_{x \in V(Q)} \left( \sum_{a=1}^k \delta(x,z_{i_a})\right),
$$
where $\delta(\cdot,\cdot)$ is the usual directed distance function of the digraph $Q$, illustrated in \cref{fig:Q1} below.  This distance function appears in the work of Le--O'Dorney, Le, and Fraser--Le \cite{LO,Le,FL} on higher laminations in affine buildings.

\medskip
The name \emph{scaffold} reflects multiple aspects of the story. First, we expect a scaffold to be the one-dimensional \emph{skeleton} of a $(k-1)$-dimensional simplicial complex of locally nonpositive curvature.  Second, as we discuss below, tropical Pl\"ucker vectors, matroid decompositions, and tropical linear spaces can all be \emph{reconstructed} from a scaffold.  Finally, scaffolds often embed inside the \emph{buildings} of Bruhat--Tits.

\medskip

In the case $k = 2$, the graphs $(Q,\z)$ are exactly the metric or phylogenetic trees, where the distinguished vertices $\z$ are the leaves of the tree.  The main result of this paper is a complete combinatorial classification of positive scaffolds for the positive tropical Grassmannian $\Trop_{>0}\Gr(3,n)$ of three-planes.  The positive tropical Grassmannian $\Trop_{>0} \Gr(k,n) \subset \Trop\,\Gr(k,n)$ is a subfan introduced by Speyer and Williams \cite{SW}, motivated by total positivity.  

For $k=3$, positive scaffolds take a concrete combinatorial form: they are the normal \emph{CAT(0) planar graphs} $(Q,\z)$ (\cref{def:normal}).

We summarize our main theorem as follows.
\begin{theorem}[{\cref{thm:main}}]\label{thm:uniquenormal}
Every integer point in $\Trop_{>0} \Gr(3,n)$, modulo lineality, is modeled by a unique cyclic-less normal CAT(0) planar graph $(Q,\z)$.
\end{theorem}

CAT(0) planar graphs are the dual quivers (cf. Postnikov \cite{Pos} and Shen--Sun--Weng \cite{SSW}) to the non-elliptic $\SL(3)$-webs in the sense of Kuperberg.  They are an orientation of the one-skeleton of the diskoids of Fontaine--Kamnitzer--Kuperberg \cite{FKK}.  

We demonstrate the utility of scaffolds through several structural results:
\begin{enumerate}
\item We show that the normal CAT(0) planar graphs of \cref{thm:uniquenormal} embed into the tropical linear spaces of Speyer \cite{Spe} as a membrane in the sense of Lam and Postnikov \cite{LP}.  This allows us to relate (\cref{thm:LPKT}) the membranes of Lam and Postnikov to the membranes of Keel and Tevelev \cite{KT}.
\item Associated to a point $\tropP_\bullet \in \Trop \, \Gr(k,n)$ is a matroid decomposition of the hypersimplex $\Delta(k,n)$.  We show (\cref{sec:LPKT}) that these matroids are recovered from the combinatorics of distance minimizers in $(Q,\z)$.
\item We give an explicit combinatorial formula (\cref{thm:pbexpansion}) for the expansion of $\tropP_\bullet(Q,\z)$ in terms of the planar basis of Early \cite{E22,Ear24}.  This solves an open problem concerning the construction of Feynman diagrams for the $\mathbb{P}^2$ case of generalized biadjoint scalar amplitudes $m^{(3)}_n$ of Cachazo-Early-Guevara-Mizera \cite{CEGM}, providing a path to the study of their singularities and residues. 
\item As another application, we prove (\cref{prop:RSV}) the conjectured recursion for the (tropical) symbol alphabet of Ren--Spradlin--Volovich \cite{RSV} in the case $k=3$.  Our results here rely on a duality theorem from our companion paper \cite{ELnc}.
\end{enumerate}
A normal CAT(0) planar graph representing a point in $\Trop_{>0} \, \Gr(3,12)$ is illustrated in \cref{fig:312canvas}.  We remark that our proof of the main theorem does not invoke the machinery of affine buildings.

\begin{figure}[h!]
	\centering
	\includegraphics[width=0.49\linewidth]{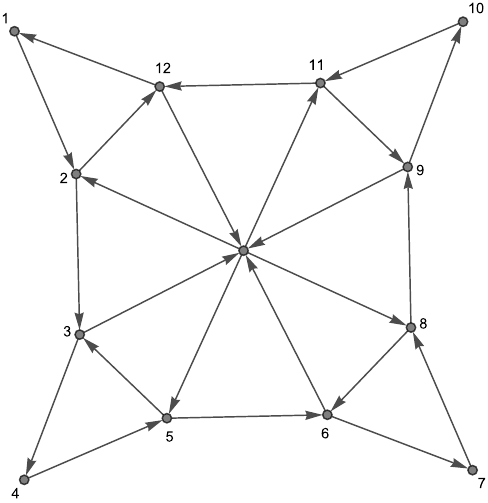}	\includegraphics[width=0.49\linewidth]{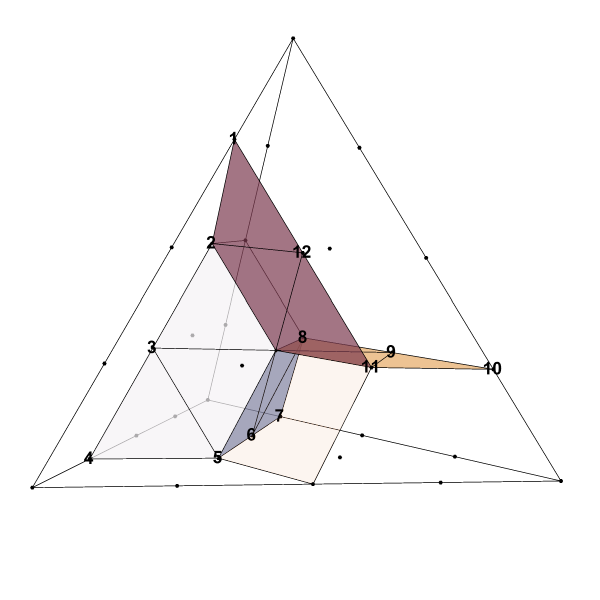}
	\caption{A scaffold embeds isometrically in the tropical linear space, illustrated on the right by an embedding in $\mathbb{R}^3$ with all triangles equilateral.  In particular, note that the edge length of the bounding simplex is four, which is the PK weight of the positive tropical Pl\"ucker vector, which is also the number of columns in the noncrossing tableau; see our companion paper \cite{ELnc}.  The scaffold on the left will be a running example through its dual web in \cref{fig:web-examples}(c), the planar basis expansion of its positive tropical Pl\"ucker vector in \cref{example: 12-gonPlucker} and the embedding into the $\PGL(3)$ affine building in \cref{fig:312buildingex}.}
\label{fig:312canvas}
\end{figure}

We turn now to discuss our combinatorial model in the context of the following set-theoretic inclusions of polyhedral fans:
\begin{equation}\label{eq:3fans}
\Trop_{>0} \Gr(k,n) \subset \Trop\, \Gr(k,n) \subset \Dr(k,n).
\end{equation}
The \emph{Dressian} $\Dr(k,n)$ is the space of all tropical Pl\"ucker vectors, also known as valuated matroids (see \cref{ssec:tropPluck}).

\smallskip
\noindent
(A) We expect that describing explicit scaffolds for $\Trop_{>0}\Gr(k,n)$ is equivalent to the combinatorial-representation-theoretic problem of constructing explicit $\SL(k)$-web bases.  In future work, we plan to describe explicitly scaffolds for $\Trop_{>0} \Gr(4,n)$, noting the recently discovered hourglass plabic graph $\SL(4)$-web basis \cite{GPPSS}.  The case $k = 4$ is of particular interest because of its applications to scattering amplitudes in physics \cite{ALS2,DFGK,HP}.  The positive tropical Grassmannians $\Trop_{>0} \Gr(4,n)$ have been conjectured to control the symbol of $N=4$ super Yang-Mills amplitudes.

We expect that in general scaffolds for $\Trop_{>0} \Gr(k,n)$ are one-skeleta of $(k-1)$-dimensional complexes satisfying a local curvature condition.  Explicit descriptions of these scaffolds could yield an $\SL(k)$-web basis for $k\geq 5$.

\smallskip
\noindent

(B) Whereas $\Trop_{>0} \Gr(2,n)$ is the moduli space of planar phylogenetic trees, the full tropical Grassmannian $\Trop \, \Gr(2,n)$ is the moduli space of all phylogenetic trees.  We establish an analogous (but partial) result for $k=3$: we show in \cref{thm:main} that every labeled CAT(0) planar graph $(\Q, \z)$ gives a point in $\Dr(3, n)$, where the labeled vertices $\z$ can now be placed anywhere in the graph - not necessarily in cyclic order, and not necessarily on the boundary.  

We show in \cref{thm:modelexists} that every integer point in $\Trop \, \Gr(k,n)$ has a weak scaffold, up to lineality.  This is based on the connection between the tropical Grassmannian $\Trop \, \Gr(k,n)$ and the affine building of $\PGL(k)$, and in particular on the results of Le \cite{Le}.  However, the scaffolds of \cref{thm:modelexists} are not combinatorially explicit.  In principle, they can be constructed algorithmically using the geometry of the affine building, but we leave as an open problem how to construct models that are canonical and combinatorially tractable.

\smallskip
\noindent
(C) We do not know whether every integer point $\tropP_\bullet \in \Dr(k,n)$ in the Dressian has a scaffold in our sense (\cref{sec:tropmodels}).  Nevertheless, we explain in \cref{subsec:decomposition} that the bounded subcomplex $B(\tropP_\bullet)$ of the tropical linear space $L(\tropP_\bullet)$ provides a natural minimizer description of $\tropP_\bullet$.  The bounded subcomplex captures half of the data in a scaffold $(Q,\z)$: it encodes the distance function but not the edge structure.

\medskip
Let us remark that our work has close relations to duality conjectures in cluster algebras, and to intersection pairings for higher laminations.  We hope to return to this in the future.

\medskip
Both authors of the present work were led to the study of tropical Grassmannians by questions from scattering amplitudes and related higher structures in theoretical physics.  We have focused on mathematical aspects in this work, and will return to physical applications in future works.  For a broad survey of the developing interface between combinatorial geometry and scattering amplitudes, see \cite{Lam}.

\medskip
\noindent
{\bf Organization.}
In \cref{part:main}, we state the main definitions and results.  In \cref{sec:trop,sec:tropmodels}, we recall the tropical Grassmannian and its properties, and introduce scaffolds.  In \cref{sec:definitions}, we introduce CAT(0) planar graphs and state our main theorem.  In \cref{sec:webs}, we introduce normal $\SL(3)$-webs, discuss their strand combinatorics, and state the planar basis expansion theorem (\cref{thm:pbexpansion}).  In \cref{sec:LPKT}, we recall the membranes of Lam--Postnikov and Keel--Tevelev and explain the relation with scaffolds.  In \cref{sec:RSV}, we apply our results to symbol alphabet recursion.

In \cref{part:CAT0}, we discuss the discrete and metric geometry of CAT(0) planar graphs.  We start by analyzing the behavior of the combinatorial distance function (\cref{sec:combinatorial-structure}) and directed distance function (\cref{sec:directed-distance}).  In \cref{sec:focal-points}, we define the crucial notion of \emph{focal point} which gives a local characterization of distance minimizers (\cref{thm:focal}).  In \cref{sec:mainproof}, we prove one of our main results that the Fermat-Le distance function of a CAT(0) planar graph is a tropical Pl\"ucker vector.  In \cref{sec:simplenormal}, we discuss normal CAT(0) planar graphs.

In \cref{part:webs}, we study strand combinatorics in non-elliptic $\SL(3)$-webs, and study properties of the class of normal webs.  We then prove the planar basis expansion theorem using a calculation involving tropical planar cross-ratios.  A key input is the duality theorem from \cite{ELnc}.

In \cref{part:noncrossing}, we establish a bijection between noncrossing tableaux and normal webs.  To do so, we recall then generalize Tymoczko's simple bijection for $\SL(3)$-webs \cite{Tym}.  In \cref{sec:proof-main-3} we complete the proof of our main result by establishing bijections between $\Trop_{>0} X(3,n)(\Z)$, the set of cyclic-less noncrossing tableaux, and the set of normal cyclic-less webs.

In \cref{part:buildings}, we turn to buildings.  After recalling results of Le \cite{Le}, we prove \cref{thm:modelexists} stating that any integer point in $\Trop \, \Gr(k,n)$ has a weak scaffold.  In \cref{sec:diskoid}, we explain how our construction is compatible with the works of Fontaine--Kamnitzer--Kuperberg \cite{FKK} and Akhmejanov \cite{Akh}.  In \cref{sec:KT} we recall Keel and Tevelev's isomorphism between tropical linear space and the membrane, and establish the compatibility with our construction.  In \cref{sec:positroids}, we give a more explicit description of the positroids that appear for normal CAT(0) planar graphs.  

In \cref{sec:appendix}, we tabulate the rays of $\Trop_{>0} X(3,n)$ for $ n = 7,8$.

\medskip
\noindent
{\bf Acknowledgements.}
We thank Ian Le for helpful conversations.  We thank Freddy Cachazo for his enthusiasm and interest in this work.  We thank the Erwin Schr\"odinger Institute and the organizers of the ESI program ``Amplitudes and algebraic geometry" for supporting a visit where part of this work was done.

N.E. was funded by the European Union (ERC, UNIVERSE PLUS, 101118787). \begin{tiny}
	Views and opinions expressed are however those of the author(s) only and do not necessarily reflect those of the European Union or the European Research Council Executive Agency. Neither the European Union nor the granting authority can be held responsible for them.
\end{tiny}
T.L. was supported by the National Science Foundation under Grant No. DMS-2348799.
	
	\part{Main results and definitions}\label{part:main}
	\section{Tropical Grassmannian}\label{sec:trop}
	Throughout, we fix integers $2 \le k < n$.  We write $[n] = \{1, \ldots, n\}$ and $\binom{[n]}{k}$ for the set of $k$-subsets of $[n]$.  Indices in $[n]$ are taken modulo $n$ when working with cyclic structures.  A \emph{cyclic interval} is a subset of the form $\{a, a+1, \ldots, a+m-1\} \pmod{n}$; a $k$-subset is \emph{non-cyclic} if it is not a cyclic interval.  We write $\binom{[n]}{k}^{\ncyc}$ for the set of non-cyclic $k$-subsets and note that $\left|\binom{[n]}{k}^{\ncyc}\right| = \binom{n}{k} - n$.  For $J \subseteq [n]$, we set $e_J := \sum_{j \in J} e_j \in \R^n$, where $e_1, \ldots, e_n$ are the standard basis vectors.  When there is no ambiguity, we write $k$-subsets without braces or commas (e.g., $ijk$ for $\{i,j,k\}$).  For $J \in \binom{[n]}{k}$, let $e^J \in \R^{\binom{[n]}{k}}$ denote the corresponding standard basis vector.
	
	\subsection{Tropical Grassmannian}
Let $\Gr(k,n)$ denote the Grassmannian of $k$-planes in $n$-space, and let $\Gr^\circ(k,n) \subset \Gr(k,n)$ denote the subspace where all Pl\"ucker coordinates are non-vanishing.  The torus $T \subset \GL(n)$ acts on $\Gr(k,n)$ preserving $\Gr^\circ(k,n)$.  The quotient $X(k,n):= \Gr^\circ(k,n)/T$ is free, called the configuration space of $n$ points in $\P^{k-1}$ in general linear position.

Let $\cR$ denote the ring of Puiseux series, with valuation $\val: \cR \to \QQ \cup \{\infty\}$.  The tropical Pl\"ucker vector $\tropP_\bullet(V)$ of $V \in \Gr^\circ(k,n)(\cR)$ is given by
$$
\tropP_I(V) =  \val (\pluck_I(V)) \in \QQ, \qquad I \in \binom{[n]}{k},
$$
where $\pluck_I(V)$ denotes a Pl\"ucker coordinate of $V$.  The \emph{tropical Grassmannian} $\Trop\, \Gr(k,n)$ is given by
$$
\Trop \, \Gr(k,n) := \overline{\{\tropP_\bullet(V) \mid V \in \Gr^\circ(k,n)(\cR) \}} \subset \R^{\binom{[n]}{k}},
$$
which we view as a polyhedral fan of dimension $k(n-k)+1$; strictly speaking, it is the tropicalization of the cone over $\Gr^\circ(k,n)$.  The positive tropical Grassmannian $\Trop_{>0} \Gr(k,n) \subset \Trop\, \Gr(k,n)$ is a subfan obtained as tropical Pl\"ucker vectors of points in $\Gr^\circ(k,n)(\cR_{>0})$, where $\cR_{>0}$ denotes a ``positive" semifield.  We will characterize this subfan in \cref{ssec:tropPluck}.
	
	The tropical configuration space $\Trop \, X(k,n)$ is the dimension $(k-1)(n-k-1)$ quotient of $\Trop\, \Gr(k,n)$ by its lineality space 
	\begin{equation}\label{eq:lin}
	\Lin_{k,n} := \sp( \sum_{I \ni i} e^I, i= 1,2,\ldots,n) \subset\R^{\binom{[n]}{k}},
	\end{equation}
	where $e^I$, $I \in \binom{[n]}{k}$ denote the standard basis of $\R^{\binom{[n]}{k}}$.  
	The positive tropical configuration space $\Trop_{>0} X(k,n)$ is the image of $\Trop_{>0} \Gr(k,n)$ in $\Trop \, X(k,n)$.  A point in $\Trop \, \Gr(k,n)$ is \emph{integral} if it belongs to $\Z^{\binom{[n]}{k}} \subset \R^{\binom{[n]}{k}}$.  A point in $\Trop \, X(k,n)$ is \emph{integral} if there exists a representative in $\Trop \, \Gr(k,n)$ that is integral.
	
	\subsection{Tropical Pl\"ucker vectors}\label{ssec:tropPluck}
	A vector $\tropP_\bullet \in \R^{\binom{[n]}{k}}$ is a \defn{tropical Pl\"ucker vector} if it satisfies the (3-term) \emph{tropical Pl\"ucker relation}:
	\begin{equation}\label{eq:tropPlucker}
		\min(\tropP_{Sab}+\tropP_{Scd}, \tropP_{Sac}+\tropP_{Sbd},\tropP_{Sad}+\tropP_{Sbc}) \text{ occurs twice}
	\end{equation}
	for any $S \in \binom{[n]}{k-2}$ and $a<b<c<d$ in $[n] \setminus S$. 
	
	A vector $\tropP_\bullet \in \R^{\binom{[n]}{k}}$ is a \defn{positive tropical Pl\"ucker vector} if it satisfies the \emph{positive tropical Pl\"ucker relation}:
	\begin{equation}\label{eq:postropPlucker}
		\tropP_{Sac}+\tropP_{Sbd} = \min(\tropP_{Sab}+\tropP_{Scd}, \tropP_{Sad}+\tropP_{Sbc}),
	\end{equation}
	for any $S \in \binom{[n]}{k-2}$ and $a<b<c<d$ in $[n] \setminus S$.  It is clear that positive tropical Pl\"ucker vectors are tropical Pl\"ucker vectors.
	
	The space of tropical Pl\"ucker vectors is called the \emph{Dressian} and denoted $\Dr(k,n)$.  We have the set-theoretic containment $\Trop\,\Gr(k,n) \subseteq \Dr(k,n)$ which is typically strict \cite{HJJS}.   In contrast, the space of positive tropical Pl\"ucker vectors is identical to the positive tropical Grassmannian $\Trop_{>0} \Gr(k,n)$ by \cite{ALS,SW20}.  Thus we obtain the inclusions of the three fans in \eqref{eq:3fans}.
	
	For a (positive) tropical Pl\"ucker vector $\tropP_\bullet$ and $\x \in \R^n$, define $\tropP^\x_\bullet$ by
	$$
	\tropP^\x_I := \tropP_I - \sum_{i \in I} x_i \qquad \text{for } I \in \binom{[n]}{k}, 
	$$
	which is again a (positive) tropical Pl\"ucker vector.  This is called the lineality action of $\x$ on $\tropP_\bullet$.  We use the notation $\tropP_\bullet \equiv \tropP_\bullet^\x$ to denote tropical Pl\"ucker vectors that are equal modulo lineality.
	
	\subsection{Matroid and positroid decompositions of the hypersimplex}
	Let $M$ be a matroid on $[n]$.  Recall that the matroid polytope $P_M$ is the convex hull in $\R^n$ of the vectors $e_I:=e_{i_1} + \cdots + e_{i_k}$ as $I = \{i_1,\ldots,i_k\}$ varies over bases of $M$.  A tropical Pl\"ucker vector (resp. positive tropical Pl\"ucker vector), when viewed as a height vector, induces a regular subdivision of the hypersimplex $\Delta(k,n)$ into matroid (resp. positroid) polytopes.  See \cite{Spe,ALS,SW20}.  This subdivision can be defined as follows.
	
	\begin{lemma}\label{lem:mintropvec}
		Let $\tropP_\bullet$ be a tropical Pl\"ucker vector (resp. positive tropical Pl\"ucker vector) and set $\alpha = \min_I \tropP_I$.  Then 
		$$
		M(\tropP_\bullet) := \{I \mid \tropP_I = \alpha\}
		$$
		is a matroid (resp. positroid).
	\end{lemma}
	\begin{proof}
		A tropical Pl\"ucker vector $\tropP_\bullet$, viewed as a height function on the vertices $e_I$ of $\Delta(k,n)$, induces a regular subdivision into matroid polytopes \cite{Spe} (resp. positroid polytopes \cite{ALS,SW20}).  The set $M(\tropP_\bullet)$ corresponds to the face of this subdivision selected by the linear functional $\sum x_i = \alpha$.  Since every face of a matroid (resp. positroid) subdivision is a matroid (resp. positroid) polytope, the claim follows.
	\end{proof}
	\noindent Let $\Delta(\tropP_\bullet)$ denote the set of matroid polytopes $P_{M(\tropP_\bullet^\x)}$ as $\x$ varies over $\R^n$.  Geometrically, the height function $\tropP_\bullet$ is affine linear over each matroid polytope $P_M$ in the subdivision, and for any such $P_M$ there exists $\x \in \R^n$ such that $\tropP_I^\x = \tropP_I - \sum_{i \in I} x_i$ is constant on the bases of $M$ and achieves its minimum there.  Thus $M = M(\tropP_\bullet^\x)$, and every cell of the subdivision arises in this way.  This collection of matroids is finite and determines a regular matroid subdivision $\Delta(\tropP_\bullet)$ of $\Delta(k,n)$.  We have $\Delta(\tropP_\bullet) = \Delta(\tropP_\bullet^\x)$ for any $\x$.
	
	\subsection{Tropical linear space}
	Denote
	$
	\one := (1,1,\ldots,1) \in \Z^n.
	$
	Let $\tropP_\bullet \in \Dr(k,n)$ be a tropical Pl\"ucker vector.  
	
	\begin{definition}\label{def:tropical}
		The \emph{tropical linear space} $L(\tropP_\bullet)$ of $\tropP_\bullet$ is 
		$$
		L(\tropP_\bullet) :=\left\{\x \mid \text{ for all } \tau \in \binom{[n]}{k+1}, \; \min_i(\tropP_{\tau \setminus \tau_i} +x_i) \text{ is achieved at least twice} \right\}
		$$
		as a subspace of $\R^n/\one$.
	\end{definition}
	
	For $\tropP_\bullet$ integral, the integer points $L_\Z(\tropP_\bullet):= L(\tropP_\bullet) \cap \Z^n/\one \subset L(\tropP_\bullet)$ can be equipped with the structure of a flag simplicial complex, called the \emph{standard simplicial complex structure}; see for instance \cite{JSY}.  The edges are the pairs $(\x,\x')$ such that $\x-\x' = e_S \mod \one$ for some nonempty proper subset $S \subset [n]$.
	
	The stratification of $\R^n$ by the matroids $M(\tropP_\bullet^\x)$ induces a different stratification of $L(\tropP_\bullet)$ into faces.  We call this the \emph{matroid complex structure} of $L(\tropP_\bullet)$ to distinguish it from the standard simplicial complex structure on $L_\Z(\tropP_\bullet)$.
	
	Lineality action becomes translation on tropical linear spaces: we have $L(\tropP^\x_\bullet) = L(\tropP_\bullet) - \x$.  Indeed, $\y \in L(\tropP^\x_\bullet)$ if and only if $M(\tropP^{\x+\y}_\bullet)$ is loopless, which holds if and only if $\x + \y \in L(\tropP_\bullet)$.

	\begin{proposition}[{\cite[Proposition 2.3]{Spe}}]\label{prop:Spe}
		We have $\x \in L(\tropP_\bullet)$ if and only if $M(\tropP^\x_\bullet)$ is a loopless matroid.  We have that $\x$ belongs to a bounded face of the matroid complex structure of $L(\tropP_\bullet)$ if and only if $M(\tropP^\x_\bullet)$ is a coloopless matroid.
	\end{proposition}

\subsection{Tropical Pl\"ucker Vector Reconstruction}\label{subsec:support}

Let $\tropP_\bullet \in \Dr(k,n)$ be a tropical Plücker vector and let $B(\tropP_\bullet) \subset L(\tropP_\bullet)$ 
denote the subcomplex of bounded faces in the matroid complex structure on tropical linear space $L(\tropP_\bullet)$.  The vertices of 
$B(\tropP_\bullet)$ are in bijection with the top-dimensional matroid polytopes in the regular 
matroid subdivision $\Delta(\tropP_\bullet)$ of $\Delta(k,n)$.  Since top-dimensional matroid polytopes 
in $\Delta(k,n)$ have dimension $n-1$, the corresponding matroids are connected, hence both 
loopless and coloopless.

It will be convenient to lift $L(\tropP_\bullet)$ and $B(\tropP_\bullet)$ to $\R^n$.  For $\y \in L(\tropP_\bullet)$, we choose the lift $\ty \in \R^n$ to satisfy $\min_I(\tropP_I^\ty) = 0$, so that $M(\tropP_\bullet^\ty) = \{I \mid \tropP_I^\ty = 0\}$.

\begin{proposition}\label{prop:support}
	Let $\tropP_\bullet \in \Dr(k,n)$.  For every $I \in \binom{[n]}{k}$, we have
	$$
	\tropP_I = - \min_{\y \in {\rm Vert}(B(\tropP_\bullet))} \sum_{i \in I} -{\tilde y}_i,
	$$
	where $\ty = (\tilde y_1, \tilde y_2,\ldots \tilde y_n) \in \R^n$ is the lift of $\y$.
\end{proposition}
\begin{proof}
	For any $\y \in L(\tropP_\bullet)$, the matroid $M(\tropP^\ty_\bullet)$ is loopless, so we have $\sum_{i \in I} \tilde y_i \leq \tropP_I$ for any $I$.  Furthermore, $I$ is a basis of $M(\tropP^\ty_\bullet)$ if and only if $\sum_{i \in I} \tilde y_i = \tropP_I$.  It remains 
	to show that equality is achieved at some $\y \in B(\tropP_\bullet)$.  By Lemma \ref{lem:mintropvec}, 
	$I$ is a basis of the matroid $M(\tropP^\x_\bullet)$ for some $\x \in \R^n$.  Since the height function 
	$\tropP_\bullet$ is affine linear over each matroid polytope in the subdivision, there exists $\ty$ 
	such that $M(\tropP^{\ty}_\bullet)$ is a top-dimensional matroid containing $I$ as a basis, and 
	$\tropP^{\ty}_I = 0$.  The corresponding point $\y \in B(\tropP_\bullet)$ achieves $\sum_{i \in I} {\tilde y}_i = \tropP_I$.
\end{proof}

\section{Scaffolds}\label{sec:tropmodels}
Fix an integer $k \geq 2$.

\subsection{The Fermat-Le distance function}
Let $Q$ be a finite loopless directed graph with vertex set $V(Q)$.  The directed distance $\dist(v,w)$ between vertices $v$ and $w$ is the minimum of the total directed distance over all paths from $v$ to $w$.  The directed distance is equal to $\infty$ if there is no (directed) path from $v$ to $w$.  We will generally assume in our graphs that directed distances between any two vertices are finite.  A \emph{labeled} graph $(Q,\z)$ is such a graph $Q$ together with a collection $\z = (z_1,z_2,\ldots,z_n) \in V(Q)^n$ of boundary vertices.  The $z_i$ are not assumed to be distinct.

\begin{definition}\label{defn:FL}
The \defn{Fermat-Le distance sum} is the function
$$
\Le_Q(i_1,\ldots,i_k) =\Le_Q(I) := \min_{x \in V(Q)} \left( \sum_{a=1}^k \delta(x,z_{i_a})\right)
$$
where $I = \{i_1,i_2,\ldots,i_k\}$ is a $k$-element (multi)subset of $[n]$.  A vertex $x$ that achieves the minimum is called a \emph{distance minimizer} for $I$ (or for $z_{i_1},\ldots,z_{i_k}$).
\end{definition}

See \cref{fig:Q1} for an example.

\begin{remark}
The \emph{Fermat point}, also called the \emph{Torricelli point}, of a triangle is the geometric median: the point that minimizes the distance sum to the three vertices.  In our current context, the function $\Le_Q$ appears in the works of Ian Le \cite{Le} in the setting of $\PGL(k)$ affine buildings; see \cref{sec:buildings}.
\end{remark}

\begin{definition}\label{def:model}
Let $\tropP_\bullet \in \R^{\binom{[n]}{k}}$.  We say that a directed graph $(Q,\z)$ is a \emph{model} for $\tropP_\bullet$ if $\tropP_I = -\frac{1}{k}\Le_Q(I)$ for all $I \in \binom{[n]}{k}$.  
\end{definition}

Any vector $\tropP_\bullet \in \Z^{\binom{[n]}{k}}$, up to a multiple of $\one \in \R^{\binom{[n]}{k}}$, has a model: start with the vertex set $\{z_1,\ldots,z_n\} \cup \{y_J \mid J \in \binom{[n]}{k}\}$ and construct directed paths from $y_J$ to $z_j$ for $j \in J$.  The vertex $y_J$ is the only vertex that can reach all $z_j$ for $j \in J$, so it must be the distance minimizer.  The lengths of the directed paths can be engineered to obtain the desired function.

For models to be interesting, we must impose conditions on the graph $Q$.

\subsection{Barycenter embedding}\label{subsec:barycenter}
Henceforth, we make the following assumption on our directed graphs: if $v \to w$ is a directed edge in $Q$, we declare the directed distances
\begin{equation}\label{eq:direct}
\dist(v,w) = 1 \qquad \text{and} \qquad \dist(w,v) = k-1.
\end{equation}
The directed distance $\dist(v,w)$ between arbitrary vertices $v$ and $w$ is again the minimum of the total directed distance over all paths from $v$ to $w$.  We always assume that adjacent vertices $v \to w$ have directed distances given by \eqref{eq:direct}, that is, there does not exist a different path from $w$ to $v$ with distance less than $k-1$.

\begin{definition}\label{def:color}
A \emph{coloring} of a directed graph $Q$ is an assignment of a color $c(v) \in \Z/k\Z$ to each vertex $v \in V(Q)$ so that if $v \to w$ then
$$
c(w) \equiv c(v) + 1 \in \Z/k\Z.
$$
\end{definition}

\begin{lemma}\label{lem:colordist}
Suppose that $c(\cdot)$ is a coloring of $Q$.  Then 
for any two vertices $v,w$, we have 
$$\dist(v,w)  \equiv c(w) - c(v)  \in \Z/k\Z.$$
\end{lemma}
\begin{proof}
Follows from \cref{def:color} and \eqref{eq:direct} by induction on the length of the shortest path from $v$ to $w$ that computes $\dist(v,w)$.  
\end{proof}

\begin{definition}\label{def:barycenter}
	Given a labeled graph $(Q,\z)$, define the \defn{barycenter coordinates} of $v \in V(Q)$ to be
	$$
	\b_v:=  \frac{1}{k} (\dist(v,z_1),\ldots,\dist(v,z_n)) \in \frac{1}{k}\Z_{\geq 0}^n.
	$$
\end{definition}

\begin{lemma}\label{lem:adjcolor}
Suppose that $c(\cdot)$ is a coloring of $(Q,\z)$.  Then for adjacent vertices $v \to w$, we have that $\b_w - \b_v = e_S:= e_{s_1} + \cdots + e_{s_r} \mod \one$ for some subset $S \subseteq [n]$.
\end{lemma}
\begin{proof}
By \cref{lem:colordist}, we have that $\dist(v,z) - \dist(w,z) \equiv c(w) - c(v) \in \Z/k\Z$.  On the other hand, using \eqref{eq:direct}, we have $1-k \leq \dist(v,z) - \dist(w,z) \leq 1$, so $\dist(v,z) - \dist(w,z) \in \{1,1-k\}$.  Thus $\left(\b_v - \frac{1}{k}\one\right) - \b_w \in \{-1,0\}^n$ and the claim follows.
\end{proof}

\begin{definition}\label{def:Mv}
	For a vertex $v \in V(Q)$, define the \defn{matroid of $v$} to be
	$$
	M_v:= \{I \mid v \text{ is a distance minimizer for } I\} \subseteq \binom{[n]}{k}.
	$$
\end{definition}

\begin{proposition}\label{prop:Mv}
	If $(Q,\z)$ is a model for a tropical Pl\"ucker vector $\tropP_\bullet$ then for any $v \in V(Q)$ such that $M_v$ is nonempty, we have that $M_v = M(\tropP^{-\b_v}_\bullet)$ is a matroid.
\end{proposition}
\begin{proof}
	Let $(Q,\z)$ be a model for $\tropP_\bullet$, and consider the tropical Plücker vector $\tropP^{-\b_v}_\bullet$ 
	given by
	\begin{equation}\label{eq:modify}
	\tropP^{-\b_v}_I = \tropP_I + \frac{1}{k} \left(\sum_{i \in I} \dist(v, z_i)\right) 
	= \frac{1}{k}\left(\sum_{i \in I} \dist(v, z_i) 
	- \min_{u \in V(Q)} \sum_{i \in I} \dist(u, z_i)\right).
	\end{equation}
	Then $\tropP^{-\b_v}_\bullet \geq 0$, and $\tropP^{-\b_v}_I = 0$ if and only if $v$ is a distance 
	minimizer for $I$.  If $M_v$ is nonempty then $M_v = M(\tropP^{-\b_v}_\bullet)$, and the claim follows 
	from \cref{lem:mintropvec}.
\end{proof}

\begin{definition}\label{def:normalmodel}
	A labeled graph $(Q,\z)$ is a \emph{weak scaffold} if it is a model for a tropical Pl\"ucker vector $\tropP_\bullet \in \Dr(k,n)$ and it admits a coloring $c(\cdot)$.	  We call a weak scaffold \emph{positive} if it is a model for $\tropP_\bullet \in \Trop_{>0}\Gr(k,n)$.
\end{definition}

We say that $(Q,\z)$ is a weak scaffold if it is a weak scaffold for some tropical Pl\"ucker vector.  In this case, we call $\tropP_\bullet = \tropP_\bullet(Q,\z)$ the tropical Pl\"ucker vector of $(Q,\z)$.  If there is ambiguity in the value of $k$, we use the nomenclature $\PGL(k)$-weak scaffold.

Let $(Q,\z)$ be a weak scaffold for $\tropP_\bullet$.   Define 
$$\bp_\bullet := \tropP_\bullet^{-c(\z)/k}, \qquad c(\z) := (c(z_1),c(z_2),\ldots,c(z_n)) \in \{0,1,2\ldots,k-1\}^n,$$
identifying $\Z/k\Z$ with $\{0,1,2\ldots,k-1\}$.
We call $\bp_\bullet$ the \emph{modified tropical Pl\"ucker vector} of $(Q,\z)$.

	\begin{proposition}\label{prop:scaffoldint}
		Let $(Q,\z)$ be a weak scaffold for $\tropP_\bullet$.  Then we have $\bp_\bullet \in \Dr(k,n)(\Z)$.  
	\end{proposition}
	\begin{proof}
		Let $v$ be a distance minimizer for $I = \{i_1,i_2,\ldots,i_k\}$.  Then 
		$$
		\bp_I = - \frac{1}{k}\left( \dist(v,z_{i_1}) + \cdots + \dist(v,z_{i_k})- c(z_{i_1}) - \cdots - c(z_{i_k}) \right), 
		$$
		which belongs to $\Z$ by \cref{lem:colordist}.  
	\end{proof}

If $(Q,\z)$ is a weak scaffold, we define 
$$
\bb_v := \b_v - \frac{1}{k} c(\z).
$$

\begin{definition}\label{def:scaffold}
	A labeled graph $(Q,\z)$ is a \emph{(strong) scaffold} if it is a weak scaffold for $\tropP_\bullet$ and	the map 
	\begin{equation}\label{eq:embedding}
	\mu: V(Q) \longrightarrow \Z^n/\one, \qquad
	v \longmapsto -\bb_v \equiv -\bb_v + c(v) \one \mod \one
	\end{equation}
	defines an embedding of $Q$ into the one-skeleton of the standard simplicial complex structure on $L_\Z(\bp_\bullet)$.
\end{definition}
Let us spell out the conditions in \cref{def:scaffold}.  First note that $-\bb_v \in \Z^n/\one$ by \cref{lem:colordist}.  By \cref{prop:Mv}, we have that $M(\bp^{-\bb_v}_\bullet) = M_v$.  So by \cref{prop:Spe}, we have that $-\bb_v \in L_\Z(\bp_\bullet)$ if and only if $M_v$ is loopless.  The condition that the map is an embedding also requires that $\b_v \neq \b_w$ for all pairs of distinct vertices $v \neq w$.

By \cref{lem:adjcolor}, we have that $\b_w - \b_v = e_S$ for some subset $S$ whenever $v, w$ are adjacent.  If $\b_v \neq \b_w$ then $S$ must be proper and nonempty, and it follows that the edges of $Q$ are sent to edges of the standard simplicial complex structure on $L_\Z(\bp_\bullet)$.  We thus have the following.

\begin{proposition}\label{prop:embedding}
	Suppose that $(Q,\z)$ is a weak scaffold.  Then for $(Q,\z)$ to be a (strong) scaffold it suffices to check that 
	\begin{enumerate}
		\item for any vertex $v \in V(Q)$, the matroid $M_v$ is loopless, that is, 
		$\bigcup_{I \in M_v} I = [n]$;
		\item the map $v \mapsto \mu(v) \in \R^n/\one$ is injective on $V(Q)$.	
		\end{enumerate} 
\end{proposition}

Under this embedding, the matroid $M_v = M(\tropP^{-\b_v}_\bullet)$ is the matroid corresponding to the (relatively open) face of the matroid complex structure on $L(\tropP_\bullet)$ containing $-\b_v$.  

\begin{theorem}\label{thm:modelexists}
Suppose that $\tropP_\bullet \in \Trop \, X(k,n)(\Z)$.  Then $\tropP_\bullet$ has a weak scaffold.  In other words, every integral point of the tropical Grassmannian has a weak scaffold up to lineality.
\end{theorem}

We shall prove \cref{thm:modelexists} in \cref{sec:buildings}.  The weak scaffolds that we construct all embed into the affine building $\B$ of $\PGL(k)$.  We expect that our construction for \cref{thm:modelexists} actually produces a strong scaffold for every point in $\Trop \, X(k,n)(\Z)$.  However, we do not know how to give an explicit combinatorial description of the scaffolds we construct in \cref{thm:modelexists}.  The main result of this paper is to construct explicit (strong) scaffolds for all points in $\Trop_{>0} X(3,n)(\Z)$.

\begin{remark}\label{rem:canvas}
We expect scaffolds to be one-dimensional skeleta of a $(k-1)$-dimensional space tiled with simplices.  For $k = 3$, scaffolds would be the one-dimensional skeleton of a surface.  On the other hand, for any $k$, we expect that a positive scaffold $(Q,\z)$ contains natural two-dimensional surfaces $(C,\z)$ sitting inside it from which the entire structure of $(Q,\z)$ can be recovered.  We call these two-dimensional surfaces \emph{canvases}.  
\end{remark}

\subsection{Decomposition of the Fermat-Le distance sum}\label{subsec:decomposition}
Let $(Q,\z)$ be a (strong) scaffold.  The Fermat-Le distance sum can be compared with \cref{prop:support}.  On the one hand, we have by definition
$$
\tropP_I = -\frac{1}{k}\Le_Q(I) = -\min_{v \in V(Q)} \sum_{i \in I} (\b_v)_i.
$$
On the other hand, \cref{prop:support} states that
$$
\tropP_I = - \min_{\y \in {\rm Vert}(B(\tropP_\bullet))} \sum_{i \in I} -{\tilde y}_i.
$$
The embedding of \cref{def:scaffold} show that these formulae are very similar: they are both taken as minimums over certain points in $L(\tropP_\bullet)$.  Neither of the sets contain the other: the vertices $v \in V(Q)$ are not all sent to vertices of the bounded complex $B(\tropP_\bullet)$; conversely, not all vertices of the bounded complex need be in the image of $V(Q)$.

In \cref{prop:support} we can of course take the minimum over all points of $L(\tropP_\bullet)$ instead of just ${\rm Vert}(B(\tropP_\bullet))$.  We may then view the Fermat-Le distance sum as the composition of the embedding \eqref{eq:embedding} together with the optimization formula \cref{prop:support}.

\subsection{Tree scaffolds for $k = 2$}
Suppose that $k = 2$.  Then in \eqref{eq:direct} we get $\delta(v,w) = 1= \delta(w,v)$ for adjacent vertices $v,w$.  So we may view $Q$ as an undirected, connected graph.  The Fermat-Le distance sum reduces to $\Sigma_Q(i,j) = \dist(z_i,z_j)$.  The condition for a coloring $c(\cdot)$ to exist is that $Q$ is bipartite. 

\begin{proposition}
Assume that $n \geq 3$.
Suppose that $Q$ is a tree with at least one edge and let $\z = (z_1,\ldots,z_n)$ be vertices of $Q$ (repetitions are allowed), so that all leaves of $Q$ are labeled by at least two elements of $\z$.  Then $(Q,\z)$ is a scaffold and the embedding \eqref{eq:embedding} identifies $Q$ with the bounded complex $B(\bp_\bullet)$.  

If $Q$ is planar and $z_1,z_2,\ldots,z_n$ are encountered in order traversing the boundary of $Q$ counterclockwise then $\tropP_\bullet(Q,\z) \in \Trop_{>0}\Gr(2,n)$.
\end{proposition}
\begin{proof}
It is straightforward to check directly that $\tropP_\bullet(Q,\z) \in \Dr(2,n)$.  Since $Q$ is a tree, it is bipartite, and a coloring $c(\cdot) \in \Z/2\Z$ can be chosen.  For a vertex $v$, we have
$$
M_v = \{(i,j) \mid \text{ the path from $i$ to $j$ passes through $v$ }\}.
$$
It is not difficult to see directly that $M_v$ is a loopless matroid.  Given an edge $e = (v,w)$ of $Q$ let $A \subset [n]$ (resp. $B \subset [n]$) denote the labeled vertices on the side of $e$ that contains $v$ (resp. contains $w$).  By the assumption that leaves are labeled, $A$ and $B$ are nonempty and we have $A \sqcup B = [n]$. The matroid decomposition $\Delta(\tropP_\bullet)$ is obtained by slicing the hypersimplex $\Delta(2,n)$ by the hyperplanes
$$
H_e = \{\sum_{a \in A} x_a = \sum_{b \in B} x_b = 1\}
$$
as $e$ varies over the edges of $Q$.  If $v$ is incident to edges $(v,w_1), (v,w_2),\ldots, (v,w_d)$ with corresponding subsets $A_i,B_i$, then the matroid $M_v$ is given by 
$$
M_v = \{I \mid |I \cap A_i| \geq 1 \text{ for } i = 1,2,\ldots,d\}.
$$
The assumption that leaves of $Q$ are labeled by at least two elements of $\z$ implies that $M_v$ is always coloopless.  It follows from \cref{prop:Spe} that \eqref{eq:embedding} sends $V(Q)$ to points in $B(\bp_\bullet)$.
\end{proof}

\begin{example}
Consider the labeled tree $(Q,\z)$:
$$
\begin{tikzpicture}
\draw (0,0) -- (-1,0) -- (-2,0);
\draw (0,0) -- (60:1);
\draw (0,0) -- (-60:1);
\fill[black!60] (0,0) circle (1pt);
\fill[black!60] (-1,0) circle (1pt);
\fill[black!60] (-2,0) circle (1pt);
\fill[black!60] (60:1) circle (1pt);
\fill[black!60] (-60:1) circle (1pt);
\node at (-2.2,0.2) {$1$};
\node at (-2.2,-0.2) {$2$};
\node at (-0.1,-0.2) {$3$};
\node at ($(0.2,-0.2)+(-60:1)$) {$4$};
\node at ($(0.2,0.2)+(-60:1)$) {$5$};
\node at ($(0.2,-0.2)+(60:1)$) {$6$};
\node at ($(0.2,0.2)+(60:1)$) {$7$};
\end{tikzpicture}
$$
The matroids are 
\begin{align*}
M_{z_1} &= \{I \mid |I \cap [1,2]| \geq 1\}, \qquad M_v = \{(i,j) \mid i \in [1,2], j \in [3,7]\}, \\
M_{z_3} &= \binom{[7]}{2} \setminus \{12,45,67\}, \qquad M_{z_4} = \{I \mid |I \cap [4,5]| \geq 1\}, \qquad M_{z_6} = \{I \mid |I \cap [6,7]| \geq 1\},
\end{align*}
where $v$ is the unique unlabeled vertex.  All of these matroids are connected except for $M_v$, which corresponds to a codimension one matroid polytope in $\Delta(\tropP_\bullet)$.  
We have
$$
\b_{z_1} = \frac{1}{2}(0,0,2,3,3,3,3), \qquad \b_{v} = \frac{1}{2}(1,1,1,2,2,2,2), \qquad \b_{z_3} = \frac{1}{2}(2,2,0,1,1,1,1),
$$
so that $\b_v$ is on the open line segment $(\b_{z_1},\b_{z_3})$.  In $B(\tropP_\bullet)$ (or $B(\bp_\bullet)$) the face (in the matroid complex structure) corresponding to the matroid $M_v$ is exactly the image of this line segment under \eqref{eq:embedding}.  The subdivision $\Delta(\tropP_\bullet)$ has exactly four top-dimensional pieces, corresponding to the four matroids $M_{z_1},M_{z_3},M_{z_4},M_{z_6}$.
\end{example}

	\section{CAT(0) planar graphs}\label{sec:definitions}
	Henceforth we work with $k = 3$, and $n \geq k$.
	
	\subsection{Main definition}
	We introduce the main object of this work.  
		
	\begin{definition}\label{def:CAT0-planar}
		A \defn{CAT(0) planar graph} $Q$ is a nonempty connected directed planar graph with all interior faces being oriented triangles, such that
		\begin{enumerate}
			\item All interior vertices have degree at least $6$, and
			\item All edges and vertices are on the perimeter of some face.
		\end{enumerate}
	\end{definition}
	
	CAT(0) planar graphs are obtained from the one-skeleta of diskoids \cite{FKK} by equipping the edges with directions; see \cref{part:buildings}.

	The boundary $\partial Q$, when traversed counterclockwise is a closed curve in the plane.  A vertex $v \in V(Q)$ is called a \emph{cut vertex} if it is a self-intersection point of this curve.  An edge $e \in E(Q)$ is called a \emph{cut edge} if it is traversed twice by the curve, necessarily once in each direction.  Often we will restrict our attention to CAT(0) planar graphs without cut edges.  However, when cut edges do exist, we use the following convention.
	
	\begin{convention}
		If $v \to u$ is a cut edge of $Q$, then $\delta(v,u) = 1$ and $\delta(u,v) = 2$.  In other words, when computing directed distance we can use the cut edge backwards where it has distance 2.
	\end{convention}
	
	Thus, in all cases, the directed distance $\dist(v,u)$ is the minimum directed distance of a path from $v$ to $u$ where traversing an edge according to the orientation (resp. against the orientation) contributes $1$ (resp. $2$), agreeing with \eqref{eq:direct}.
	
	Since a CAT(0) planar graph $Q$ is assumed to be connected, it has a finite-valued directed distance function $\dist(\cdot,\cdot)$.  For vertices $v, w \in V(Q)$, the \emph{combinatorial distance} $\ell(v,w)$ is the minimal edge count in any path, ignoring directions of edges.

	\begin{definition}
		A \defn{labeled CAT(0) planar graph} is a pair $(Q,\z)$ consisting of a CAT(0) planar graph $Q$ and a collection $\z = 	(z_1,z_2,\ldots,z_n) \in V(Q)^n$ of vertices.  If the vertices $\z$ all lie on the boundary $\partial Q$, and are encountered in increasing order as we traverse the boundary counterclockwise, then we call $(Q,\z)$ a CCW-labeled CAT(0) planar graph.
	\end{definition}
	
	The Fermat-Le distance function is well-defined for a labeled CAT(0) planar graph.
	
	\begin{definition}\label{def:simple}
		A CAT(0) planar graph $Q$ is called \emph{simple} if the boundary of $Q$ is a simple closed counterclockwise oriented curve.
	\end{definition}	
	
	We allow a single vertex to be a simple CAT(0) planar graph.  However, a single edge is not.  The graph $Q$ in \cref{fig:Q1} is simple.  The CAT(0) planar graph in \cref{fig:notsimple} is not simple:
	
	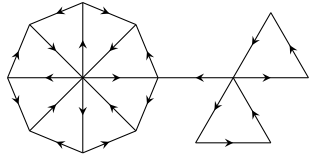
\begin{figure}[!h]
			$$
		\begin{tikzpicture}
			\tikzset{->-/.style={decoration={markings, mark=at position #1 with {\arrow{stealth}}},postaction={decorate}}, ->-/.default=0.5}
			\draw[->-] (0:1) -- (45:1);
			\draw[->-] (90:1) -- (45:1);
			\draw[->-] (90:1) -- (135:1);
			\draw[->-] (180:1) -- (135:1);
			\draw[->-] (180:1) -- (225:1);
			\draw[->-] (270:1) -- (225:1);
			\draw[->-] (270:1) -- (315:1);
			\draw[->-] (0:1) -- (315:1);
			\draw[->-] (0,0) -- (0:1);
			\draw[->-] (45:1)--(0,0);
			\draw[->-] (0,0) -- (90:1);
			\draw[->-] (135:1)--(0,0);
			\draw[->-] (0,0) -- (180:1);
			\draw[->-] (225:1)--(0,0);
			\draw[->-] (0,0) -- (270:1);
			\draw[->-] (315:1)--(0,0);
			\draw[->-] (2,0) -- (1,0);
			\draw[->-] (2,0) -- (3,0);
			\draw[->-] (3,0) -- ($(2,0)+(60:1)$);
			\draw[->-] ($(2,0)+(60:1)$)--(2,0);
			\draw[->-] (2,0) -- ($(2,0)+(240:1)$);
			\draw[->-] ($(2,0)+(240:1)$)-- ($(2,0)+(300:1)$);
			\draw[->-] ($(2,0)+(300:1)$)-- (2,0);
		\end{tikzpicture}
		$$
		\caption{A non-simple CAT(0) planar graph.}
		\label{fig:notsimple}
	\end{figure}
	
	\begin{definition}\label{def:normalvertex}
		Let $(Q, \z)$ be a labeled CAT(0) planar graph.  We call a vertex $v \in V(Q)$ \defn{normal} if $v$ is a distance minimizer for some triple $\{z_i,z_j,z_k\}$ of labeled vertices.  
	\end{definition}

	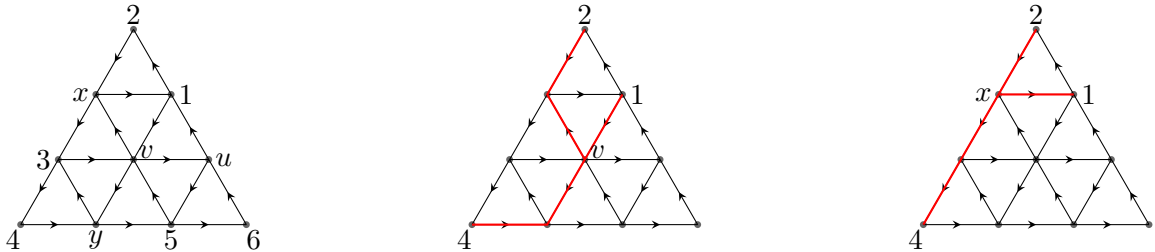
\begin{figure}
	\begin{center}
	$$
	\begin{tikzpicture}
	\tikzset{->-/.style={decoration={markings, mark=at position #1 with {\arrow{stealth}}},postaction={decorate}}, ->-/.default=0.5}
	\coordinate (z4) at (0,0);
	\coordinate (y) at (1,0);
	\coordinate (z5) at (2,0);
	\coordinate (z6) at (3,0);
	\coordinate (z3) at (60:1);
	\coordinate (v) at ($(z3) + (1,0)$);
	\coordinate (u) at ($(v) + (1,0)$);
	\coordinate (x) at (60:2);
	\coordinate (z1) at ($(x) + (1,0)$);
	\coordinate (z2) at (60:3);
	\foreach \v in {z1,z2,z3,z4,z5,z6,v,u,x,y} {
				\fill[black!60] (\v) circle (1.3pt);
			}
	\draw[->-] (z1) -- (z2);
	\draw[->-] (z2) -- (x);
	\draw[->-] (x) -- (z3);
	\draw[->-] (z3) -- (z4);
	\draw[->-] (z4) -- (y);
	\draw[->-] (y) -- (z5);
	\draw[->-] (z5) -- (z6);
	\draw[->-] (z6) -- (u);
	\draw[->-] (u) -- (z1);
	\draw[->-] (x) -- (z1);
	\draw[->-] (z1) -- (v);
	\draw[->-] (v) -- (x);
	\draw[->-] (v) -- (u);
	\draw[->-] (v) -- (y);
	\draw[->-] (z3) -- (v);
	\draw[->-] (y) -- (z3);
	\draw[->-] (z5) -- (v);
	\draw[->-] (u) -- (z5);
	\node at ($(z1)+(0.2,0)$) {\small $1$};
	\node at ($(z2)+(0,0.2)$) {\small $2$};
	\node at ($(z3)+(-0.2,0)$) {\small $3$};
	\node at ($(z4)+(-0.1,-0.2)$) {\small $4$};
	\node at ($(z5)+(0,-0.2)$) {\small $5$};
	\node at ($(z6)+(0.1,-0.2)$) {\small  $6$};
	\node at ($(x)+(-0.2,-0)$) {\small  $x$};
	\node at ($(v)+(0.17,0.1)$) {\small  $v$};
	\node at ($(u)+(0.2,-0)$) {\small  $u$};
	\node at ($(y)+(-0,-0.2)$) {\small  $y$};
	
	\begin{scope}[shift={(6,0)}]
	\coordinate (z4) at (0,0);
	\coordinate (y) at (1,0);
	\coordinate (z5) at (2,0);
	\coordinate (z6) at (3,0);
	\coordinate (z3) at (60:1);
	\coordinate (v) at ($(z3) + (1,0)$);
	\coordinate (u) at ($(v) + (1,0)$);
	\coordinate (x) at (60:2);
	\coordinate (z1) at ($(x) + (1,0)$);
	\coordinate (z2) at (60:3);
\foreach \v in {z1,z2,z3,z4,z5,z6,v,u,x,y} {
				\fill[black!60] (\v) circle (1.3pt);
			}
	\draw[->-] (z1) -- (z2);
	\draw[->-] (z2) -- (x);
	\draw[->-] (x) -- (z3);
	\draw[->-] (z3) -- (z4);
	\draw[->-] (z4) -- (y);
	\draw[->-] (y) -- (z5);
	\draw[->-] (z5) -- (z6);
	\draw[->-] (z6) -- (u);
	\draw[->-] (u) -- (z1);
	\draw[->-] (x) -- (z1);
	\draw[->-] (z1) -- (v);
	\draw[->-] (v) -- (x);
	\draw[->-] (v) -- (u);
	\draw[->-] (v) -- (y);
	\draw[->-] (z3) -- (v);
	\draw[->-] (y) -- (z3);
	\draw[->-] (z5) -- (v);
	\draw[->-] (u) -- (z5);
	\draw[thick,red] (v)--(z1);
	\draw[thick,red] (v)--(x)--(z2);
	\draw[thick,red] (v)--(y)--(z4);
	\node at ($(z1)+(0.2,0)$) {\small $1$};
	\node at ($(z2)+(0,0.2)$) {\small $2$};
	\node at ($(z4)+(-0.1,-0.2)$) {\small $4$};
	\node at ($(v)+(0.17,0.1)$) {\small  $v$};
	\end{scope}
	\begin{scope}[shift={(12,0)}]
	\coordinate (z4) at (0,0);
	\coordinate (y) at (1,0);
	\coordinate (z5) at (2,0);
	\coordinate (z6) at (3,0);
	\coordinate (z3) at (60:1);
	\coordinate (v) at ($(z3) + (1,0)$);
	\coordinate (u) at ($(v) + (1,0)$);
	\coordinate (x) at (60:2);
	\coordinate (z1) at ($(x) + (1,0)$);
	\coordinate (z2) at (60:3);
\foreach \v in {z1,z2,z3,z4,z5,z6,v,u,x,y} {
				\fill[black!60] (\v) circle (1.3pt);
			}
	\draw[->-] (z1) -- (z2);
	\draw[->-] (z2) -- (x);
	\draw[->-] (x) -- (z3);
	\draw[->-] (z3) -- (z4);
	\draw[->-] (z4) -- (y);
	\draw[->-] (y) -- (z5);
	\draw[->-] (z5) -- (z6);
	\draw[->-] (z6) -- (u);
	\draw[->-] (u) -- (z1);
	\draw[->-] (x) -- (z1);
	\draw[->-] (z1) -- (v);
	\draw[->-] (v) -- (x);
	\draw[->-] (v) -- (u);
	\draw[->-] (v) -- (y);
	\draw[->-] (z3) -- (v);
	\draw[->-] (y) -- (z3);
	\draw[->-] (z5) -- (v);
	\draw[->-] (u) -- (z5);
	\draw[thick,red] (x)--(z1);
	\draw[thick,red] (x)--(z2);
	\draw[thick,red] (x)--(z3)--(z4);
	\node at ($(z1)+(0.2,0)$) {\small $1$};
	\node at ($(z2)+(0,0.2)$) {\small $2$};
	\node at ($(z4)+(-0.1,-0.2)$) {\small $4$};	
	\node at ($(x)+(-0.2,-0)$) {\small  $x$};
	\end{scope}

	\end{tikzpicture}
	$$
	\caption{Left: A CCW-labeled CAT(0) planar graph $Q$.  The boundary vertices $\z$ are simply labeled $1,\ldots,6$.  Middle: We have $\dist(v,1) = 2$, $\dist(v,2)= 3$, and $\dist(v,4)=3$.  Right: We have $\dist(x,1) = 1$, $\dist(x,2) = 2$, $\dist(x,4) = 2$.  The Fermat-Le distance function is $\Sigma_Q(1,2,4) =5$.  The vertex $x$ is a distance-minimizer but $v$ is not.}
	\label{fig:Q1}
	\end{center}
	\end{figure}

	Let $Q$ be a CAT(0) planar graph with boundary oriented counterclockwise.  Note that this implies that every vertex of $Q$ has even degree and that every vertex is on the boundary of a triangular face of $Q$.  A vertex $v \in \partial Q$ is called \defn{acute} if it has degree $2$, and is a vertex of exactly one triangular face of $Q$.
	
	\begin{definition}\label{def:normal}
		Let $(Q,\z)$ be a CCW-labeled CAT(0) planar graph.  Then $(Q,\z)$ is \defn{normal} if the boundary is oriented counterclockwise and every acute vertex of $\partial Q$ is labeled by (at least) one of the $\z$.  
	\end{definition}
	
	We allow the graph $Q$ with a single vertex labeled with all the $z_i$ to be considered a normal CAT(0) planar graph.  The graph $Q$ in \cref{fig:Q1} is a normal CAT(0) planar graph.  Note that normal CAT(0) planar graphs do not have cut edges.
	
	\begin{definition}
	Let $(Q,\z)$ be a normal CAT(0) planar graph.  We call $(Q,\z)$ \defn{cyclic} if it contains two acute labeled vertices $z_i$ and $z_{i+1}$ such that every intermediate vertex on the boundary path from $z_i$ to $z_{i+1}$ has degree four.  If $(Q,\z)$ is not cyclic, we call it \defn{cyclic-less}.
	\end{definition}
	
	\subsection{Main result}\label{sec:main}
	Let $(Q,\z)$ be a CAT(0) planar graph.  Since $Q$ is tiled by oriented triangles it is easy to see that the vertices of $Q$ can be colored by $\{0,1,2\}$ so that if $v \to w$ then $c(w) \equiv c(v) + 1 \mod 3$, satisfying \cref{def:color}.  Since $Q$ is connected, this coloring is unique up to a global shift by an element of $\Z/3\Z$.  In any triangle $T$ of $Q$, the three vertices use each of the colors once.  We assume such a \emph{coloring} has been fixed.  
	
	We can now state our main theorem.
	
	\begin{theorem}\label{thm:main}
		Let $(Q, \z)$ be a labeled CAT(0) planar graph.
		\begin{enumerate}
			\item
			Suppose that $\z$ are arbitrary.  Then $(Q,\z)$ is a weak scaffold.  Thus $\tropP_\bullet(Q,\z) \in \Dr(3,n)$ satisfies the tropical Pl\"ucker relations \eqref{eq:tropPlucker}.
			\item
			Suppose that $(Q,\z)$ is CCW-labeled.  Then $\tropP_\bullet(Q,\z) \in \Trop_{>0} \Gr(3,n)$ satisfies the positive tropical Pl\"ucker relations \eqref{eq:postropPlucker}.
			\item
			Suppose that $(Q,\z)$ is normal.  Then $(Q,\z)$ is a scaffold for $\tropP_\bullet(Q,\z)$.  Every point in $\Trop_{>0} X(3,n)(\Z)$ has a unique scaffold that is a cyclic-less normal CAT(0) planar graph $(Q,\z)$.
		\end{enumerate}
	\end{theorem}
	
	In particular, we shall see that (3) implies that for a normal $(Q,\z)$, every vertex is normal.
	
	We shall prove \cref{thm:main} in \cref{sec:mainproof} and \cref{sec:proof-main-3}.  


	\begin{example}\label{example: non CCW 37}
		Let $Q$ be the CAT(0) graph 
		shown on the left in \cref{fig: 37 nonpos}.  This is a hexagon with alternating boundary and a seventh labeled point in the interior.  We have
		$$\tropP_\bullet(Q,\z) \equiv e^{234} + e^{456} + e^{126} \in \Trop\, X(3,7)$$
		modulo lineality, an instance of \cref{thm:main}(1).  Recall that $e^I$ denotes the standard basis vector of $\R^{\binom{[n]}{k}}$. 
	\end{example}
		
	\begin{example} In \cref{example: non CCW 37}, there is no permutation of $z_1,\ldots,z_7$ that gives a point in $\Trop_{>0}X(3,7)$.  However, if we do not label $z_7$, as in the middle of \cref{fig: 37 nonpos}, we get the point
	\begin{equation}
	\label{eq:36}
	\tropP_\bullet(Q,\z) \equiv e^{234} + e^{456} + e^{126} \in \Trop_{>0}X(3,6),
	\end{equation}
	an instance of \cref{thm:main}(2).  The hexagon $(Q,(z_1,\ldots,z_6))$ is however not normal because the boundary is not oriented counterclockwise.
	\end{example}
	
	\begin{figure}
	\begin{center}
			\begin{tikzpicture}[scale=0.5]
			\tikzset{->-/.style={decoration={markings, mark=at position 0.55 with {\arrow{stealth}}},postaction={decorate}}, ->-/.default=0.55}
			
			\coordinate (1) at (-60:2);
			\coordinate (2) at (0:2);
			\coordinate (3) at (60:2);
			\coordinate (4) at (120:2);
			\coordinate (5) at (180:2);
			\coordinate (6) at (240:2);
			\coordinate (7) at (0,0);
			
			\draw[->-] (1) -- (2);
			\draw[->-] (3) -- (2);
			\draw[->-] (3) -- (4);
			\draw[->-] (5) -- (4);
			\draw[->-] (5) -- (6);
			\draw[->-] (1) -- (6);
			
			\draw[->-] (7) -- (1);
			\draw[->-] (2) -- (7);
			\draw[->-] (7) -- (3);
			\draw[->-] (4) -- (7);
			\draw[->-] (7) -- (5);
			\draw[->-] (6) -- (7);
			
			\foreach \v in {1,2,3,4,5,6,7} {
				\fill[black!60] (\v) circle (2.5pt);
			}
			
			\node at ($(1)+(0.3,-0.2)$) {\small $1$};
			\node at ($(2)+(0.25,0)$) {\small $2$};
			\node at ($(3)+(0.15,0.35)$) {\small $3$};
			\node at ($(4)+(-0.15,0.35)$) {\small $4$};
			\node at ($(5)+(-0.25,0)$) {\small $5$};
			\node at ($(6)+(-0.25,-0.2)$) {\small $6$};
			\node at ($(7)+(0.2,0.15)$) {\small $7$};
		\end{tikzpicture}
		\hspace{30pt}
		\begin{tikzpicture}[scale=0.5]
			\tikzset{->-/.style={decoration={markings, mark=at position 0.55 with {\arrow{stealth}}},postaction={decorate}}, ->-/.default=0.55}
			
			\coordinate (1) at (-60:2);
			\coordinate (2) at (0:2);
			\coordinate (3) at (60:2);
			\coordinate (4) at (120:2);
			\coordinate (5) at (180:2);
			\coordinate (6) at (240:2);
			\coordinate (7) at (0,0);
			
			\draw[->-] (1) -- (2);
			\draw[->-] (3) -- (2);
			\draw[->-] (3) -- (4);
			\draw[->-] (5) -- (4);
			\draw[->-] (5) -- (6);
			\draw[->-] (1) -- (6);
			
			\draw[->-] (7) -- (1);
			\draw[->-] (2) -- (7);
			\draw[->-] (7) -- (3);
			\draw[->-] (4) -- (7);
			\draw[->-] (7) -- (5);
			\draw[->-] (6) -- (7);
			
			\foreach \v in {1,2,3,4,5,6,7} {
				\fill[black!60] (\v) circle (2.5pt);
			}
			
			\node at ($(1)+(0.3,-0.2)$) {\small $1$};
			\node at ($(2)+(0.25,0)$) {\small $2$};
			\node at ($(3)+(0.15,0.35)$) {\small $3$};
			\node at ($(4)+(-0.15,0.35)$) {\small $4$};
			\node at ($(5)+(-0.25,0)$) {\small $5$};
			\node at ($(6)+(-0.25,-0.2)$) {\small $6$};
		\end{tikzpicture}
		\hspace{30pt}
				\begin{tikzpicture}[scale=0.5]
			\tikzset{->-/.style={decoration={markings, mark=at position 0.55 with {\arrow{stealth}}},postaction={decorate}}}
			
			\pgfmathsetmacro{\h}{1.732}
			
			\coordinate (1) at (3, 3*\h);
			
			\coordinate (y) at (4, 2*\h);
			\coordinate (2) at (2, 2*\h);
			
			\coordinate (6) at (5, \h);
			\coordinate (w) at (3, \h);
			\coordinate (z) at (1, \h);
			
			\coordinate (5) at (6, 0);
			\coordinate (x) at (4, 0);
			\coordinate (4) at (2, 0);
			\coordinate (3) at (0, 0);
			
			\draw[->-] (5) -- (6);
			\draw[->-] (x) -- (5);
			\draw[->-] (4) -- (x);
			\draw[->-] (3) -- (4);
			\draw[->-] (6) -- (y);
			\draw[->-] (y) -- (1);
			\draw[->-] (1) -- (2);
			\draw[->-] (2) -- (z);
			\draw[->-] (z) -- (3);
			
			\draw[->-] (6) -- (x);
			\draw[->-] (x) -- (w);
			\draw[->-] (w) -- (4);
			\draw[->-] (4) -- (z);
			\draw[->-] (z) -- (w);
			\draw[->-] (w) -- (6);
			\draw[->-] (w) -- (2);
			\draw[->-] (y) -- (w);
			\draw[->-] (2) -- (y);
			
			\foreach \v in {1,2,3,4,5,6,x,y,w,z} {
				\fill[black!60] (\v) circle (2.5pt);
			}
			
			\node at ($(1)+(0,0.35)$) {\small $1$};
			\node at ($(2)+(-0.3,0.1)$) {\small $2$};
			\node at ($(3)+(-0.25,-0.35)$) {\small $3$};
			\node at ($(4)+(0,-0.35)$) {\small $4$};
			\node at ($(5)+(0.1,-0.35)$) {\small $5$};
			\node at ($(6)+(0.3,0)$) {\small $6$};
			\node at ($(x)+(0,-0.35)$) {\small $x$};
			\node at ($(y)+(0.3,0.1)$) {\small $y$};
			\node at ($(w)+(-0.15,-0.25)$) {\small $w$};
			\node at ($(z)+(-0.3,0)$) {\small $z$};
			
		\end{tikzpicture}

		\caption{Left: a labeled CAT(0) planar graph $(Q,\z)$ that is neither CCW-labeled nor normal.  \\ Middle: the same graph with the label $z_7$ removed is a model for  the positive tropical Pl\"ucker vector $\tropP_\bullet \equiv e^{234} + e^{456} + e^{126}$. \\Right: the normal cyclic-less CAT(0) planar graph for the positive tropical Pl\"ucker vector $\tropP_\bullet \equiv e^{234} + e^{456} + e^{126}$. }
		\label{fig: 37 nonpos}
		\end{center}
	\end{figure}
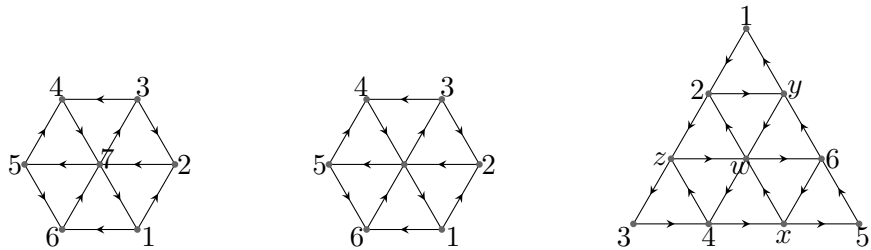
	
	\begin{example}\label{example: normal model 36}
		Let $(Q,\z)$ be the CAT(0) graph depicted on the right in \cref{fig: 37 nonpos}.  Then $(Q,\z)$ is normal since all acute vertices are labeled, and it is a scaffold for the positive tropical Pl\"ucker vector \eqref{eq:36}.  This $(Q,\z)$ is the unique normal cyclic-less CAT(0) planar graph that is a scaffold for \eqref{eq:36}, an instance of \cref{thm:main}(3).
	\end{example}	
	
	\begin{example}
	The fan $\Trop_{>0}X(3,6)$ has four rays up to cyclic rotation.  In \cref{fig:examples36} we present representative scaffolds for these rays. 
		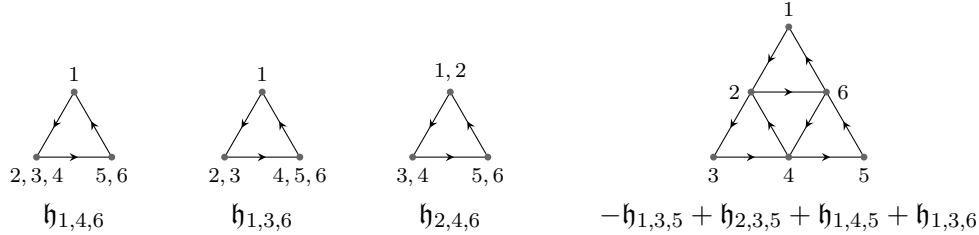
\begin{figure}[h!]
			\centering
			\begin{tikzpicture}
			\tikzset{->-/.style={decoration={markings, mark=at position 0.55 with {\arrow{stealth}}},postaction={decorate}}}
			\draw[->-] (0,0) -- (1,0);
			\draw[->-] (1,0) -- (60:1);
			\draw[->-] (60:1) -- (0,0);
			\node[above] at (60:1) {\tiny $1$};
			\node[below] at (0,0) {\tiny $2,3,4$};
			\node[below] at (1,0) {\tiny $5,6$};
			\node at (0.5,-0.8) {\small $\h_{1,4,6}$};
			\foreach \v in {(0,0),(1,0),(60:1)} {
				\fill[black!60] \v circle (1.3pt);
			}
			\begin{scope}[shift={(2.5,0)}]
			\draw[->-] (0,0) -- (1,0);
			\draw[->-] (1,0) -- (60:1);
			\draw[->-] (60:1) -- (0,0);
			\node[above] at (60:1) {\tiny $1$};
			\node[below] at (0,0) {\tiny $2,3$};
			\node[below] at (1,0) {\tiny$4,5,6$};
			\node at (0.5,-0.8) {\small $\h_{1,3,6}$};
			\foreach \v in {(0,0),(1,0),(60:1)} {
				\fill[black!60] \v circle (1.3pt);
			}
			\end{scope}
			\begin{scope}[shift={(5,0)}]
			\draw[->-] (0,0) -- (1,0);
			\draw[->-] (1,0) -- (60:1);
			\draw[->-] (60:1) -- (0,0);
			\node[above] at (60:1) {\tiny $1,2$};
			\node[below] at (0,0) {\tiny $3,4$};
			\node[below] at (1,0) {\tiny$5,6$};
			\node at (0.5,-0.8) {\small $\h_{2,4,6}$};
			\foreach \v in {(0,0),(1,0),(60:1)} {
				\fill[black!60] \v circle (1.3pt);
			}
			\end{scope}
			\begin{scope}[shift={(9,0)}]
			\draw[->-] (0,0) -- (1,0);
			\draw[->-] (1,0) -- (2,0);
			\draw[->-] (2,0) -- ($(2,0)+(120:1)$);
			\draw[->-] ($(2,0)+(120:1)$) -- (60:2);
			\draw[->-] (60:2) -- (60:1);
			\draw[->-] (60:1) -- ($(2,0)+(120:1)$);
			\draw[->-] (1,0) -- (60:1);
			\draw[->-] (60:1)--(0,0);
			\draw[->-] ($(2,0)+(120:1)$) -- (1,0);
			\node[above] at (60:2) {\tiny $1$};
			\node[left] at (60:1) {\tiny $2$};
			\node[below] at (0,0) {\tiny$3$};
			\node[below] at (1,0) {\tiny$4$};
			\node[below] at (2,0) {\tiny$5$};
			\node[right] at ($(2,0)+(120:1)$) {\tiny$6$};
			\node at (1,-0.8) {\small $-\h_{1,3,5} + \h_{2,3,5} + \h_{1,4,5} + \h_{1,3,6}$};
			\foreach \v in {(0,0),(1,0),(2,0),(60:1),(60:2),($(2,0)+(120:1)$)} {
				\fill[black!60] \v circle (1.3pt);
			}
			\end{scope}
			\end{tikzpicture}
			\caption{Scaffolds for the four rays of $\text{Trop}_{>0}X(3,6)$ up to cyclic rotation.  Below each scaffold we have recorded the planar basis expansion (\cref{thm:pbexpansion}).}
			\label{fig:examples36}
		\end{figure}
	\end{example}
	More rays are tabulated in \cref{sec:appendix}.
	\section{Webs and strands}\label{sec:webs}
	\subsection{Webs and their dual graphs}
	\begin{definition}\label{def:SL3-web}
		An \defn{$\mathrm{SL}_3$-web} $W$ is a planar graph embedded in 
		a disk with $n$ boundary vertices $b_1,b_2,\ldots,b_n$ in counterclockwise order on the boundary of the disk, such that:
		\begin{itemize}
			\item Every vertex is colored either black or white.
			\item Each internal vertex is trivalent.
			\item Each edge connects vertices of different colors (black and white).
		\end{itemize}
		The web $W$ is \defn{non-elliptic} if each internal face (that is, a face that has no boundary vertex) has at least 
		$6$ edges.
	\end{definition}
	
	In the rest of this paper, we refer to $\mathrm{SL}_3$-webs simply as webs.  
	%
	
	A non-elliptic web $W$ is called \defn{standard} if every boundary vertex is black and has 
	valence $1$.  
	
	In \cref{fig:web-examples}, we depict three webs of increasing complexity; the rightmost web will return several times later in the paper, including in the context of buildings, see  \cref{fig:312buildingex}.
\begin{figure}[ht]
	\centering
	\begin{minipage}[b]{0.28\textwidth}
		\centering
		\begin{tikzpicture}[scale=0.8,
			bvert/.style={circle, fill=black, inner sep=0, minimum size=5pt},
			wvert/.style={circle, draw=black, fill=white, inner sep=0, minimum size=5pt}
			]
			\node[wvert, minimum size=7pt] (w) at (0,0) {};
			\node[bvert] (b1) at (90:1.8) {};
			\node[bvert] (b2) at (210:1.8) {};
			\node[bvert] (b3) at (330:1.8) {};
			\node[bvert] at (40:1.8) {};
			\node[bvert] at (270:1.8) {};
			\draw (w) -- (b1);
			\draw (w) -- (b2);
			\draw (w) -- (b3);
			\node[above] at (b1) {$3$};
			\node[below left] at (b2) {$4$};
			\node[below right] at (b3) {$1$};
			\node[right] at (40:1.8) {$2$};
			\node[right] at (270:1.8) {$5$};
			\draw[thin, dashed, gray!40] (0,0) circle (1.8);
		\end{tikzpicture}
		
		\smallskip
		(a) $n=5$, non-standard
	\end{minipage}
	\hfill
	\begin{minipage}[b]{0.32\textwidth}
		\centering
		\begin{tikzpicture}[scale=0.8,
			bvert/.style={circle, fill=black, inner sep=0, minimum size=5pt},
			wvert/.style={circle, draw=black, fill=white, inner sep=0, minimum size=5pt}
			]
			\node[bvert] (bc) at (0,0) {};
			\node[wvert] (w1) at (90:1) {};
			\node[wvert] (w2) at (210:1) {};
			\node[wvert] (w3) at (330:1) {};
			\draw (bc) -- (w1);
			\draw (bc) -- (w2);
			\draw (bc) -- (w3);
			\node[bvert] (b1) at (60:2.2) {};
			\node[bvert] (b2) at (120:2.2) {};
			\draw (w1) -- (b1);
			\draw (w1) -- (b2);
			\node[bvert] (b3) at (180:2.2) {};
			\node[bvert] (b4) at (240:2.2) {};
			\draw (w2) -- (b3);
			\draw (w2) -- (b4);
			\node[bvert] (b5) at (300:2.2) {};
			\node[bvert] (b6) at (360:2.2) {};
			\draw (w3) -- (b5);
			\draw (w3) -- (b6);
			\node[above right] at (b1) {$1$};
			\node[above left] at (b2) {$2$};
			\node[left] at (b3) {$3$};
			\node[below left] at (b4) {$4$};
			\node[below right] at (b5) {$5$};
			\node[right] at (b6) {$6$};
			\draw[thin, dashed, gray!40] (0,0) circle (2.2);
		\end{tikzpicture}
		
		\smallskip
		(b) $n=6$, standard
	\end{minipage}
	\hfill
	\begin{minipage}[b]{0.36\textwidth}
		\centering
		\begin{tikzpicture}[scale=0.55,
			bvert/.style={circle, fill=black, inner sep=0, minimum size=4pt},
			wvert/.style={circle, draw=black, fill=white, inner sep=0, minimum size=4pt}
			]
			\def\R{2.2}
			\node[wvert, minimum size=5pt] (wT)  at (90:\R)  {};
			\node[bvert]                   (bTL) at (135:\R) {};
			\node[wvert, minimum size=5pt] (wL)  at (180:\R) {};
			\node[bvert]                   (bBL) at (225:\R) {};
			\node[wvert, minimum size=5pt] (wB)  at (270:\R) {};
			\node[bvert]                   (bBR) at (315:\R) {};
			\node[wvert, minimum size=5pt] (wR)  at (0:\R)   {};
			\node[bvert]                   (bTR) at (45:\R)  {};
			%
			\draw (wT) -- (bTR) -- (wR) -- (bBR) -- (wB) -- (bBL) -- (wL) -- (bTL) -- (wT) -- cycle;
			%
			\node[bvert] (d1)  at (90:4.2)  {};  \draw (wT) -- (d1);
			\node[bvert] (d4)  at (180:4.2) {};  \draw (wL) -- (d4);
			\node[bvert] (d7)  at (270:4.2) {};  \draw (wB) -- (d7);
			\node[bvert] (d10) at (0:4.2)   {};  \draw (wR) -- (d10);
			%
			%
			\node[wvert, minimum size=5pt] (eTL) at (135:3.2) {};
			\draw (bTL) -- (eTL);
			\node[bvert] (d2) at (112.5:4.2) {};
			\node[bvert] (d3) at (157.5:4.2) {};
			\draw (eTL) -- (d2);
			\draw (eTL) -- (d3);
			%
			\node[wvert, minimum size=5pt] (eBL) at (225:3.2) {};
			\draw (bBL) -- (eBL);
			\node[bvert] (d5) at (202.5:4.2) {};
			\node[bvert] (d6) at (247.5:4.2) {};
			\draw (eBL) -- (d5);
			\draw (eBL) -- (d6);
			%
			\node[wvert, minimum size=5pt] (eBR) at (315:3.2) {};
			\draw (bBR) -- (eBR);
			\node[bvert] (d8) at (292.5:4.2) {};
			\node[bvert] (d9) at (337.5:4.2) {};
			\draw (eBR) -- (d8);
			\draw (eBR) -- (d9);
			%
			\node[wvert, minimum size=5pt] (eTR) at (45:3.2) {};
			\draw (bTR) -- (eTR);
			\node[bvert] (d11) at (22.5:4.2) {};
			\node[bvert] (d12) at (67.5:4.2) {};
			\draw (eTR) -- (d11);
			\draw (eTR) -- (d12);
			%
			\node[above]       at (d1)  {$12$};
			\node[above left]  at (d2)  {$1$};
			\node[left]        at (d3)  {$2$};
			\node[left]        at (d4)  {$3$};
			\node[below left]  at (d5)  {$4$};
			\node[below left]  at (d6)  {$5$};
			\node[below]       at (d7)  {$6$};
			\node[below right] at (d8)  {$7$};
			\node[right]       at (d9)  {$8$};
			\node[right]        at (d10) {$9$};
			\node[above right] at (d11) {$10$};
			\node[above right] at (d12) {$11$};
			%
			\draw[thin, dashed, gray!40] (0,0) circle (4.2);
		\end{tikzpicture}
		
		\smallskip
		(c) $n=12$, standard
	\end{minipage}
	\caption{Examples of non-elliptic webs.  Each edge is incident to one white vertex ($\circ$) and one black vertex ($\bullet$).  All interior vertices are trivalent.  In~(a), the web is nonstandard since some boundary vertices have degree $0$.  In~(b) and~(c), the webs are standard: all boundary vertices are black with valence~$1$.  The interior face in~(c) is an octagon.}
	\label{fig:web-examples}
\end{figure}
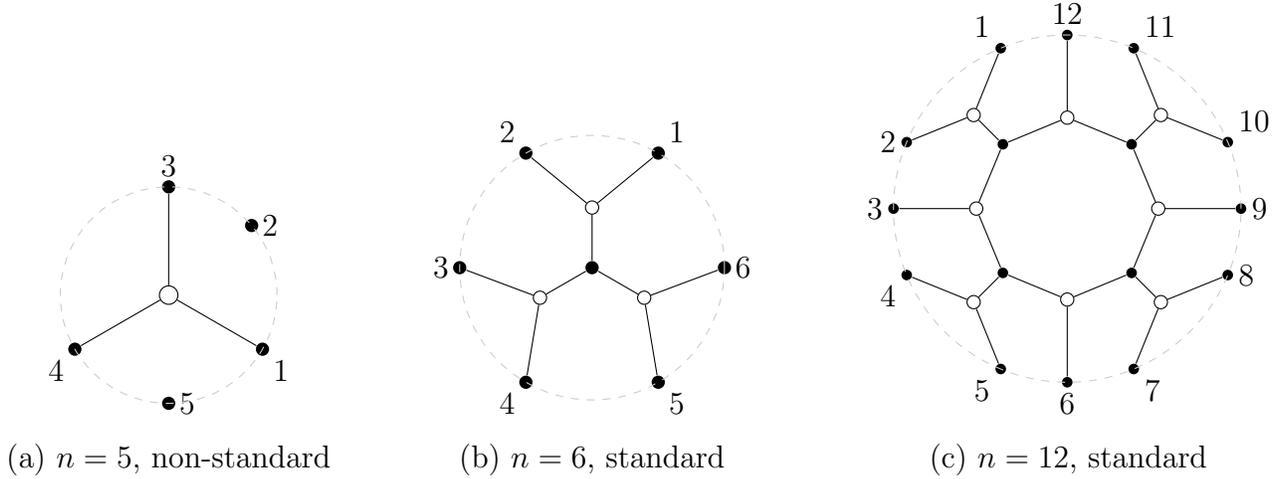

	For a non-elliptic web $W$, we can construct a \defn{dual CAT(0) planar graph} $Q = Q(W)$ with:
	\begin{itemize}
		\item Vertex set: the faces of $W$.
		\item Directed edges: $F \to F'$ if $F$ and $F'$ share an edge in $W$, 
		oriented so that black vertices of $W$ lie to the right of the directed 
		edge in $Q$.
	\end{itemize}
	This definition is called the ``dual quiver" in the work of Shen--Sun--Weng~\cite{SSW} and appeared earlier in the setting of plabic graphs \cite{Pos}.  Under the duality $W \to Q(W)$, white (resp. black) vertices become counterclockwise (resp. clockwise) oriented triangles.
	
	If $W$ is a labeled web, then we obtain a CCW-labeled CAT(0) planar graph $(Q,\z)$ so that:
	\begin{enumerate}
		\item if $b_a$ has degree $0$, then $z_a$ labels the face with $b_a$ on the boundary, and
		\item if $b_a$ has degree $\geq 1$, then $z_a$ labels the face that contains both $b_{a-1}$ and $b_{a}$.
	\end{enumerate}
	
	\cref{fig:web-dual-overlay} illustrates these conventions with a standard non-elliptic web $W$ and its dual standard CAT(0) graph $Q$. 
	\begin{figure}[ht]
		\centering
		\begin{tikzpicture}[scale=1.1,
			bvert/.style={circle, fill=black, inner sep=0, minimum size=4.5pt},
			wvert/.style={circle, draw=black, fill=white, inner sep=0, minimum size=5.5pt},
			qvert/.style={circle, fill=blue!60!black, inner sep=0, minimum size=6pt},
			midarrow/.style={decoration={markings, mark=at position 0.5 with {\arrow{stealth}}}, postaction={decorate}},
			]
			\coordinate (z2) at (0, 2.078);
			\coordinate (z4) at (-1.8, -1.039);
			\coordinate (z6) at (1.8, -1.039);
			\coordinate (z1) at ($(z6)!0.5!(z2)$);
			\coordinate (z3) at ($(z2)!0.5!(z4)$);
			\coordinate (z5) at ($(z4)!0.5!(z6)$);
			
			\coordinate (m12) at ($(z1)!0.5!(z2)$);
			\coordinate (m23) at ($(z2)!0.5!(z3)$);
			\coordinate (m34) at ($(z3)!0.5!(z4)$);
			\coordinate (m45) at ($(z4)!0.5!(z5)$);
			\coordinate (m56) at ($(z5)!0.5!(z6)$);
			\coordinate (m61) at ($(z6)!0.5!(z1)$);
			
			\coordinate (w1) at ($(z1)!0.6667!($(z2)!0.5!(z3)$)$);
			\coordinate (w2) at ($(z3)!0.6667!($(z4)!0.5!(z5)$)$);
			\coordinate (w3) at ($(z5)!0.6667!($(z6)!0.5!(z1)$)$);
			\coordinate (bc) at ($(z1)!0.6667!($(z3)!0.5!(z5)$)$);
			
			\coordinate (b1) at ($(m12)!-0.55!(w1)$);
			\coordinate (b2) at ($(m23)!-0.55!(w1)$);
			\coordinate (b3) at ($(m34)!-0.55!(w2)$);
			\coordinate (b4) at ($(m45)!-0.55!(w2)$);
			\coordinate (b5) at ($(m56)!-0.55!(w3)$);
			\coordinate (b6) at ($(m61)!-0.55!(w3)$);
			
			\draw[gray!35, line width=0.8pt] (w1) -- (bc);
			\draw[gray!35, line width=0.8pt] (w2) -- (bc);
			\draw[gray!35, line width=0.8pt] (w3) -- (bc);
			\draw[gray!35, line width=0.8pt] (w1) -- (b1);
			\draw[gray!35, line width=0.8pt] (w1) -- (b2);
			\draw[gray!35, line width=0.8pt] (w2) -- (b3);
			\draw[gray!35, line width=0.8pt] (w2) -- (b4);
			\draw[gray!35, line width=0.8pt] (w3) -- (b5);
			\draw[gray!35, line width=0.8pt] (w3) -- (b6);
			
			\node[wvert, opacity=0.4] at (w1) {};
			\node[wvert, opacity=0.4] at (w2) {};
			\node[wvert, opacity=0.4] at (w3) {};
			\node[bvert, opacity=0.3, minimum size=4pt] at (bc) {};
			\node[bvert, opacity=0.35, minimum size=3.5pt] at (b1) {};
			\node[bvert, opacity=0.35, minimum size=3.5pt] at (b2) {};
			\node[bvert, opacity=0.35, minimum size=3.5pt] at (b3) {};
			\node[bvert, opacity=0.35, minimum size=3.5pt] at (b4) {};
			\node[bvert, opacity=0.35, minimum size=3.5pt] at (b5) {};
			\node[bvert, opacity=0.35, minimum size=3.5pt] at (b6) {};
			
			\node[above right, gray, font=\footnotesize] at (b1) {$b_1$};
			\node[above left, gray, font=\footnotesize] at (b2) {$b_2$};
			\node[left, gray, font=\footnotesize] at (b3) {$b_3$};
			\node[below left, gray, font=\footnotesize] at (b4) {$b_4$};
			\node[below right, gray, font=\footnotesize] at (b5) {$b_5$};
			\node[right, gray, font=\footnotesize] at (b6) {$b_6$};
			
			\definecolor{qedge}{RGB}{30,130,90}
			\definecolor{qlabel}{RGB}{12,68,124}
			
			\draw[qedge, thick, midarrow] (z1) -- (z2);
			\draw[qedge, thick, midarrow] (z2) -- (z3);
			\draw[qedge, thick, midarrow] (z3) -- (z4);
			\draw[qedge, thick, midarrow] (z4) -- (z5);
			\draw[qedge, thick, midarrow] (z5) -- (z6);
			\draw[qedge, thick, midarrow] (z6) -- (z1);
			\draw[qedge, thick, midarrow] (z3) -- (z1);
			\draw[qedge, thick, midarrow] (z5) -- (z3);
			\draw[qedge, thick, midarrow] (z1) -- (z5);
			
			\node[qvert] at (z1) {};
			\node[qvert] at (z2) {};
			\node[qvert] at (z3) {};
			\node[qvert] at (z4) {};
			\node[qvert] at (z5) {};
			\node[qvert] at (z6) {};
			
			\node[right=2pt, color=qlabel, font=\small\bfseries] at (z1) {$z_1$};
			\node[above=2pt, color=qlabel, font=\small\bfseries] at (z2) {$z_2$};
			\node[left=2pt, color=qlabel, font=\small\bfseries] at (z3) {$z_3$};
			\node[left=2pt, color=qlabel, font=\small\bfseries] at (z4) {$z_4$};
			\node[below=2pt, color=qlabel, font=\small\bfseries] at (z5) {$z_5$};
			\node[right=2pt, color=qlabel, font=\small\bfseries] at (z6) {$z_6$};
			
		\end{tikzpicture}
		\caption{A standard web $W$ with $n=6$ and its dual $Q(W)$.  The web $W$ (gray) has three white interior vertices and one black interior vertex; its dual $Q(W)$ (green edges, blue vertices) is a CCW-labeled CAT(0) planar graph.  Each face of $W$ contains a vertex $z_a$ of $Q$, with $z_a$ assigned to the face containing $b_{a-1}$ and $b_a$ on its boundary.  White vertices of $W$ become CCW-oriented triangles $z_1z_2z_3$, $z_3z_4z_5$, $z_5z_6z_1$, and the black vertex becomes the CW-oriented triangle $z_1z_5z_3$.}
		\label{fig:web-dual-overlay}
	\end{figure}
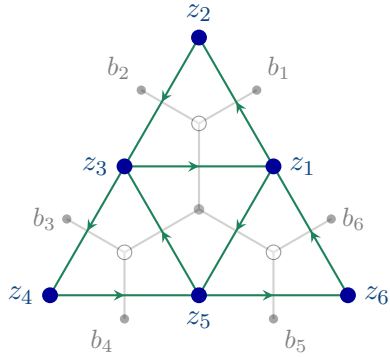

	By convention the $n$ boundary vertices are labeled $1, 2, \ldots,n$ counterclockwise around the boundary of the disk.  
	A web $W$ can be made standard by the following operations, producing a standard web $W^\std$, see \cref{fig:webops1}:
	\begin{enumerate}
		\item Valence $0$ boundary vertices can be removed.
		\item Valence $d > 1$ boundary vertices can be split into $d$ boundary vertices.
		\item A white boundary vertex can be pushed into the interior to become a white trivalent vertex resulting in two new black boundary vertices.
	\end{enumerate}

	\begin{figure}
		\centering
		\includegraphics[width=0.8\linewidth]{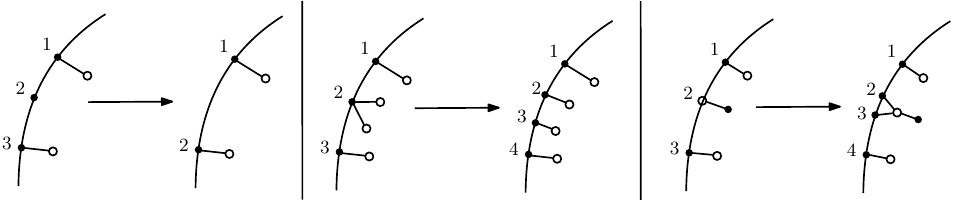}
		\caption{Three basic operations on non-elliptic webs produce standard webs.}
		\label{fig:webops1}
	\end{figure}
		\begin{figure}[h!]
		\centering
		\includegraphics[width=1.0\linewidth]{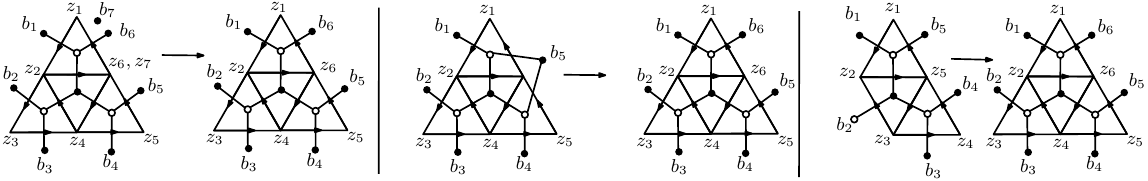}
		\caption{Basic operations for standardizing webs and scaffolds.}
		\label{fig:cat0qops}
	\end{figure}
	
	A standard CAT(0) planar graph $(Q,\z)$ is a CCW-labeled planar graph such that the boundary is oriented counterclockwise.  The standardization procedure above corresponds to the following three operations in the dual CAT(0) planar graph:
	\begin{enumerate}
		\item Boundary vertices labeled by a subset can be reduced to have a single label.
		\item Boundary vertices without a label are given labels.		
		\item A triangle can be attached to $Q$ whenever a boundary edge is in the clockwise direction.
	\end{enumerate}

	It is straightforward to see that a web $W$ is connected if and only if $Q(W)$ is simple.  If $W$ is disconnected with components $W_1,W_2,\ldots$, then $Q(W)$ is obtained by gluing $Q(W_1), Q(W_2), \ldots$ along vertices.		
		
	\subsection{Strands and the planar basis expansion}
	
	\begin{definition}
		Let $W$ be a web.  The \defn{strands} of $W$ are the directed paths (or cycles) in $W$ obtained by the rule:
		\begin{itemize}
			\item turn \textbf{left} at white interior vertices, and
			\item turn \textbf{right} at black interior vertices.
		\end{itemize}
	\end{definition}
	
	In a non-elliptic web $W$, all strands $\str$ are directed paths from a boundary vertex to a boundary vertex; we denote by $\str_i$ the strand that ends at vertex $b_i$.  If $b_i$ has degree greater than one, then there will be multiple such strands.  We typically use the notation $\str_i$ when $W$ is standard.

\begin{figure}
\begin{center}
$$
\begin{tikzpicture}[scale=0.4,
			bvert/.style={circle, fill=black, inner sep=0, minimum size=4pt},
			wvert/.style={circle, draw=black, fill=white, inner sep=0, minimum size=4pt},
			midarrow/.style={decoration={markings, mark=at position 0.5 with {\arrow{stealth}}}, postaction={decorate}},
			]
			\def\R{2.2}
			
			\node[wvert, minimum size=5pt] (wT)  at (90:\R)  {};
			\node[bvert]                   (bTL) at (135:\R) {};
			\node[wvert, minimum size=5pt] (wL)  at (180:\R) {};
			\node[bvert]                   (bBL) at (225:\R) {};
			\node[wvert, minimum size=5pt] (wB)  at (270:\R) {};
			\node[bvert]                   (bBR) at (315:\R) {};
			\node[wvert, minimum size=5pt] (wR)  at (0:\R)   {};
			\node[bvert]                   (bTR) at (45:\R)  {};
			
			%
			\draw (wT) -- (bTR) -- (wR) -- (bBR) -- (wB) -- (bBL) -- (wL) -- (bTL) -- (wT) -- cycle;
			%
			\node[bvert] (d1)  at (90:4.2)  {};  \draw (wT) -- (d1);
			\node[bvert] (d4)  at (180:4.2) {};  \draw (wL) -- (d4);
			\node[bvert] (d7)  at (270:4.2) {};  \draw (wB) -- (d7);
			\node[bvert] (d10) at (0:4.2)   {};  \draw (wR) -- (d10);
			
			%
			%
			\node[wvert, minimum size=5pt] (eTL) at (135:3.2) {};
			\draw (bTL) -- (eTL);
			\node[bvert] (d2) at (112.5:4.2) {};
			\node[bvert] (d3) at (157.5:4.2) {};
			\draw (eTL) -- (d2);
			\draw (eTL) -- (d3);
			%
			\node[wvert, minimum size=5pt] (eBL) at (225:3.2) {};
			\draw (bBL) -- (eBL);
			\node[bvert] (d5) at (202.5:4.2) {};
			\node[bvert] (d6) at (247.5:4.2) {};
			\draw (eBL) -- (d5);
			\draw (eBL) -- (d6);
			%
			\node[wvert, minimum size=5pt] (eBR) at (315:3.2) {};
			\draw (bBR) -- (eBR);
			\node[bvert] (d8) at (292.5:4.2) {};
			\node[bvert] (d9) at (337.5:4.2) {};
			\draw (eBR) -- (d8);
			\draw (eBR) -- (d9);
			%
			\node[wvert, minimum size=5pt] (eTR) at (45:3.2) {};
			\draw (bTR) -- (eTR);
			\node[bvert] (d11) at (22.5:4.2) {};
			\node[bvert] (d12) at (67.5:4.2) {};
			\draw (eTR) -- (d11);
			\draw (eTR) -- (d12);
			%
			\node[above]       at (d1)  {$12$};
			\node[above left]  at (d2)  {$1$};
			\node[left]        at (d3)  {$2$};
			\node[left]        at (d4)  {$3$};
			\node[below left]  at (d5)  {$4$};
			\node[below left]  at (d6)  {$5$};
			\node[below]       at (d7)  {$6$};
			\node[below right] at (d8)  {$7$};
			\node[right]       at (d9)  {$8$};
			\node[right]        at (d10) {$9$};
			\node[above right] at (d11) {$10$};
			\node[above right] at (d12) {$11$};
			%
			\draw[thin, dashed, gray!40] (0,0) circle (4.2);
			\draw[red,thick,midarrow] (d7)--(wB);
			\draw[red,thick,midarrow] (wB)--(bBL);
			\draw[red,thick,midarrow] (bBL)--(wL);
			\draw[red,thick,midarrow] (wL)--(d4);
			
			\draw[blue,thick,midarrow] (d12)--(eTR);
			\draw[blue,thick,midarrow] (eTR)--(d11);
		\end{tikzpicture}
		\qquad \qquad
\begin{tikzpicture}[scale=0.4,
			bvert/.style={circle, fill=black, inner sep=0, minimum size=4pt},
			wvert/.style={circle, draw=black, fill=white, inner sep=0, minimum size=4pt},
			midarrow/.style={decoration={markings, mark=at position 0.5 with {\arrow{stealth}}}, postaction={decorate}},
			]
			\def\R{2.2}
			
			\node[wvert, minimum size=5pt] (wT)  at (90:\R)  {};
			\node[bvert]                   (bTL) at (135:\R) {};
			\node[wvert, minimum size=5pt] (wL)  at (180:\R) {};
			\node[bvert]                   (bBL) at (225:\R) {};
			\node[wvert, minimum size=5pt] (wB)  at (270:\R) {};
			\node[bvert]                   (bBR) at (315:\R) {};
			\node[wvert, minimum size=5pt] (wR)  at (0:\R)   {};
			\node[bvert]                   (bTR) at (45:\R)  {};
			
			%
			\draw (wT) -- (bTR) -- (wR) -- (bBR) -- (wB) -- (bBL) -- (wL) -- (bTL) -- (wT) -- cycle;
			%
			\node[bvert] (d1)  at (90:4.2)  {};  \draw (wT) -- (d1);
			\node[bvert] (d4)  at (180:4.2) {};  \draw (wL) -- (d4);
			\node[bvert] (d7)  at (270:4.2) {};  \draw (wB) -- (d7);
			\node[bvert] (d10) at (0:4.2)   {};  \draw (wR) -- (d10);
			
			%
			%
			\node[wvert, minimum size=5pt] (eTL) at (135:3.2) {};
			\draw (bTL) -- (eTL);
			\node[bvert] (d2) at (112.5:4.2) {};
			\node[bvert] (d3) at (157.5:4.2) {};
			\draw (eTL) -- (d2);
			\draw (eTL) -- (d3);
			%
			\node[wvert, minimum size=5pt] (eBL) at (225:3.2) {};
			\draw (bBL) -- (eBL);
			\node[bvert] (d5) at (202.5:4.2) {};
			\node[bvert] (d6) at (247.5:4.2) {};
			\draw (eBL) -- (d5);
			\draw (eBL) -- (d6);
			%
			\node[wvert, minimum size=5pt] (eBR) at (315:3.2) {};
			\draw (bBR) -- (eBR);
			\node[bvert] (d8) at (292.5:4.2) {};
			\node[bvert] (d9) at (337.5:4.2) {};
			\draw (eBR) -- (d8);
			\draw (eBR) -- (d9);
			%
			\node[wvert, minimum size=5pt] (eTR) at (45:3.2) {};
			\draw (bTR) -- (eTR);
			\node[bvert] (d11) at (22.5:4.2) {};
			\node[bvert] (d12) at (67.5:4.2) {};
			\draw (eTR) -- (d11);
			\draw (eTR) -- (d12);
			%
			\node[above]       at (d1)  {$12$};
			\node[above left]  at (d2)  {$1$};
			\node[left]        at (d3)  {$2$};
			\node[left]        at (d4)  {$3$};
			\node[below left]  at (d5)  {$4$};
			\node[below left]  at (d6)  {$5$};
			\node[below]       at (d7)  {$6$};
			\node[below right] at (d8)  {$7$};
			\node[right]       at (d9)  {$8$};
			\node[right]        at (d10) {$9$};
			\node[above right] at (d11) {$10$};
			\node[above right] at (d12) {$11$};
			
			\node[above] at (wB) {$v$};
			%
			\draw[thin, dashed, gray!40] (0,0) circle (4.2);
		
			\draw[red,thick,midarrow] (wB)--(bBL);
			\draw[red,thick,midarrow] (bBL)--(wL);
			\draw[red,thick,midarrow] (wL)--(d4);
			\draw[blue,thick,midarrow] (wB)--(d7);
			
			\draw[green,thick,midarrow] (wB)--(bBR);
			\draw[green,thick,midarrow] (bBR)--(eBR);
			\draw[green,thick,midarrow] (eBR)--(d9);
			
		\end{tikzpicture}
$$
\end{center}
\caption{Left: A web $W$ and two strands $\eta_3$ and $\eta_{10}$.  Right: the vertex $v$ has strand triple $(3,6,8)$.}
\label{fig:strands}
\end{figure}
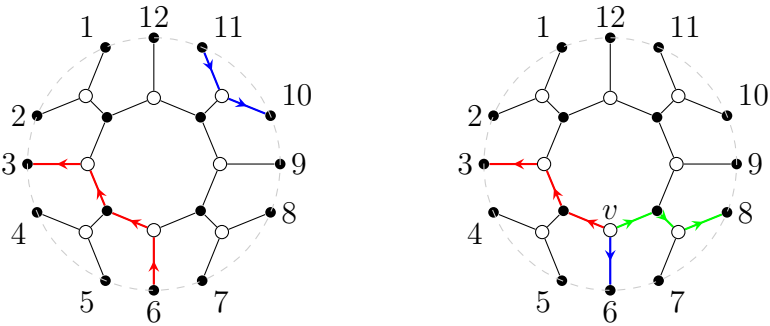
	
	We say that two strands $\str,\str'$ in a web $W$ \emph{intersect} if they intersect in the interior of the web $W$.  The following result supports the intuition that strands in non-elliptic webs behave like geodesics in nonpositive curvature.
	
	\begin{proposition}\label{prop:strandintersect}
		Any two strands in a non-elliptic web $W$ intersect at most once.  If two strands $S, S'$ intersect along an edge $e$ then the only faces that $S$ and $S'$ are both adjacent to are faces that are incident to one of the vertices of $e$.
	\end{proposition}
	\cref{prop:strandintersect} is proved in \cref{sec:strands}.

	\begin{definition}\label{def:strandtriple}
		If $a$ is an interior vertex of $W$, then there are three strands passing through $a$, ending at boundary vertices labeled $(i_a,j_a,k_a)$, called the \defn{strand triple} of $a$.  The interior vertex $a$ corresponds to a triangle $T = T_a$ of $Q(W)$, and we also call $(i_a,j_a,k_a)$ the strand triple of $T$.
	\end{definition}
	
	\subsection{Normal webs}
	
	\begin{definition}\label{def:normalweb}
		A non-elliptic web is \defn{normal} if whenever two strands $\eta, \eta'$ intersect they have different labels.
	\end{definition}
	
	It follows immediately that a non-elliptic web $W$ is normal if and only if all strand triples contain distinct elements.  A web $W$ has black boundary if all boundary vertices are black.  This is equivalent to $Q(W)$ having counterclockwise oriented boundary.
	
	\begin{definition}
	A normal web $W$ is \defn{cyclic} if it has a cyclic strand triple.  Otherwise $W$ is \emph{cyclic-less}.
	\end{definition}
	
	\begin{theorem}\label{thm:normalnormal}
		Let $W$ be a non-elliptic web with black boundary and $(Q(W),\z(W))$ be its dual CAT(0) planar graph.  Then $W$ is normal if and only if $Q(W)$ is normal.  If $W$ is normal, then $W$ is cyclic-less if and only if $Q(W)$ is cyclic-less.
	\end{theorem}
	
	\cref{thm:normalnormal} is proven in \cref{sec:normal}.

	\subsection{Planar basis expansion}
	Recall that the tropical Grassmannian is a polyhedral fan in the ambient space $\R^{\binom{[n]}{k}}$.  The planar basis $\{\h_J \mid J \in \binom{[n]}{k}\}$ (\cref{def:planar-basis}) is a basis of $\R^{\binom{[n]}{k}}$ adapted to the study of the tropical Grassmannian.  In particular, every element $\h_J$ belongs to $\Trop_{>0} \Gr(k,n)$.
	
	\begin{theorem}\label{thm:pbexpansion}
		Let $W$ be a non-elliptic web where all boundary vertices are black.
		Let $(Q,\z)$ be the dual CAT(0) planar graph. The Pl\"ucker function 
		$\tropP_\bullet = \tropP_\bullet(Q,\z)$ of $(Q,\z)$ has the planar basis expansion:
		\[
		\tropP_\bullet \equiv \sum_{T \text{ white}} \mathfrak{h}_{i_T,j_T, k_T} - 
		\sum_{T \text{ black}} \mathfrak{h}_{i_T, j_T, k_T} \pmod{\Lin_{3,n}},
		\]
		where the sums are over white and black interior triangles $T$ of $Q$, 
		respectively, and $(i_T, j_T, k_T)$ denotes the strand triple of $T$.
	\end{theorem}

	\begin{example}\label{example: 12-gonPlucker}
		Let $W$ be the web in \cref{fig:web-examples}(c) and \cref{fig:strands}, and $(Q,\z)$ be the dual CAT(0) planar graph.  Computing all the strand triples, we find
		\begin{eqnarray*}
			\tropP_\bullet(Q,\z) & = & (\h_{3,4,7}+\h_{1,4,12}+\h_{6,7,10}+\h_{1,9,10})+(\h_{2,4,11}+\h_{2,5,7}+\h_{5,8,10}+\h_{1,8,11})\\
							& - & (\h_{1,4,11}+\h_{1,8,10}+\h_{2,4,7}+\h_{5,7,10}),
		\end{eqnarray*}
		where the first four positive terms belong to the four white corners, the second four positive terms belong to the midpoints of the horizontal and vertical edges, and the four negative terms correspond to the four black corner vertices. 
		\end{example}
\cref{example: 12-gonPlucker} is continued in 
\cref{example: big 312}.
		\section{Membranes: LP vs. KT}\label{sec:LPKT}
	\def\membrane{{\mathcal{M}}}

	In this section, we discuss membranes in the sense of Lam-Postnikov (LP) \cite{LP} and in the sense of Keel-Tevelev (KT) \cite{KT}, respectively.  Roughly speaking, KT membranes are naturally identified with triangulated tropical linear spaces, while LP membranes are certain two-dimensional surfaces in the bounded complex of a \emph{positive} tropical linear space.  The map $\mu$ in \eqref{eq:embedding} sends a normal CAT(0) planar graph, or more generally a canvas in a scaffold (\cref{rem:canvas}), to a LP membrane.  The Keel-Tevelev isomorphism (\cref{thm:KT}) then maps the LP membrane into a KT membrane within the affine building $\B$ of $\PGL(k)$.

In this section, we let $W$ be a normal web and $(Q,\z)$ be the dual normal CAT(0) planar graph.  In particular, the boundary vertices $\z$ are arranged in counterclockwise cyclic order on the boundary $\partial Q$, and we always assume that $n \geq 3$.
	
	\subsection{Lam-Postnikov membranes}
	
	In \cite{LP}, membranes are defined in the general setting of plabic graphs.  We specialize to non-elliptic webs and change coordinates to our conventions.
	
	\begin{definition}\label{def:membrane}
		Let $W$ be a web where boundary vertices have degree one and $(Q,\z)$ its dual labeled CAT(0) planar graph.  A \defn{membrane} $\membrane(W)$ for $W$ is a lifting of each vertex $v \in V(Q)$ to an integer point $\mu(v) \in \Z^n/\one$ satisfying the following property: if $v \to w$ is an edge of $Q$ and $e$ denotes the dual edge of $W$, let $\eta_i$ be the strand passing through $e$ with $v$ on the right and $\eta_j$ be the strand passing through $e$ with $w$ on the right, then $\mu(w) - \mu(v) = e_{[i+1,j]}$.
	\end{definition}
	
	\begin{rem}
	The \emph{boundary loop} $L = \partial \membrane$ of the membrane $\membrane(W)$ is the restriction of $\mu$ to $\partial Q$.  In \cite{LP}, the problem of finding a membrane with a specified boundary loop is considered.  Under a unimodality condition \cite[Theorem 23.6]{LP} on the boundary loop $L$, membranes with minimal area with $L =  \partial \membrane$ are in bijection with reduced plabic graphs within a move equivalence class.
	\end{rem}	
	
	In a normal web, the strand triples $(i,j,k)$ consist of three distinct elements.
	
	\begin{proposition}[{\cite[Section 21]{LP}}]\label{prop:membrane-triangle}
		Suppose $\membrane(W)$ is a membrane for a normal web $W$ with dual CAT(0) planar graph $Q$.  For a white (resp. black) triangle $T$ of $Q$ with vertices $v,u,w$, the triangle $(\mu(v),\mu(u),\mu(w))$ is equal to $\conv(h_{i_T}, h_{j_T}, h_{k_T})$ (resp. $\conv(-h_{i_T}, -h_{j_T}, -h_{k_T})$) up to an affine translation, where $h_i := e_1 + e_2 + \cdots + e_i$.
	\end{proposition}

	\subsection{Membrane of a normal web}\label{sec:membrane}
	Recall that in \cref{sec:main} we have fixed a coloring $c(\cdot)$ on $V(Q)$ satisfying \cref{def:color}.  By \cref{lem:colordist}, we have 
	\begin{lemma}\label{lem:color}
		For any two vertices $v,w$, we have $\frac{1}{3}\left(\dist(v,w) +c(v) - c(w)\right) \in \Z$.
	\end{lemma}
	
	As in \eqref{eq:modify}, we write $c(\z):= (c(z_1),c(z_2),\ldots,c(z_n)) \in \{0,1,2\}^n$ and define $\bp_\bullet := \tropP_\bullet^{-c(\z)/3}$.  The following result combines \cref{prop:scaffoldint} with \cref{thm:main}(1,3).
		
	\begin{proposition}
		Let $(Q,\z)$ be a labeled CAT(0) planar graph.  Then we have $\bp_\bullet \in \Dr(3,n)(\Z)$.  If $(Q,\z)$ is normal, then $\bp_\bullet \in \Trop_{>0}\Gr(3,n)(\Z)$.
	\end{proposition}

	\begin{theorem}\label{thm:membrane}
		Let $W$ be a normal web with dual $(Q,\z)$.  The formula (cf. \eqref{eq:embedding})
		\begin{equation}\label{eq:mu}
		\mu: V(Q) \to \Z^n/\one, \qquad \mu(v):= \frac{1}{3}\left((\dist(v,1),\dist(v,2),\ldots,\dist(v,n)) + c(v)\one - c(\z) \right).
	\end{equation}
		defines a membrane for $W$.
	\end{theorem}
	
	Recall that we defined a matroid $M_v$ for $v \in V(Q)$ in \cref{def:Mv}.
	\begin{proposition}\label{prop:Mv-positroid}
		Let $(Q,\z)$ be a normal CAT(0) planar graph.  For any $v \in V(Q)$, the matroid $M_v$ is a loopless positroid.
	\end{proposition}
	
	We will prove \cref{thm:membrane} and \cref{prop:Mv-positroid} in \cref{sec:positroids}.

\begin{example}\label{example: big 312}
	
	Let us apply \cref{thm:membrane} to compute the image in the membrane and tropical linear space of the triangle $(v,z_3,z_5)$ in \cref{fig:bigexample312}.  We compute the vertices of the triangle $(v,z_3,z_5)$.  Setting $c(v)=0$, we have
	$$c(\z) = (0,1,2,0,1,2,0,1,2,0,1,2).$$
	By computing the directed distances we obtain 
	\begin{eqnarray*}
		\mu(v) & = & \frac{1}{3}\left((\delta(v,1),\ldots, \delta(v,12)) + 0\cdot \one - c(\z) \right)\\
		& = & \frac{1}{3}\left((3,1,2,3,1,2,3,1,2,3,1,2)- (0,1,2,0,1,2,0,1,2,0,1,2) \right)\\
		& = & (1,0,0,1,0,0,1,0,0,1,0,0).
	\end{eqnarray*}
	Similarly,
	\begin{eqnarray*}
		\mu(z_3) & = & \frac{1}{3}\left((4,2,0,1,2,3,4,2,3,4,2,3) + 2\cdot \one - (0,1,2,0,1,2,0,1,2,0,1,2)\right)\\
		& = & (2,1,0,1,1,1,2,1,1,2,1,1).
	\end{eqnarray*}
	Finally, $\mu(z_5) = (2,1,0,1,0,0,1,1,1,2,1,1)$, hence up to translation we have
	\begin{eqnarray*}
	(\mu(v),\mu(z_3),\mu(z_5)) & \equiv & (-h_{2},-h_4,-h_7) = (-e_{[1,2]},-e_{[1,4]},-e_{[1,7]}),
	\end{eqnarray*}
	agreeing with \cref{prop:membrane-triangle}.  
	For instance, the difference of the two is the indicator vector of the cyclic interval $[5,2]$:
	\begin{eqnarray*}
		\mu(z_3) - \mu(v) & = & (1,1,0,0,1,1,1,1,1,1,1,1).
	\end{eqnarray*}	
	One can check that the image of $\mu$ satisfies the axioms in \cref{def:membrane} of the membrane. 
	In \cref{example: octagon end} we return to compute the facet inequalities defining the matroid polytope $P_{M_v}$ corresponding to $v$. 


\end{example}
	\begin{figure}[h!]
	\centering
	\includegraphics[width=0.65\linewidth]{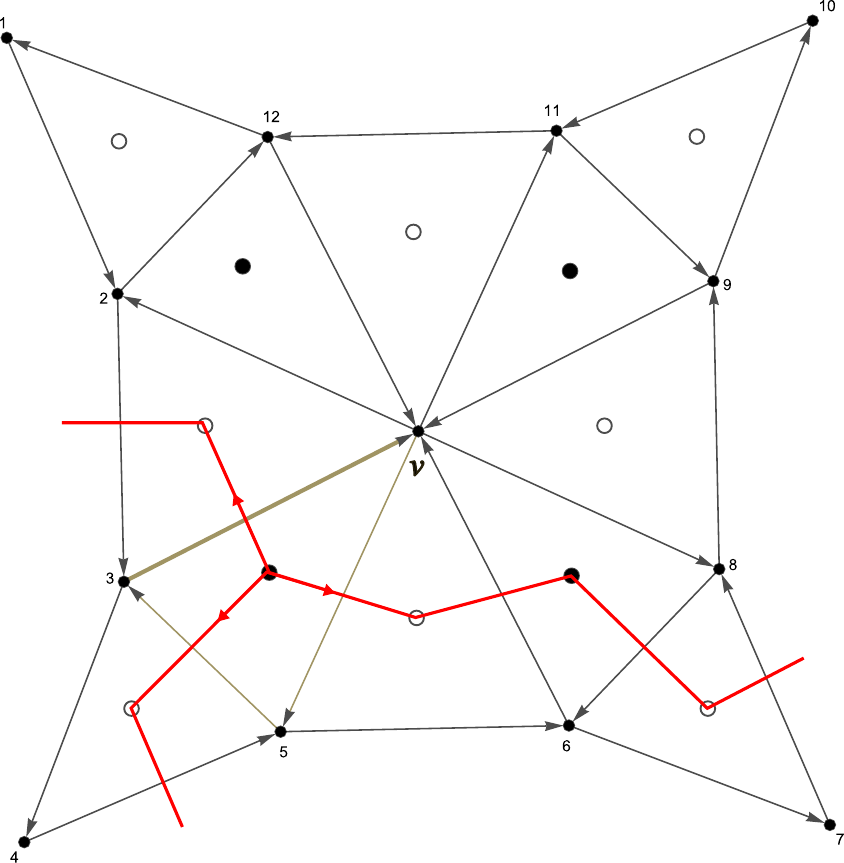}
	\caption{The CAT(0) planar graph in \cref{example: big 312}.  The brown edge connecting $z_3$ and $v$ is computed and is found to be equal to $e_{[5,2]}$.  The red strand triple $(2,4,7)$ encodes the vertices of the image under $\mu$ of the triangle $(v,z_3,z_5)$ in the membrane.  Additionally, all eight facets of $M_v$ are identified in \cref{example: octagon end}.}
	\label{fig:bigexample312}
\end{figure}


	\subsection{Keel-Tevelev membranes}
	Let $\K:= \C((t))$ and $\O:=\C[[t]]$.  A \emph{lattice} $\Lambda$ is a free $\O$-submodule of $\K^3$ of full rank.  We define an equivalence relation on lattices by $\Lambda \sim \Lambda'$ if $\Lambda = c \Lambda'$ for $c \in \K^\times$.  Let $\B$ denote the affine $\PGL(3)$-building (see \cref{sec:buildings}), which we view as a simplicial complex, on the set of equivalence classes $[\Lambda]$ of lattices.  
	
	\begin{definition}
		Let $M=(v_1,\ldots,v_n) \in (\K^3)^n$.  The \defn{membrane} $[M]$ is the set of lattices of the form
		$$
		{\rm span}_{\O}\left(t^{a_1} v_1, \ldots, t^{a_n} v_n\right)
		$$
		as $\a = (a_1,\ldots,a_n)$ varies over $\Z^n$.  The membrane $[M]$ has the structure of a simplicial complex inherited from $\B$.
	\end{definition}
	
	Let $\tropP_\bullet \in \Trop \, \Gr(3,n)(\Z)$ arise from a $3 \times n$ matrix $M$ with columns $v_1,\ldots,v_n \in \K^3$.  Keel and Tevelev \cite{KT} define an isomorphism of simplicial complexes (see \cref{thm:KT})
	$$
	\iota: L(\tropP_\bullet) \to [M].
	$$
	
	Using tropical linear spaces, we connect the notions of Lam-Postnikov membranes and Keel-Tevelev membranes.  Let $\K_{>0} \subset \K$ denote the semifield of formal power series with a positive leading coefficient.
	\begin{theorem}\label{thm:LPKT}
		Let $\bp_\bullet \in \Trop_{>0} \Gr(3,n)(\Z)$ be the (modified) positive tropical Pl\"ucker vector of a normal CAT(0) planar graph $(Q,\z)$, and suppose that $\bp_\bullet$ is the tropical Pl\"ucker vector of $V = (v_1,v_2,\ldots,v_n) \in \Gr(3,n)(\K_{>0})$.  The composition $\iota \circ \mu$ identifies $Q$ with the (one-skeleton of the) diskoid of Fontaine--Kamnitzer--Kuperberg, sitting inside the membrane $[M]$.  The map $\iota$ restricted to $\membrane(W) = \mu(Q)$ sends the LP membrane $\membrane(W)$ into the KT membrane $[M]$.
	\end{theorem}
	We prove \cref{thm:LPKT} in \cref{sec:KT}.

	\section{Application to symbol alphabet recursion}\label{sec:RSV}
	
	In this section we explain how our results prove a conjecture of Ren--Spradlin--Volovich \cite{RSV} concerning symbol alphabets for $N=4$ super Yang-Mills amplitudes.  A full review of the physics background is beyond the scope of this work.  We will be content to show that our construction explains the conjectural recursion of \cite[(4.10)]{RSV}.
	
	In the following, we set $\tropP_\bullet(W) := \tropP_\bullet(Q(W),\z(W))$.
	\begin{proposition}\label{prop:RSV}
	We have the following identity:
	$$
		 \tropP_\bullet\left( \vcenter{\hbox{\begin{tikzpicture}
	[scale=0.6,bvert/.style={circle, fill=black, inner sep=0, minimum size=5pt},
			wvert/.style={circle, draw=black, fill=white, inner sep=0, minimum size=5pt}]
	\node[wvert] (x) at (0,0) {};
	\node[bvert] (y) at  (60:1) {};
	\node[wvert] (w) at  ($(y)+(120:1)$) {};
	\node[bvert] (a) at ($(w) + (180:1)$) {};
	\node[bvert] (b) at ($(w) + (60:1)$) {};
	\node[wvert] (z) at  ($(y)+(1,0)$) {};
	\node[bvert] (v) at  ($(z)+(-60:1)$) {};
	\node[bvert] (c) at ($(z)+(60:1)$) {};
	\node at (1,0.2) {$F$};
	\draw (x)--(y);
	\draw[red,thick] (y)--(z);
	\draw (z)--(v);
	\draw (y)--(w)--(a);
	\draw (z)--(c);
	\draw (w)--(b);
	\draw[dashed] (x) to[out = -60, in=240] (v);
	\node at ($(a) + (-0.3,0.2)$) {$a$};
	\node at ($(b) + (-0.3,0.2)$) {$b$};
	\node at ($(c) + (0,0.4)$) {$c$};
	\end{tikzpicture}}}\right) \equiv
	\iotaStd\left(\tropP_\bullet\left( \vcenter{\hbox{\begin{tikzpicture}
	[scale=0.6,bvert/.style={circle, fill=black, inner sep=0, minimum size=5pt},
			wvert/.style={circle, draw=black, fill=white, inner sep=0, minimum size=5pt}]
	\node[bvert] (x) at (2,0) {};
	\node[wvert] (y) at  ($(x)+(100:1)$) {};
	\node[bvert] (aa) at  ($(y)+(120:1)$) {};
	\node[bvert] (c) at  ($(y)+(60:1)$) {};
	\node[wvert] (v) at  ($(0,0)$) {};
	\node[bvert] (z) at  ($(v)+(90:1)$) {};
	\node[wvert] (w) at  ($(z)+(60:1)$) {};
	\node[wvert] (u) at  ($(z)+(160:1)$) {};
	\node[bvert] (cc) at  ($(w)+(130:1)$) {};
	\node[bvert] (m) at ($(w)+(70:1)$) {};
	\node[bvert] (b) at  ($(u)+(90:1)$) {};
	\node[bvert] (a) at ($(u)+(-1,0)$) {};
	\draw (x)--(y)--(c);
	\draw[red,thick] (y)--(aa);
	\draw (v)--(z)--(u)--(a);
	\draw (u)--(b);
	\draw[red,thick] (z)--(w)--(m);
	\draw[red,thick] (w)--(cc);
	\draw[dashed] (v) to[out = -60, in=240] (x);
	\node at ($(aa) + (-0.3,0.2)$) {$a'$};
	\node at ($(a) + (-0.3,0.2)$) {$a$};
	\node at ($(m) + (0,0.3)$) {$m$};
	\node at ($(b) + (-0.3,0.2)$) {$b$};
	\node at ($(c) + (0,0.3)$) {$c$};
	\node at ($(cc) + (0,0.3)$) {$c'$};
	\end{tikzpicture}}}\right) - \h_{m,a,c}
\right)
	$$
	up to lineality.  Here, $a,b,c,\ldots$ are boundary vertices and $F$ is an interior face.  The rest of the webs are unchanged.  The map $\iotaStd$ sends $a' \mapsto a$ and $c' \mapsto c$ when those letters index a planar basis element.
	\end{proposition}
	
	\begin{proof}
	Let the first diagram be $Q$ and the modified one be $Q'$.
	Then by \cref{thm:pbexpansion}, the planar basis expansions are
	\begin{align*}
	\tropP_\bullet(Q) &\equiv \h_{x,a,b} - \h_{x,a,c} + \h_{a,c,v} + \cdots \\
	\tropP_\bullet(Q') &\equiv \h_{x,a,b} + \h_{a,c',m} - \h_{x,a,c'} + \h_{a',c,v} + \cdots 
	\end{align*}
	where the $\cdots$ terms agree.  Here, $x$ (resp. $v$) are the targets of the strands that head down along the perimeter of $F$ on the left (resp. right) side.
	\end{proof}
	
	\begin{rem}
	The \emph{kinematic function} $X_R(W)$ of \cite{RSV} can be defined to be the lift of $\rho(\Psi(\tropP_\bullet(W)))$ where $\rho: {\mathbb T}^{3,n} \to \Trop_{>0} \Gr(3,n)$ denotes the positive parametrization of \cite{ELnc}, and $\Psi$ is defined in \cref{sec: planar basis}.
	
	\end{rem} 
	
	\part{CAT(0) planar graphs}\label{part:CAT0}	
	
	In this part we discuss the discrete geometry of CAT(0) 
	planar graphs.
	\section{Discrete Distance Geometry}\label{sec:combinatorial-structure}
	
	\subsection{The Discrete Gauss-Bonnet formula}\label{subsec:boundary-degree}
	\cref{lem:graph} is a discrete Gauss--Bonnet theorem for triangulated disks with nonpositive curvature.  The condition $\deg(v) \geq 6$ at interior vertices is the discrete analogue of nonpositive curvature; if $\deg(v) = 6$ then we are locally flat (curvature 0).  The local condition on curvature controls the global geometry of CAT(0) planar graphs throughout the paper: it guarantees the existence of at least three acute boundary vertices (\cref{lem:acute}), which is the starting point for the study of normal graphs (\cref{prop:normal}), and it provides the area-reduction arguments that establish unique minimizing neighbors (see \cref{lem:drop}) and the local structure of level sets (see \cref{lem:combdist}).
	\begin{lemma}\label{lem:graph}
		Let $H$ be an undirected connected planarly embedded graph whose 
		perimeter is a simple polygonal curve such that all interior vertices 
		have degree at least $6$, and all interior faces are triangles. For a 
		boundary vertex $x \in \partial H$, let $\deg(x) \geq 2$ denote its degree. Then
		\[
		\sum_{x \in \partial H} \max(4 - \deg(x), 0) \geq 6,
		\]
		where the summation is over all boundary vertices.
	\end{lemma}
	
	\begin{proof}
		Let $V$, $E$, $F$ denote the number of vertices, edges, and faces 
		respectively, and let $B$ denote the number of boundary vertices. Since 
		all interior faces are triangles, we have
		\[
		3(F - 1) = 2E - B, \qquad 2E = \sum_{y \in V(H)} \deg(y),
		\]
		where the second summation is over all vertices.  Euler's formula gives
		\[
		2 = V - E + F = V - E + \left(\frac{2}{3}E - \frac{1}{3}B + 1\right) =1 + V - \frac{1}{3}B - 
		\frac{1}{6}\sum_{y} \deg(y),
		\]
		where the summation is over all vertices $y$.  Since interior vertices have degree at least $6$,
		\[
		1 \leq V - \frac{1}{3}B - (V - B) - \frac{1}{6}\sum_{x \in \partial H} 
		\deg(x) = \frac{2}{3}B - \frac{1}{6}\sum_{x \in \partial H} \deg(x).
		\]
		
		Since $\frac{2}{3}B = \frac{1}{6}\sum_{x \in \partial H} 4$, it follows that $\sum_{x \in \partial H} \max(4 - \deg(x), 0) \geq 6$.
	\end{proof}

	\subsection{Combinatorial Distance and Level Sets}\label{subsec:comb-distance}
	
	For vertices $v, w$ in a CAT(0) planar graph $Q$, the \emph{combinatorial 
		distance} $\ell(v, w)$ is the minimal number of edges in any path from 
	$v$ to $w$. A path achieving this minimum is called a \emph{combinatorial 
		geodesic}. The \emph{level set} $S_\ell(v)$ is the union of all vertices at combinatorial distance $\leq \ell$ from $v$, together with all edges and triangles involving those vertices only.
	
	\begin{lemma}[Combinatorial Distance in Triangles]\label{lem:combdist}
		Let $Q$ be a CAT(0) planar graph. Fix a vertex $v \in V(Q)$.
		\begin{enumerate}
			\item In any triangle $(a,b,c)$, the three values 
			$\ell(a,v), \ell(b,v), \ell(c,v)$ equal $\{\ell, \ell, \ell-1\}$ or 
			$\{\ell, \ell-1, \ell-1\}$ for some $\ell \geq 1$. The level set $S_\ell=S_\ell(v)$ 
			is a connected planar graph together with all its interior faces, hence 
			contractible as a two-dimensional cell complex.
			
			\item If $(a,b)$ is an edge with $\ell(a,v) = \ell(b,v) = \ell$, then 
			there exists a unique triangle $(a,b,c)$ with $\ell(c,v) = \ell - 1$.
			
			\item Let $a \neq v$ with $\ell(a,v) = \ell$. Among the neighbors of $a$:
			\begin{itemize}
				\item Exactly one or two have combinatorial distance $\ell - 1$ to $v$.
				\item At most two have combinatorial distance $\ell$ to $v$.
				\item The rest have combinatorial distance $\ell + 1$ to $v$.
			\end{itemize}
			Furthermore, the neighbors at distance $\ell - 1$ or $\ell$ are cyclically 
			consecutive around $a$, with those at distance $\ell - 1$ in the ``middle''.
		\end{enumerate}
	\end{lemma}
	See \cref{fig:combindist} for an illustration of \cref{lem:combdist}.
	
	\begin{figure}[h!]
		\centering
		\includegraphics[width=0.5\linewidth]{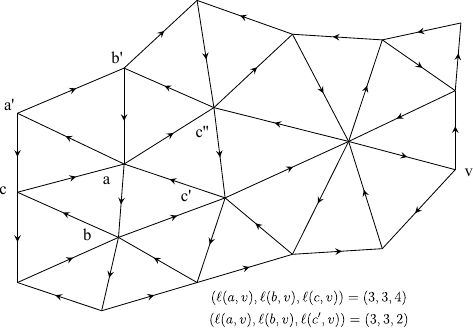}
		\caption{Illustrating \cref{lem:combdist}.  Claim (1): combinatorial distances from the vertices of a triangle are either $\{\ell,\ell,\ell+1\}$ or $\{\ell,\ell+1,\ell+1\}$. Claim (2): $(a,b)$ is an edge with $\ell(a,v) = \ell(b,v) = \ell$, and $(a,b,c')$ is the unique triangle with $\ell(c',v) = \ell-1$.  Claim (3): among the neighbors of $a$, (i) exactly the vertices $c'$ and $c''$ have combinatorial distance $\ell-1$ to $v$, (ii) exactly $b$ and $b'$ have combinatorial distance $\ell$ to $v$, and (iii) all others $a'$ and $c$ have combinatorial distance $\ell+1$ to $v$. }
		\label{fig:combindist}
	\end{figure}

	\begin{proof}
		We prove (1), (2), and (3) simultaneously by induction on $\ell$.
		
		\medskip
		\noindent\textbf{Base Case: $\ell \leq 1$.}
		
		For $\ell = 0$: The only vertex at distance $0$ from $v$ is $v$ itself. 
		The set $S_0 = \{v\}$ is trivially contractible.
		
		For $\ell = 1$: The vertices at distance $1$ from $v$ are precisely the 
		neighbors of $v$. Any triangle containing $v$ has vertex distances 
		$(0, 1, 1)$, fitting the pattern $\{\ell, \ell-1, \ell-1\}$ with $\ell = 1$. 
		The set $S_1$ consists of $v$, all neighbors of $v$, all edges incident 
		to $v$, and all triangles containing $v$. This is a star-like structure 
		with filled triangular faces around $v$, which is contractible. For 
		part (3): Each neighbor $a$ of $v$ has $\ell(a,v) = 1$. The unique 
		neighbor of $a$ at distance $0$ is $v$ itself. The other neighbors of 
		$a$ are either also neighbors of $v$ (distance $1$) or are at distance $2$.
		
		\medskip
		\noindent\textbf{Inductive Step.}
		
		Assume (1), (2), (3) hold for all distances up to $\ell$. We prove them 
		for $\ell + 1$.
		
		By the inductive hypothesis (1), $S_\ell$ is a connected planar graph (together with its interior faces).			
		The boundary $\partial S_\ell$ consists of vertices and edges that are 
		either on the boundary $\partial Q$ of the entire graph, or at distance 
		exactly $\ell$ from $v$. Let $\Gamma$ denote a connected component of 
		the part of $\partial S_\ell$ consisting of vertices and edges at 
		distance exactly $\ell$ from $v$. Each such $\Gamma$ is one of: a closed 
		cycle, a path with both endpoints on $\partial Q$, or a single vertex 
		(on $\partial Q$).
		
		
		Let $a$ be an interior vertex with $\ell(a, v) = \ell$. By the inductive 
		hypothesis (3), $a$ has $1$ or $2$ neighbors at distance $\ell - 1$, at 
		most $2$ neighbors at distance $\ell$ (call them $a'$ and $a''$), and 
		the remaining neighbors have distance $\ell + 1$. Since $a$ is interior, 
		it has degree $\geq 6$. Therefore:
		\[
		\text{(neighbors at distance } \ell+1) \geq 6 - 2 - 2 = 2.
		\]
		
		
		
		Let $b$ be a vertex with $\ell(b, v) = \ell + 1$. We establish three claims:
		
		\medskip
		\noindent\emph{Claim (a):} If $a, a'$ are distinct neighbors of $b$ with 
		$\ell(a,v) = \ell(a',v) = \ell$, then $(a, a', b)$ is a triangle. 
		Consequently, $b$ has at most two neighbors at distance $\ell$, and if 
		two, they are adjacent.
		
		\begin{proof}[Proof of Claim (a)]
			Since both $a$ and $a'$ are at distance $\ell$ and adjacent to $b$ (at 
			distance $\ell + 1$), they must lie on the same boundary component 
			$\Gamma$ of $S_\ell$.
			
			Suppose for contradiction that $a$ and $a'$ are not adjacent. Let 
			$\gamma$ be the path along $\Gamma$ from $a$ to $a'$. Let $R$ be the 
			region bounded by the path $\gamma$ and the edges $(a, b)$ and $(b, a')$.
			
			Since $a$ and $a'$ are not adjacent along $\gamma$, the region $R$ 
			contains at least one vertex of $\gamma$ other than $a, a'$.
			
			Consider any vertex $w$ on $\gamma$ strictly between $a$ and $a'$. This 
			vertex $w$ is at distance $\ell$ from $v$, so by the inductive 
			hypothesis (3), $w$ has exactly one or two neighbors at distance 
			$\ell-1$ and at most two at distance $\ell$. The neighbors at distance 
			$\ell+1$ lie on the exterior of $S_\ell$. The vertex $w$ 
			has at least two neighbors inside $R$, and two neighbors on $\gamma$,
			so $\deg_R(w) \geq 4$. 
			
			Apply Lemma~\ref{lem:graph} to $R$.  Boundary vertices $a$, $a'$, and $b$ each have $\deg \geq 2$, and the rest of the boundary vertices have $\deg \geq 4$.  We conclude that $\deg_R(a) = \deg_R(a') = \deg_R(b) = 2$ and it follows from this that $\gamma$ does not contain intermediate vertices. Therefore $(a, a', b)$ must be a triangle.
		\end{proof}
		
		\medskip
		\noindent\emph{Claim (b):} If $(b, b')$ is an edge with 
		$\ell(b, v) = \ell(b', v) = \ell + 1$, then there exists a unique triangle 
		$(a, b, b')$ with $\ell(a, v) = \ell$.
		
		\begin{proof}[Proof of Claim (b)]
			Let $a \in \partial S_\ell$ (resp. $a'$) be adjacent to $b$ (resp. $b'$) satisfying $
			\ell(a,v) = \ell(a',v) = \ell$.  If $\gamma$ is a path in $\partial S_\ell$ connecting $a$ and $a'$, then all vertices on $\gamma$ are at distance $\ell$ from $v$.  Using the same argument as in the proof of (a), we conclude that $a = a'$.
		\end{proof}
		
		\medskip
		\noindent\emph{Claim (c):} $b$ has at most two neighbors at distance 
		$\ell + 1$; if exactly two, they are not adjacent.
		
		\begin{proof}[Proof of Claim (c)]
			This follows from the cyclic structure around $b$ and Claim (a): there are $1$ or $2$ 
			neighbors at distance $\ell$, and if $2$ 
			they are adjacent.
		\end{proof}

		From Claims (a), (b), (c), we deduce (1), (2), (3) for $\ell + 1$:
		
		\begin{itemize}
			\item (1) for $\ell + 1$: Any triangle in $S_{\ell + 1}$ involving a vertex at 
			distance $\ell + 1$ has distances $\{\ell+1, \ell+1, \ell\}$ or 
			$\{\ell+1, \ell, \ell\}$ by Claim (a) and the structure of how 
			distance $(\ell+1)$ vertices attach to $S_\ell$. The set $S_{\ell+1}$ 
			remains connected with all interior faces filled.
			
			\item (2) for $\ell + 1$: Follows from Claim (b).
			
			\item (3) for $\ell + 1$: Follows from Claims (a) and (c), 
			plus the cyclic ordering from the planar embedding.
		\end{itemize}
		
		This completes the induction.
	\end{proof}

	\section{Directed Distance Geometry}\label{sec:directed-distance}
	
	The \emph{directed distance} $\delta(v, w)$ is the length of the shortest directed path from $v$ to $w$.  Recall our convention that traversing an edge against the orientation contributes 2 to the distance.  		
	
	This section establishes 
	how the function $\dist(\cdot, v)$ behaves locally.  The key results are~\cref{lem:drop} stating that for any vertex $u \neq v$, there 
	is a \emph{unique} neighbor $u'$ of $u$ minimizing $\delta(u', v)$, and \cref{cor:localdistance} which describes the behavior of $\dist(\cdot,v)$ when restricted to the neighbors of some vertex $u$. 
	
	Note that if $Q$ is a CAT(0) planar graph then so is the graph $Q^{\rm op}$ obtained from $Q$ by reversing all the orientations.  We have $\dist_Q(v,w) = \dist_{Q^{\rm op}}(w,v)$.  While many of the results are stated for the function $\dist(\cdot, v)$, they also hold for $\dist(v, \cdot)$ after reversing orientations. 
	
	\subsection{Local Distance Lemmas}\label{subsec:local-distance}
	
	\begin{lemma}\label{lem:distmod}
		Let $p, q$ be two walks from $v$ to $w$ (possibly traveling against orientations).  Then the directed distance along $p$ and the directed distance along $q$ are congruent modulo $3$.
	\end{lemma}
	\begin{proof}
		The two walks $p$ and $q$ can be related by a finite sequence of local moves: (a) use one edge of a triangle instead of the other two, and (b) replace traversing an edge in both directions by not moving.  Both operations preserve directed distance modulo $3$.
	\end{proof}
	
	The following lemmas describe how directed distance changes across 
	single edges and within triangles.
	
	\begin{lemma}\label{lem:edgedist1}
		Let $Q$ be a CAT(0) planar graph. Fix a target vertex $v \in V(Q)$. 
		Suppose we have a directed edge $u \to w$. Then $\dist(u,v) \in \{\dist(w,v)+1,\dist(w,v)-2\}$.
	\end{lemma}
	\begin{proof}
		It is clear that $\dist(w,v) -2 \leq \dist(u,v) \leq \dist(w,v)+1$.  But by \cref{lem:distmod}, $\dist(u,v) \equiv \dist(w,v)+1$ modulo $3$, so we have $\dist(u,v) \in \{\dist(w,v)+1,\dist(w,v)-2\}$.
	\end{proof}
	
	\begin{lemma}\label{lem:triangledist}
		Let $Q$ be a CAT(0) planar graph. Fix a target vertex $v \in V(Q)$. 
		For any triangle with vertices $u, w, x$, the three values 
		$\delta(u, v), \delta(w, v), \delta(x, v)$ equal $d, d+1, d+2$ for 
		some integer $d$ (in some order).
	\end{lemma}
	\begin{proof}
		Apply \cref{lem:edgedist1}.
	\end{proof}

	\begin{lemma}\label{lem:paths}
		Let $u$ and $v$ be distinct vertices and let $p,q$ be two paths from $u$ to $v$ with minimal directed distance, where the first step $u \to u'$ is outgoing.  Then either
		\begin{enumerate}
			\item both paths use the same initial edge $u \to u_1$, or
			\item the paths use initial edges $u \to u_1$ and $u \to u'_1$, and in cyclic order around $u$ these two edges are separated by a single incoming edge $w \to u$ such that both triangles $u\to u_1 \to w \to u$ and $u\to u'_1 \to w \to u$ are present in $Q$.  Furthermore, $\dist(w,v) = \dist(u,v) - 2$.
		\end{enumerate}
	\end{lemma}
	
	\begin{proof}
		By adding some triangles to $Q$ to avoid cut edges, we may assume that $p,q$ are paths following orientations of edges.   We proceed by induction on $\dist(u,v)$.  If $\dist(u,v) = 1$, the claim is clear.  Now, suppose $\dist(u,v) > 1$, and suppose that $p$, $q$ are two shortest directed paths 
		$$
		p = (u = u_0 \to u_1 \to \cdots \to u_d = v), \qquad q = (u = u'_0 \to u'_1 \to \cdots \to u'_d = v)
		$$
		such that $u_1 \neq u'_1$.  If $p$ and $q$ intersect before reaching $v$, then the result holds by induction (replacing $v$ by the point of intersection), so we may assume that $p$ and $q$ together bound a simple polygonal region $R$.  Let $Q(R)$ be our graph restricted to $R$.  Then $Q(R)$ is a connected planar graph.
		
		Let us consider one of the vertices $u_i$ on the path $p$ that is neither the start nor the end.  Since we have $u_{i-1} \to u_i \to u_{i+1}$, either $\deg_{Q(R)}(u_i) = 2$ pointing into $R$, or $\deg_{Q(R)}(u_i) \geq 4$.  It follows from \cref{lem:graph} that at least one internal vertex of $p$ or $q$, say $u_i$ on $p$, satisfies $\deg_{Q(R)}(u_i) = 2$.  In other words, we have a triangle $u_{i-1} \to u_i \to u_{i+1} \to u_{i-1}$ in $R$.  The edge $u_{i+1} \to u_{i-1}$ cannot be on $q$, so there is another triangle $u_{i-1} \to w \to u_{i+1} \to u_{i-1}$ on the other side of the edge $u_{i+1} \to u_{i-1}$, and this triangle also belongs to $R$.  We obtain another shortest path from $u$ to $v$ by using $w$ instead of $u_{i}$.  This modification strictly reduces the area of the region $R$.
		
		Repeating this argument, eventually one of the vertices $u_1$ or $u'_1$ will be the vertex to be modified.  Without loss of generality, we assume it to be $u_1$.  If the modification changes $u \to u_1$ to $u \to u'_1$, then we are in situation (2).  We assume otherwise to obtain a contradiction.
		
		Let the vertices adjacent to $u$ belonging to $R$ be $u_1, x_1, y_1,\ldots,y_{s-1},x_s,u'_1$ in order, and assume that $s \geq 1$.  Then the path $p$ must start $u \to u_1 \to x_1 \to \cdots$.  It follows that $\dist(x_1,v) = \dist(u,v) - 2$ and $\dist(y_1,v) = \dist(u,v) -1$.  Thus there is also a shortest path from $u$ to $v$ passing through $y_1$.  Repeating this argument, we see that there are shortest paths passing through all of $u_1,y_1,y_2,\ldots,y_{s-1},u'_1$, and $\dist(x_i,v) = d(u,v) -2$ for $i =1,2,\ldots,s$.
		
		We now apply the inductive hypothesis to $y_1$.  There are shortest paths from $y_1$ to $v$ using both the edges $y_1 \to x_1$ and $y_1 \to x_2$, and we conclude that there exists a vertex $z$, not equal to $u$, such that we have two triangles $y_1 \to x_1 \to z \to y_1$ and $y_1 \to x_2 \to z \to y_1$.  This contradicts that $y_1$, if interior, should have degree $\geq 6$.
	\end{proof}

	\subsection{Unique Minimizing Neighbors}\label{subsec:unique-minimizer}
	
	The following lemma establishes that for any vertex $u \neq v$, there 
	is a unique neighbor minimizing directed distance to $v$. 
	
	\begin{lemma}\label{lem:drop}
		Let $Q$ be a CAT(0) planar graph. Fix a target vertex $v \in V(Q)$. 
		For any vertex $u \neq v$, there exists a unique neighbor $u'$ of $u$ 
		that minimizes $\delta(u', v)$ among all neighbors of $u$. Moreover, 
		$\delta(u, v) - \delta(u', v) \in \{1, 2\}$.
	\end{lemma}
	
	\begin{proof}
		By adding triangles, we can assume that $Q$ has no cut edges without changing the validity of the statement.  Then there is always at least one path $u \to u_1 \to \cdots \to v$, following orientations of edges, of minimal directed distance from $u$ to $v$.  Apply \cref{lem:paths}.  If \cref{lem:paths}(1) holds then $u_1$ is the unique neighbor of $u$ that minimizes $\delta(u',v)$, and we have $\dist(u,v) - \dist(u_1,v) = 1$.  If \cref{lem:paths}(2) holds then $w$ is the unique neighbor of $u$ that minimizes $\dist(u',v)$, and we have $\dist(u,v) - \dist(w,v) = 2$.		
	\end{proof}

	\begin{corollary}\label{cor:localdistance}
		Fix $v$ and consider the distance function $\delta(\cdot, v)$ directed towards $v$.  Let $u$ be an interior vertex such that $\dist(u,v) = d > 0$, and let the vertices incident to $u$ be $x_1,y_1,x_2,y_2,\ldots$ where there are edges $x_i \to u$ and $u \to y_i$.  Then the function $\dist(\cdot,v)$ restricted to $\ldots,x_i,y_i,x_{i+1},y_{i+1},\ldots$ is, up to cyclic rotation, given by
		$$
		\ldots, d+2,d+1,d-1,d+1,d+2,d+1,d+2, \ldots,
		$$
		or
		$$
		\ldots, d+2,d+1,d-1,d-2,d-1,d+1,d+2,d+1,d+2,\ldots,
		$$
		corresponding to the two cases of \cref{lem:drop}.
		If instead, $u$ is a boundary vertex, then these distance values could be truncated, however the value $d-1$ is always present in the first case, and the value $d-2$ is always present in the second case.
	\end{corollary}
	\begin{proof}
		The two cases correspond to (1) and (2) of \cref{lem:paths} respectively.  The pattern of distances follows from \cref{lem:drop} and \cref{lem:triangledist}.
	\end{proof}

	\subsection{Combinatorial geodesics achieve minimum directed distance}\label{subsec:geodesics-shortest}

	We establish that combinatorial geodesics are also paths of 
	shortest directed distance.
	
	\begin{proposition}\label{prop:short}
		Let $Q$ be a CAT(0) planar graph. Any combinatorial geodesic is a path 
		of shortest directed distance between its endpoints. That is, if 
		$\gamma$ is a path from $u$ to $v$ achieving the minimum number of 
		edges $\ell(u,v)$, then $\gamma$ also achieves the minimum directed 
		distance $\delta(u,v)$.
	\end{proposition}
	
	\begin{proof}
		We proceed by induction on the length $\ell$ of the combinatorial 
		geodesic $\gamma$.  The base cases $\ell \leq 1$ are clear.

		Assume the statement holds for all combinatorial geodesics of length at 
		most $\ell - 1$, where $\ell \geq 2$. Let 
		$\gamma = (u = u_0, u_1, \ldots, u_\ell = v)$ be a combinatorial geodesic 
		of length $\ell$ from $u$ to $v$.	  It suffices to show that 
		\begin{equation}\label{eq:deltaadd}
			\delta(u, v) = \delta(u, u_1) + \delta(u_1, v).
		\end{equation}

		\medskip
		\noindent{Case 1:} Suppose that every path from $u$ to $v$ passes 
		through $u_1$.
		
		In this case, $u_1$ is a cut vertex separating $u$ from $v$. Any path 
		from $u$ to $v$ decomposes as a path from $u$ to $u_1$ followed by a 
		path from $u_1$ to $v$. Therefore, \eqref{eq:deltaadd} holds.
		
		\medskip
		\noindent{Case 2:} There exists a path from $u$ to $u_1$ that does 
		not pass through $u_1$.  There must a triangle $(u, u_1, u')$ for 
		some vertex $u'$ adjacent to both $u$ and $u_1$.  Let $d:= \dist(u_1,v)$.
		
		\medskip
		\noindent{Subcase 2a:} The triangle is oriented 
		$u \to u' \to u_1 \to u$.  If \eqref{eq:deltaadd} does not hold, then by \cref{lem:edgedist1}, we have $\dist(u,v) = d- 1$ and $\dist(u',v)= d -2$.  We have $\ell(u',v) \in \{\ell-1, \ell\}$.  Either way, applying the inductive hypothesis and \cref{cor:localdistance} to $u_1$ we deduce that the unique neighbor of $u_1$ that minimizes $\dist(\cdot, v)$ should have combinatorial distance $\ell-2$ from $v$.  This contradicts $\dist(u',v)= d -2$.
		
		\medskip
		\noindent{Subcase 2b:} The triangle is oriented 
		$u \to u_1 \to u' \to u$.  If \eqref{eq:deltaadd} does not hold, then by \cref{lem:edgedist1}, we have $\dist(u,v) = d- 2$ and $\dist(u',v)= d -1$.  Applying the inductive hypothesis and \cref{cor:localdistance} to $u_1$ we deduce that the unique neighbor of $u_1$ that minimizes $\dist(\cdot, v)$ should have combinatorial distance $\ell-2$ from $v$.  This contradicts $\dist(u,v)= d -2$.
		
		This proves \eqref{eq:deltaadd}, the inductive step, and the claim.
	\end{proof}
	The following result follows immediately.
	
	\begin{corollary}\label{cor:edge-count}
		Let $\gamma$ be a combinatorial geodesic from $u$ to $v$. The number 
		of edges traversed in the forward direction, and the number traversed 
		in the backward direction, are invariants of the endpoints $(u, v)$.
	\end{corollary}

	We may now sharpen \cref{lem:edgedist1}.
	\begin{lemma}\label{lem:edgedist}
		Let $Q$ be a CAT(0) planar graph. Fix a target vertex $v \in V(Q)$. 
		Suppose we have a directed edge $u \to w$. Then 
		\[
		\delta(u, v) - \delta(w, v) = \begin{cases}
			+1 & \text{if } \ell(u, v) \geq \ell(w, v), \\[4pt]
			-2 & \text{if } \ell(u, v) = \ell(w, v) - 1.
		\end{cases}
		\]
	\end{lemma}
	\begin{proof}
		The cases $\ell(u,v) = \ell(w,v) \pm 1$ follow immediately from \cref{lem:edgedist1} and \cref{prop:short}.  When $\ell(u,v) = \ell(w,v)$, by \cref{lem:combdist}, there exists a triangle $u \to w \to x \to u$ such that $\ell(x,v) = \ell(u,v) -1$.  We have that $\dist(x,v) = \dist(w,v) -1$ so by \cref{lem:triangledist} and \cref{prop:short} we have $\dist(u,v) = \delta(w,v)+1$.
	\end{proof}

	\section{Focal Points and Distance Minimizers}\label{sec:focal-points}
	
	Let $(Q,\z)$ be a labeled CAT(0) planar graph.  Fix three not necessarily distinct vertices $z_i, z_j, z_k$ in $\z$.		
	The \emph{distance sum} from a vertex $x$ to this triple is
	\begin{equation}\label{eq:Dx}
		D(x) := \delta(x, z_i) + \delta(x, z_j) + \delta(x, z_k),
	\end{equation}
	and the \emph{Fermat-Le distance sum} (\cref{defn:FL}) is
	\[
	\Le_Q(z_i,z_j,z_k) = \Sigma_Q(i, j, k) := \min_{x \in V(Q)} D(x).
	\]
	A vertex achieving this minimum is called a \emph{distance minimizer} 
	for $(z_i, z_j, z_k)$ (or for $(i,j,k)$).  We may omit the subscript $Q$ and write $\Le(i,j,k)$ for brevity.
	
	Our goal in this section is to characterize distance minimizers in terms 
	of \emph{focal points} (\cref{def:focal}). The main idea of a focal point is that the three paths
	from a minimizer to the given triple of boundary vertices can be chosen to
	be sufficiently spread out and non-intersecting.  This choice is made by the
	construction of the three canonical geodesics.  The main result is:

	\begin{theorem}[Minimizers and Focal Points Coincide]\label{thm:focal}
		Let $Q$ be a CAT(0) planar graph and let $(z_i, z_j, z_k)$ 
		be three boundary vertices. A vertex $v$ is a distance minimizer for 
		$(z_i, z_j, z_k)$ if and only if $v$ is a focal point for 
		$(z_i, z_j, z_k)$.
	\end{theorem}
	
	\cref{thm:focal} is established in \cref{subsec:minimizers-focal} and \cref{subsec:focalproof}.
	
	\subsection{Definition of focal point}\label{subsec:focal-definitions}
	
	We introduce the neighbor sets that encode how geodesics from $v$ to 
	different boundary targets depart from $v$.
	
	\begin{definition}[Neighbor Sets $E(v)$ and $F(v)$]
		For a vertex $v \in V(Q)$:
		\begin{itemize}
			\item $E(v)$ denotes the set of neighbors of $v$ connected by an 
			\emph{outgoing} edge from $v$.
			\item $F(v)$ denotes the set of neighbors of $v$ connected by an 
			\emph{incoming} edge to $v$.
		\end{itemize}
		For technical reasons, if $v$ is a boundary vertex and there is a 
		boundary edge $v \to u$, we augment $F(v)$ by adding vertices $u^-$ and/or 
		$u^+$ (if necessary) so that $u^-, u, u^+$ are in cyclic order around $v$.  That is, we pretend that there are fake edges $u^- \to v$ and $u^+ \to v$.
		If $|F(v)| < 3$, we add auxiliary elements $\tilde v$ (and $\tilde v'$ if necessary) so that $|F(v)| = 3$.
	\end{definition}
	
	Ignoring the auxiliary elements $\tilde v$, $\tilde v'$, the set $F(v)$ comes equipped with a cyclic ordering, see \cref{fig:fcyclic}.

	\begin{figure}[h!]
		\centering
		\includegraphics[width=0.15\linewidth]{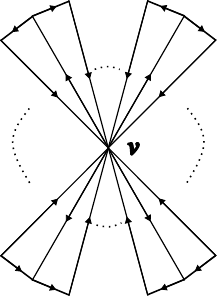}
		\caption{Cyclic ordering on neighbors of $v$; in particular on the set $F(v)$ with an edge directed to $v$.}
		\label{fig:fcyclic}
	\end{figure}

	\begin{definition}[Minimizing Neighbor $f_z(v)$]
		For vertices $v \neq z$, define $f_z(v)$ to be the unique neighbor of 
		$v$ (guaranteed by \cref{lem:drop}) that minimizes $\delta(\cdot, z)$ among all neighbors of $v$.
	\end{definition}
	
	\begin{definition}\label{def:F-set}
		For vertices $v, z \in V(Q)$, define:
		\[
		F(v, z) = \begin{cases}
			F(v) & \text{if } v = z, \\[4pt]
			\{a, a'\} & \text{if } v \neq z, \text{ there is an edge } v \to f_z(v), \\
			& \quad \text{and } a, f_z(v), a' \text{ are consecutive in cyclic order around } v, \\[4pt]
			\{f_z(v)\} & \text{if } v \neq z \text{ and there is an edge } f_z(v) \to v.
		\end{cases}
		\]
		When $F(v, z) = \{a, a'\}$, we write $c \triangleleft F(v, z)$ if 
		$a, c, a'$ are in cyclic order around $v$ with $c$ strictly between 
		$a$ and $a'$.
	\end{definition}
	
	\begin{definition}[Parallel Vertices]
		We say that vertices $z$ and $z'$ are \defn{parallel at $v$}, written 
		$z \parallel_v z'$, if $F(v, z) = F(v, z')$ and one of the following holds:
		\begin{enumerate}
			\item $|F(v, z)| = 1$, or
			\item $|F(v, z)| = 2$ and $z, z'$ are parallel at $f_z(v) = f_{z'}(v)$, or
			\item $z = z' \neq v$.
		\end{enumerate}
	\end{definition}
	
	\begin{definition}[Witness Path]
		When $F = \{a, a'\} = F(v, z) = F(v, z')$, a 
		\defn{witness path} is a sequence $v = v_0 \to v_1 \to \cdots \to v_r = w$ 
		such that $|F(v_i, z)| = |F(v_i, z')| = 2$ for $i \in [0, r-1]$ and 
		$|F(v_r, z)| = |F(v_r, z')| = 1$. The vertex $w = v_r$ is called the 
		\textbf{witness}.  We have $z \parallel_v z'$ if and only if $z \parallel_w z'$ if and only if $F(w,z) = F(w,z')$.
	\end{definition}
	
	\begin{definition}[Focal Point]\label{def:focal}
		A vertex $v$ is a \defn{focal point} for $(z_i, z_j, z_k)$ if:
		\begin{enumerate}
			\item The sets $F(v, z_i), F(v, z_j), F(v, z_k)$ admit a system of 
			distinct representatives (SDR).
			\item No two of $z_i, z_j, z_k$ are parallel at $v$.
		\end{enumerate}
	\end{definition}
	
	\begin{example}
		In \cref{fig:v124}, $v$ is not a focal point for $(z_1,z_2,z_4)$ but in \cref{fig:x124}, $x$ is a focal point for $(z_1,z_2,z_4)$.  In \cref{fig:examplefocalpoint}, the vertex $v$ is a focal point for any of $(z_2,z_8,z_{13}), (z_2,z_8,z_{17}), (z_2,z_{13},z_{17}),(z_8,z_{13},z_{17})$.
	\end{example}

	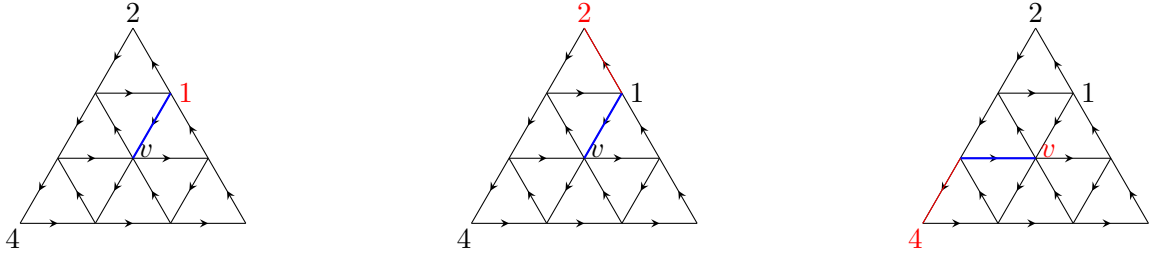
\begin{figure}
	\begin{center}
	$$
	\begin{tikzpicture}
	\tikzset{->-/.style={decoration={markings, mark=at position #1 with {\arrow{stealth}}},postaction={decorate}}, ->-/.default=0.5}
	\coordinate (z4) at (0,0);
	\coordinate (y) at (1,0);
	\coordinate (z5) at (2,0);
	\coordinate (z6) at (3,0);
	\coordinate (z3) at (60:1);
	\coordinate (v) at ($(z3) + (1,0)$);
	\coordinate (u) at ($(v) + (1,0)$);
	\coordinate (x) at (60:2);
	\coordinate (z1) at ($(x) + (1,0)$);
	\coordinate (z2) at (60:3);
	\draw[->-] (z1) -- (z2);
	\draw[->-] (z2) -- (x);
	\draw[->-] (x) -- (z3);
	\draw[->-] (z3) -- (z4);
	\draw[->-] (z4) -- (y);
	\draw[->-] (y) -- (z5);
	\draw[->-] (z5) -- (z6);
	\draw[->-] (z6) -- (u);
	\draw[->-] (u) -- (z1);
	\draw[->-] (x) -- (z1);
	\draw[->-] (z1) -- (v);
	\draw[->-] (v) -- (x);
	\draw[->-] (v) -- (u);
	\draw[->-] (v) -- (y);
	\draw[->-] (z3) -- (v);
	\draw[->-] (y) -- (z3);
	\draw[->-] (z5) -- (v);
	\draw[->-] (u) -- (z5);
	\node[red] at ($(z1)+(0.2,0)$) {\small $1$};
	\node at ($(z2)+(0,0.2)$) {\small $2$};
	\node at ($(z4)+(-0.1,-0.2)$) {\small $4$};
	\node at ($(v)+(0.17,0.1)$) {\small  $v$};
	\draw[blue, thick] (v) -- (z1);
	
	\begin{scope}[shift={(6,0)}]
	\coordinate (z4) at (0,0);
	\coordinate (y) at (1,0);
	\coordinate (z5) at (2,0);
	\coordinate (z6) at (3,0);
	\coordinate (z3) at (60:1);
	\coordinate (v) at ($(z3) + (1,0)$);
	\coordinate (u) at ($(v) + (1,0)$);
	\coordinate (x) at (60:2);
	\coordinate (z1) at ($(x) + (1,0)$);
	\coordinate (z2) at (60:3);
	\draw[->-] (z1) -- (z2);
	\draw[->-] (z2) -- (x);
	\draw[->-] (x) -- (z3);
	\draw[->-] (z3) -- (z4);
	\draw[->-] (z4) -- (y);
	\draw[->-] (y) -- (z5);
	\draw[->-] (z5) -- (z6);
	\draw[->-] (z6) -- (u);
	\draw[->-] (u) -- (z1);
	\draw[->-] (x) -- (z1);
	\draw[->-] (z1) -- (v);
	\draw[->-] (v) -- (x);
	\draw[->-] (v) -- (u);
	\draw[->-] (v) -- (y);
	\draw[->-] (z3) -- (v);
	\draw[->-] (y) -- (z3);
	\draw[->-] (z5) -- (v);
	\draw[->-] (u) -- (z5);
	\node at ($(z1)+(0.2,0)$) {\small $1$};
	\node[red] at ($(z2)+(0,0.2)$) {\small $2$};
	\node at ($(z4)+(-0.1,-0.2)$) {\small $4$};
	\node at ($(v)+(0.17,0.1)$) {\small  $v$};
	\draw[blue, thick] (v) -- (z1);
	\draw[red] (z1)--(z2);
	\end{scope}
	
		\begin{scope}[shift={(12,0)}]
	\coordinate (z4) at (0,0);
	\coordinate (y) at (1,0);
	\coordinate (z5) at (2,0);
	\coordinate (z6) at (3,0);
	\coordinate (z3) at (60:1);
	\coordinate (v) at ($(z3) + (1,0)$);
	\coordinate (u) at ($(v) + (1,0)$);
	\coordinate (x) at (60:2);
	\coordinate (z1) at ($(x) + (1,0)$);
	\coordinate (z2) at (60:3);
	\draw[->-] (z1) -- (z2);
	\draw[->-] (z2) -- (x);
	\draw[->-] (x) -- (z3);
	\draw[->-] (z3) -- (z4);
	\draw[->-] (z4) -- (y);
	\draw[->-] (y) -- (z5);
	\draw[->-] (z5) -- (z6);
	\draw[->-] (z6) -- (u);
	\draw[->-] (u) -- (z1);
	\draw[->-] (x) -- (z1);
	\draw[->-] (z1) -- (v);
	\draw[->-] (v) -- (x);
	\draw[->-] (v) -- (u);
	\draw[->-] (v) -- (y);
	\draw[->-] (z3) -- (v);
	\draw[->-] (y) -- (z3);
	\draw[->-] (z5) -- (v);
	\draw[->-] (u) -- (z5);
	\node at ($(z1)+(0.2,0)$) {\small $1$};
	\node at ($(z2)+(0,0.2)$) {\small $2$};
	\node[red] at ($(z4)+(-0.1,-0.2)$) {\small $4$};
	\node[red] at ($(v)+(0.17,0.1)$) {\small  $v$};
	\draw[blue, thick] (v) -- (z3);
	\draw[red] (z3)--(z4);
	\end{scope}
\end{tikzpicture}
$$
\end{center}
\caption{The canonical geodesics from $v$ to $z_1$, $z_2$, $z_4$ respectively have been indicated; the sets $F(v,z_1), F(v,z_2), F(v,z_4)$ have been indicated in blue.  Since $z_1$ and $z_2$ are parallel at $v$, we have that $v$ is not a focal point.}
\label{fig:v124}
\end{figure}	

	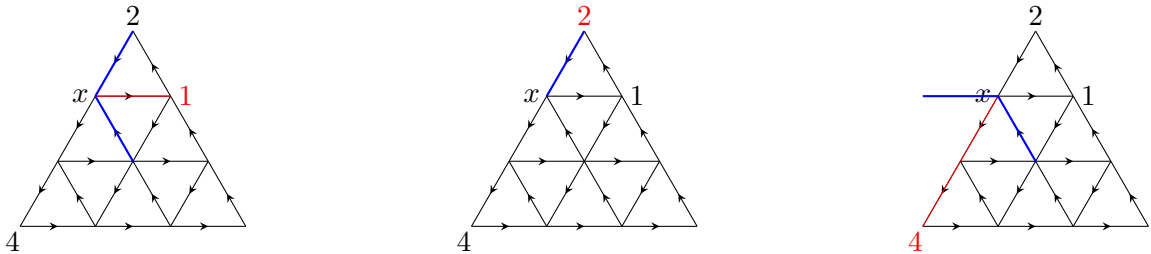
\begin{figure}
	\begin{center}
	$$
	\begin{tikzpicture}
	\tikzset{->-/.style={decoration={markings, mark=at position #1 with {\arrow{stealth}}},postaction={decorate}}, ->-/.default=0.5}
	\coordinate (z4) at (0,0);
	\coordinate (y) at (1,0);
	\coordinate (z5) at (2,0);
	\coordinate (z6) at (3,0);
	\coordinate (z3) at (60:1);
	\coordinate (v) at ($(z3) + (1,0)$);
	\coordinate (u) at ($(v) + (1,0)$);
	\coordinate (x) at (60:2);
	\coordinate (z1) at ($(x) + (1,0)$);
	\coordinate (z2) at (60:3);
	\draw[->-] (z1) -- (z2);
	\draw[->-] (z2) -- (x);
	\draw[->-] (x) -- (z3);
	\draw[->-] (z3) -- (z4);
	\draw[->-] (z4) -- (y);
	\draw[->-] (y) -- (z5);
	\draw[->-] (z5) -- (z6);
	\draw[->-] (z6) -- (u);
	\draw[->-] (u) -- (z1);
	\draw[->-] (x) -- (z1);
	\draw[->-] (z1) -- (v);
	\draw[->-] (v) -- (x);
	\draw[->-] (v) -- (u);
	\draw[->-] (v) -- (y);
	\draw[->-] (z3) -- (v);
	\draw[->-] (y) -- (z3);
	\draw[->-] (z5) -- (v);
	\draw[->-] (u) -- (z5);
	\node[red] at ($(z1)+(0.2,0)$) {\small $1$};
	\node at ($(z2)+(0,0.2)$) {\small $2$};
	\node at ($(z4)+(-0.1,-0.2)$) {\small $4$};
	\node at ($(x)+(-0.2,-0)$) {\small  $x$};
	\draw[red] (x) -- (z1);
	\draw[blue, thick] (x) -- (z2);
	\draw[blue, thick] (x) -- (v);
	
	\begin{scope}[shift={(6,0)}]
	\coordinate (z4) at (0,0);
	\coordinate (y) at (1,0);
	\coordinate (z5) at (2,0);
	\coordinate (z6) at (3,0);
	\coordinate (z3) at (60:1);
	\coordinate (v) at ($(z3) + (1,0)$);
	\coordinate (u) at ($(v) + (1,0)$);
	\coordinate (x) at (60:2);
	\coordinate (z1) at ($(x) + (1,0)$);
	\coordinate (z2) at (60:3);
	\draw[->-] (z1) -- (z2);
	\draw[->-] (z2) -- (x);
	\draw[->-] (x) -- (z3);
	\draw[->-] (z3) -- (z4);
	\draw[->-] (z4) -- (y);
	\draw[->-] (y) -- (z5);
	\draw[->-] (z5) -- (z6);
	\draw[->-] (z6) -- (u);
	\draw[->-] (u) -- (z1);
	\draw[->-] (x) -- (z1);
	\draw[->-] (z1) -- (v);
	\draw[->-] (v) -- (x);
	\draw[->-] (v) -- (u);
	\draw[->-] (v) -- (y);
	\draw[->-] (z3) -- (v);
	\draw[->-] (y) -- (z3);
	\draw[->-] (z5) -- (v);
	\draw[->-] (u) -- (z5);
	\node at ($(z1)+(0.2,0)$) {\small $1$};
	\node[red] at ($(z2)+(0,0.2)$) {\small $2$};
	\node at ($(z4)+(-0.1,-0.2)$) {\small $4$};
	\node at ($(x)+(-0.2,-0)$) {\small  $x$};
	\draw[blue, thick] (x) -- (z2);
	\end{scope}
	
		\begin{scope}[shift={(12,0)}]
	\coordinate (z4) at (0,0);
	\coordinate (y) at (1,0);
	\coordinate (z5) at (2,0);
	\coordinate (z6) at (3,0);
	\coordinate (z3) at (60:1);
	\coordinate (v) at ($(z3) + (1,0)$);
	\coordinate (u) at ($(v) + (1,0)$);
	\coordinate (x) at (60:2);
	\coordinate (z1) at ($(x) + (1,0)$);
	\coordinate (z2) at (60:3);
	\draw[->-] (z1) -- (z2);
	\draw[->-] (z2) -- (x);
	\draw[->-] (x) -- (z3);
	\draw[->-] (z3) -- (z4);
	\draw[->-] (z4) -- (y);
	\draw[->-] (y) -- (z5);
	\draw[->-] (z5) -- (z6);
	\draw[->-] (z6) -- (u);
	\draw[->-] (u) -- (z1);
	\draw[->-] (x) -- (z1);
	\draw[->-] (z1) -- (v);
	\draw[->-] (v) -- (x);
	\draw[->-] (v) -- (u);
	\draw[->-] (v) -- (y);
	\draw[->-] (z3) -- (v);
	\draw[->-] (y) -- (z3);
	\draw[->-] (z5) -- (v);
	\draw[->-] (u) -- (z5);
	\node at ($(z1)+(0.2,0)$) {\small $1$};
	\node at ($(z2)+(0,0.2)$) {\small $2$};
	\node[red] at ($(z4)+(-0.1,-0.2)$) {\small $4$};
	\node at ($(x)+(-0.2,-0)$) {\small  $x$};
	\draw[red] (x) -- (z3);
	\draw[red] (z3)--(z4);
	\draw[blue, thick] (x) -- (v);
	\draw[blue,thick] (x) -- ($(x)+(-180:1)$);
	\end{scope}
\end{tikzpicture}
$$
\end{center}
\caption{The three sets $F(x,z_1), F(x,z_2), F(x,z_4)$ have been indicated in blue.  Since none of them coincide, and they have an SDR, $x$ is a focal point for $(z_1,z_2,z_4)$.}
\label{fig:x124}
\end{figure}	

\begin{figure}[h!]
	\centering
	\includegraphics[width=0.5\linewidth]{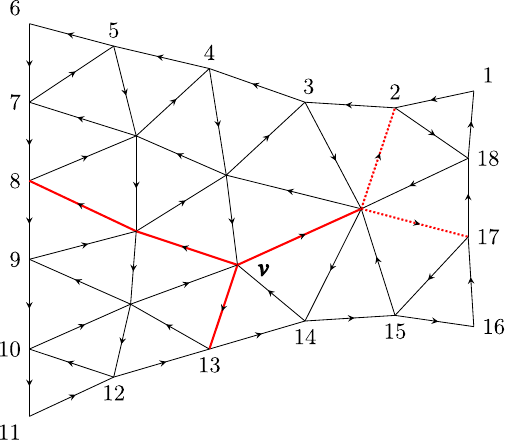}
	\caption{The red edges are the canonical geodesics from $v$ to $z_2,z_8, z_{13},z_{17}$.  The vertex $v$ is a focal point for $(z_2,z_8,z_{13}), (z_2,z_8,z_{17}), (z_2,z_{13},z_{17}),(z_8,z_{13},z_{17})$.  Note that $F(v,z_{2}) = F(v,z_{17})$ but $z_2$ and $z_{17}$ are not parallel at $v$.  All of $\{2,8,13\},\{2,8,17\}, \{2,13,17\},\{8,13,17\}$ are among the bases for the matroid $M_v$ (\cref{def:Mv}).}
	\label{fig:examplefocalpoint}
\end{figure}

	\subsection{Canonical geodesic}
	Let $v$ and $w$ be two vertices of $Q$.  The \defn{canonical geodesic} from $v$ to $w$ is the path
	$$
	\gamma_w(v) = (v, f_w(v), f_w(f_w(v)), \ldots, w)
	$$
	such that every step maximally reduces directed distance.  It follows from \cref{prop:short} (and its proof) that this is both a combinatorial geodesic and a path of shortest directed distance.

	\subsection{Cyclic Ordering of Neighbor Sets}\label{subsec:cyclic-ordering}
	For concreteness, \emph{cyclic ordering} in this section can be taken to mean \emph{counterclockwise ordering}.

	Subets $F_1, F_2, \ldots, F_m$ of $F$ are in \emph{weak cyclic order} around 
	$v$ if, after removing those $F_i$ that are equal to $F$, in the cyclic ordering of the neighbors of $v$ we have
	that each $F_i$ is a contiguous arc, and the 
	arcs appear in the cyclic sequence $F_1, F_2, \ldots, F_m$ (wrapping around $v$ at most once), allowing 
	for overlaps and coincidences between consecutive sets.

	\begin{lemma}\label{lem:cyclic}
		Let $Q$ be a CAT(0) planar graph. Let $v$ be a vertex and let 
		$z_{i_1}, z_{i_2}, \ldots, z_{i_m}$ be boundary vertices listed in 
		cyclic order around $\partial Q$. Then the neighbor sets 
		$F(v, z_{i_1}), F(v, z_{i_2}), \ldots, F(v, z_{i_m})$ are arranged in 
		weak cyclic order around $v$.
	\end{lemma}
	
	\begin{proof}
		By definition of weak cyclic order, we may assume that $v$ is not equal to any of the $z_i$.
		By definition, each of the sets $F(v,z_{i_j})$ is a one or two-element set, and if it has size two, it consists of two cyclically consecutive elements of $F$.
		Now we note that two canonical geodesics $\gamma_{z_i}(v)$ and $\gamma_{z_j}(v)$ cannot cross.  Indeed, the two paths can travel together for some time and diverge, but once they diverge they will not touch again.
		Since $F(v,z_i)$ is determined by the first step of $\gamma_{z_i}(v)$, we conclude that $F(v, z_{i_1}), F(v, z_{i_2}), \ldots, F(v, z_{i_m})$ are arranged in 
		weak cyclic order around $v$.
	\end{proof}
	
	\subsection{Minimizers are Focal Points}\label{subsec:minimizers-focal}
	
	\begin{proposition}[Minimizers are Focal Points]\label{prop:focal-a}
		Let $(Q,\z)$ be a CAT(0) planar graph. For any three boundary vertices 
		$(z_i, z_j, z_k)$, any distance minimizer is a focal point.
	\end{proposition}
	
	\begin{proof}
		Let $v$ be a distance minimizer for $(z_i, z_j, z_k)$, meaning $v$ 
		achieves the minimum value of \eqref{eq:Dx} over all vertices $x \in V(Q)$.
		
		\medskip
		\noindent\textbf{Step 1: No two targets are parallel at $v$.}
		
		Suppose for contradiction that two of the targets, say $z_j$ and $z_k$, 
		are parallel at $v$.
		
		\emph{Case 1a:} $F(v, z_j) = F(v, z_k) = \{x\}$ where $x \to v$.
		
		By Lemma~\ref{lem:edgedist}, since $x \to v$ and $x = f_{z_j}(v) = f_{z_k}(v)$:
		\[
		\delta(x, z_j) = \delta(v, z_j) - 2, \quad \delta(x, z_k) = \delta(v, z_k) - 2.
		\]
		
		For the third target $z_i$, the vertex $x$ is a neighbor of $v$, so 
		by Lemma~\ref{lem:edgedist}:
		\[
		\delta(x, z_i) \leq \delta(v, z_i) + 1.
		\]
		
		Therefore:
		\[
		D(x) = \delta(x, z_i) + \delta(x, z_j) + \delta(x, z_k) \leq 
		(\delta(v, z_i) + 1) + (\delta(v, z_j) - 2) + (\delta(v, z_k) - 2) = D(v) - 3.
		\]
		
		This contradicts the assumption that $v$ is a distance minimizer.
		
		\emph{Case 1b:} $F(v, z_j) = F(v, z_k) = \{x, y\}$ and $z_j \parallel_v z_k$.
		
		Let $v = v_0 \to v_1 \to \cdots \to v_r = w$ be a witness path. Consider 
		the vertex $v_1 = f_{z_j}(v) = f_{z_k}(v)$.
		
		Since $v \to v_1$, by Lemma~\ref{lem:edgedist}:
		\[
		\delta(v_1, z_j) = \delta(v, z_j) - 1, \quad \delta(v_1, z_k) = \delta(v, z_k) - 1.
		\]
		
		For $z_i$, we have:
		\[
		\delta(v_1, z_i) \leq \delta(v, z_i) + 2.
		\]
		
		Therefore:
		\[
		D(v_1) \leq (\delta(v, z_i) + 2) + (\delta(v, z_j) - 1) + (\delta(v, z_k) - 1) = D(v).
		\]
		
		Since $v$ is a distance minimizer, we must have equality, so $v_1$ is 
		also a distance minimizer.
		
		Repeating this argument along the witness path, all vertices 
		$v_0, v_1, \ldots, v_r = w$ are distance minimizers. At the witness $w$, 
		we have $|F(w, z_j)| = |F(w, z_k)| = 1$, reducing to Case 1a. This gives 
		a contradiction.
		
		We conclude that no two of $z_i, z_j, z_k$ are parallel at $v$.
		
		\medskip
		\noindent\textbf{Step 2: The sets $F(v, z)$ admit an SDR.}
		
		If $v = z \in \{z_i,z_j,z_k\}$, then $F(v, z) = F(v)$ has at least three 
		elements (by the technical convention for boundary vertices), so an 
		SDR trivially exists.
		
		Suppose $v \notin \{z_i,z_j,z_k\}$, so $|F(v, z)| \in \{1, 2\}$ for each $z \in  \{z_i,z_j,z_k\}$.  Since no two targets are parallel at $v$ (by Step 1), the only way an 
		SDR could fail to exist is if $|\bigcup_z F(v, z)| \leq 2$ and it follows that not 
		all $F(v, z)$ have size one.
		
		We may thus suppose that $\bigcup_z F(v, z) = F(v, z_i) = \{a, a'\}$ 
		for some cyclically adjacent pair $\{a, a'\}$, with $F(v, z_j), F(v, z_k) \subseteq \{a, a'\}$.
		
		Let $c \triangleleft F(v, z_i)$ be the vertex such that $a, c, a'$ are in cyclic order around 
		$v$ (i.e., $c = f_{z_i}(v)$, with $v \to c$). By Lemma~\ref{lem:edgedist},
		$\delta(c, z_i) = \delta(v, z_i) - 1$.  For $z_j$ and $z_k$, since $F(v, z_j), F(v, z_k) \subseteq \{a, a'\}$, we have			
		\[
		\delta(c, z_j) < \delta(v, z_j), \quad \delta(c, z_k) < \delta(v, z_k).
		\]
		It follows that $D(c) < D(v)$, contradicting the assumption that $v$ is a distance minimizer.
	\end{proof}
	
	\subsection{Focal Points are Distance Minimizers}\label{subsec:focalproof}
	
	We now prove the ``if" part of \cref{thm:focal}.
	
	\begin{lemma}\label{lem:Dsame}
		Let $(z_i, z_j, z_k)$ be fixed and recall \eqref{eq:Dx}
		$D(v) := \delta(v, z_i) + \delta(v, z_j) + \delta(v, z_k)$.
		\begin{enumerate}
			\item If $v \to u$ is a directed edge, then $D(u) = D(v)$ if and only 
			if exactly two of the three distances $\delta(\cdot, z)$ decrease 
			(by $1$ each) when moving from $v$ to $u$. In particular, in this case, 
			$v$ is a distance minimizer if and only if $u$ is a distance minimizer.
			
			\item If $w \to v$ is a directed edge, then $D(w) = D(v)$ if and only 
			if exactly one of the three distances $\delta(\cdot, z)$ decreases 
			(by $2$) when moving from $v$ to $w$. In particular, in this case, $v$ 
			is a distance minimizer if and only if $w$ is a distance minimizer.
		\end{enumerate}
	\end{lemma}
	
	\begin{proof}
		(1) By Lemma~\ref{lem:edgedist}, for each $z \in \{z_i,z_j,z_k\}$:
		\[
		\delta(u, z) - \delta(v, z) \in \{-1, +2\}.
		\]
		We have $D(u) = D(v)$ if and only if the sum of changes is zero. This 
		occurs precisely when two distances decrease by $1$ (contributing $-2$) 
		and one increases by $2$ (contributing $+2$).
		
		(2) By Lemma~\ref{lem:edgedist} applied to $w \to v$:
		\[
		\delta(w, z) - \delta(v, z) \in \{+1, -2\}.
		\]
		We have $D(w) = D(v)$ if and only if two distances increase by $1$ 
		(contributing $+2$) and one decreases by $2$ (contributing $-2$).
	\end{proof}
	
	\begin{definition}
		Let $\mathbf{v} = (v_0, v_1, \ldots, v_r)$ be a walk in $Q$. The three 
		$F$-sets $F(v_m, z_i), F(v_m, z_j), F(v_m, z_k)$ are 
		\emph{directed backwards} at $v_m$ (for $m \geq 1$) if:
		\begin{enumerate}
			\item $v_{m-1} \to v_m$ and at least two of the $F$-sets equal 
			$\{v_{m-1}\}$, or
			\item $v_{m-1} \leftarrow v_m$ and all three $F$-sets are contained 
			in $\kappa(v_{m-1})$,
		\end{enumerate}
		where $\kappa(v_{m-1})$ denotes the two outgoing neighbors of $v_i$ 
		that are adjacent to $v_{m-1}$ in the cyclic order around $v_i$.
	\end{definition}
	
	\begin{lemma}\label{lem:backwards}
		Let $Q$ be a CAT(0) planar graph. Fix three boundary vertices 
		$(z_i, z_j, z_k)$. Let $\mathbf{v} = (v_0, v_1, v_2, \ldots, v_r)$ be 
		a walk in $Q$.
		\begin{enumerate}
			\item If $v_0$ is a focal point but $v_1$ is not a focal point, then 
			$F(v_1, z_i), F(v_1, z_j), F(v_1, z_k)$ are directed backwards at $v_1$.
			
			\item If $\mathbf{v}$ is a combinatorial geodesic and 
			$F(v_m, z_i), F(v_m, z_j), F(v_m, z_k)$ are directed backwards at $v_m$, 
			then $F(v_{m+1}, z_i), F(v_{m+1}, z_j), F(v_{m+1}, z_k)$ are directed 
			backwards at $v_{m+1}$.
		\end{enumerate}
	\end{lemma}
	
	\begin{proof}
		We prove (1).  Assume $v_0$ is a focal point for $(z_i, z_j, z_k)$ 
		but $v_1$ is not.  Since $v_0$ is a focal point,
		\begin{itemize}
			\item the sets $F(v_0, z_i), F(v_0, z_j), F(v_0, z_k)$ admit an SDR, and
			\item no two of $z_i, z_j, z_k$ are parallel at $v_0$.
		\end{itemize}
		Since $v_1$ is not a focal point, at least one of these conditions 
		fails at $v_1$. We consider the two cases for the edge between $v_0$ 
		and $v_1$.
		
		\medskip
		\noindent{Case 1: The edge is directed $v_0 \to v_1$.}  By Definition~\ref{def:F-set}, for each $z$:
		if $v_1 \neq f_{z}(v_0)$ and $v_0 \to v_1$, then 
		$F(v_1, z) = \{v_0\}$.  Since $v_0$ is a focal point and the $F$-sets admit an SDR, there can 
		be at most one $z \in \{z_i,z_j,z_k\}$ for which $v_1 = f_{z}(v_0)$.
		
		\noindent {Subcase 1a:} For all three targets, $v_1 \neq f_{z}(v_0)$.  Then $F(v_1, z) = \{v_0\}$ for all $z \in \{z_i,z_j,z_k\}$. This means 
		all three $F$-sets at $v_1$ equal $\{v_0\}$, so the directed backwards 
		condition (1) is satisfied.
		
		\noindent {Subcase 1b:} For exactly one target, say $z_i$, we have 
		$v_1 = f_{z_i}(v_0)$.  Then $F(v_1, z_j) = F(v_1, z_k) = \{v_0\}$. At least two $F$-sets equal 
		$\{v_0\}$, so the directed backwards condition (1) is satisfied.
		
		\medskip
		\noindent{Case 2: The edge is directed $v_1 \to v_0$.} Since $v_0$ is a focal point and $v_1$ is not, and the edge is 
		$v_1 \to v_0$, the $F$-sets at $v_1$ must fail the focal point conditions. By a detailed analysis of Definition~\ref{def:F-set} (omitted for 
		brevity), when $v_1 \to v_0$ and $v_1$ fails to be a focal point, all 
		three $F$-sets $F(v_1, z)$ are contained in $\kappa(v_0)$, satisfying 
		condition (2) of directed backwards.
		
		We have now proven (1).
		
		The proof of (2) follows a similar case analysis to (1).		
	\end{proof}
	
	\begin{proof}[Proof of \cref{thm:focal}]
		After \cref{prop:focal-a}, it suffices to show that focal points are distance minimizers.
		Let $v$ be a focal point for $(z_i, z_j, z_k)$. Let $w$ be a distance 
		minimizer (which exists since $Q$ is finite). Let 
		$\mathbf{v} = (v = v_0, v_1, \ldots, v_r = w)$ be a combinatorial 
		geodesic from $v$ to $w$.
		
		\medskip
		\noindent\textbf{Claim:} One of the following holds:
		\begin{itemize}
			\item[(A)] $v_1$ is also a focal point, or
			\item[(B)] $v_1$ is not a focal point, and none of $v_2, \ldots, v_r$ 
			are focal points.
		\end{itemize}
		
		\begin{proof}[Proof of Claim]
			Suppose $v_1$ is not a focal point. By Lemma~\ref{lem:backwards}(1), 
			the $F$-sets at $v_1$ are directed backwards. By 
			Lemma~\ref{lem:backwards}(2), the $F$-sets at $v_i$ are directed 
			backwards for all $i = 1, 2, \ldots, r$. A vertex with $F$-sets directed 
			backwards cannot be a focal point (the SDR condition or non-parallelism 
			fails). Thus (B) holds.
		\end{proof}
		
		Option (B) implies that $w = v_r$ is not a focal point. But by 
		Proposition~\ref{prop:focal-a}, any distance minimizer is a focal point. 
		This is a contradiction.  Therefore, option (A) must always hold: $v_1$ is a focal point whenever 
		$v$ is a focal point.  Applying this inductively along the geodesic, all of $v_0, v_1, \ldots, v_r$ 
		are focal points.
		
		By Lemma~\ref{lem:Dsame}, when moving between adjacent focal points, 
		the total distance $D$ is preserved (the changes balance out). Therefore:
		\[
		D(v) = D(v_1) = \cdots = D(v_r) = D(w),
		\]
		and $v$ is also a distance minimizer.
	\end{proof}
	
	\subsection{Focal points exist}\label{subsec:focal-summary}
	
	\begin{corollary}[Existence of Focal Points]\label{cor:focal-exist}
		For any labeled CAT(0) planar graph $(Q,\z)$ and any three boundary vertices 
		$(z_i, z_j, z_k)$, a focal point exists.
	\end{corollary}
	
	\begin{proof}
		Since $Q$ is finite, a distance minimizer exists. By 
		\cref{thm:focal}, any distance minimizer is a focal point.
	\end{proof}

	\section{Tropical Pl\"ucker relations}\label{sec:mainproof}
	In this section we prove \cref{thm:main}(1),(2).  We assume that $(Q,\z)$ is a CCW-labeled CAT(0) planar graph, and show that $\tropP_\bullet = \tropP_\bullet(Q)$ satisfies the positive tropical Pl\"ucker relations.  This will prove \cref{thm:main}(2); the proof of \cref{thm:main}(1) is strictly easier.
	
	Let $a,b,c,d,e \in \z$ be five vertices, possibly not distinct, in CCW order around the boundary $\partial Q$.   We show that the positive tropical Pl\"ucker relation holds:
	\begin{equation}\label{eq:trop}
		\tropP_{abd} + \tropP_{ace} = \min( \tropP_{abc} + \tropP_{ade}, \tropP_{abe}+\tropP_{acd}).
	\end{equation}
	Note that $a$ plays a special role in \eqref{eq:trop}, but by cyclic symmetry there is no loss of generality.
	
	Let $v$ be a vertex and $F = F(v,z)$ be a neighbor set, which we assume to be of size one or two.  By \emph{following $F$} we mean moving from $v$ to the neighboring vertex $f_z(v) = v'$, or alternatively, taking one step along the canonical geodesic towards $z$.  Note that $v'$ depends only on $F = F(v,z)$ and not on $z$ itself.

		Our strategy is to either find a single focal point $p$ for 
		\begin{equation}\label{eq:single}
			\mbox{all four triples $abd,ace,abc,ade$ or all four triples $abd,ace,abe,acd$}
		\end{equation}
		or to find two focal points, one focal point $p = p_x$ for three triples
		$$
		\mbox{$abc,abd,abe$ or $abc,acd,ace$ or $abd,acd,ade$ or $abe,ace,ade$},
		$$
		where we set $x = b,c,d,e$ respectively, and the other focal point $q$ for at least two of the remaining three sets.  In the case of two focal points $p$ and $q$, we will also establish that 
		\begin{equation}\label{eq:dprop}
			\begin{split}
				\dist(q,a) &= \dist(q,p) + \dist(p,a) \\
				\dist(p,y) &= \dist(p,q) + \dist(q,y) \text{ for $y \in \{b,c,d,e\}\setminus\{x\}$}.
			\end{split}
		\end{equation}
		
		In the case of a single focal point $p$, applying \cref{thm:focal}, the equation \eqref{eq:trop} follows immediately.
		
		In the case of two focal points, and supposing that $x = b$ for concreteness, we get
		\begin{align*}
			\tropP_{abc} + \tropP_{ade} &= -\frac{1}{3}(\dist(p,a)+\dist(p,b)+\dist(p,c) + \dist(q,a)+\dist(q,d)+\dist(q,e) )\\
			\tropP_{abd} + \tropP_{ace} &= -\frac{1}{3}(\dist(p,a)+\dist(p,b)+\dist(p,d) + \dist(q,a)+\dist(q,c)+\dist(q,e) )\\
			\tropP_{abe} + \tropP_{acd} &= -\frac{1}{3}(\dist(p,a)+\dist(p,b)+\dist(p,e) + \dist(q,a)+\dist(q,c)+\dist(q,d) )
		\end{align*}
		and all three quantities are equal.
		
		We start by considering the vertex $a$.  We have that $a$ is a focal point for $axy$ with $x,y \in \{b,c,d,e\}$ if and only if $x$ and $y$ are not parallel at $a$.  Applying \cref{lem:cyclic} and possibly a cyclic rotation of $b,c,d,e$, the parallel equivalence classes at $a$ among $\{b,c,d,e\}$ can be assumed to be one of the following:
		\begin{enumerate}
			\item[(i)]
			four equivalence classes,
			\item[(ii)]
			three equivalence classes, $\{b,c\}$, $\{d\}$, and $\{e\}$,
			\item[(iii)]
			two equivalence classes, $\{b,c\}$ and $\{d,e\}$,
			\item[(iv)]
			two equivalence classes $\{b\}$ and $\{c,d,e\}$,
			\item[(v)]
			one equivalence class $\{b,c,d,e\}$.
		\end{enumerate}
		
		In the following we repeatedly use the following statement that follows from \cref{prop:short}: if $\gamma$ is a combinatorial geodesic from $u$ to $v$ passing through $w$, then $\dist(u,v) = \dist(u,w) + \dist(w,v)$.
		
		\noindent Case (i): In this case, $p = a$ is a focal point for all six of the triples $abc,abd,abe,acd,ace,ade$. 
		
		\noindent Case (ii): In this case, $p =a$ is a focal point for $abd,ace,abe,acd$, and we are in situation \eqref{eq:single}.  The other cases obtained by cyclically relabeling $b,c,d,e$ work similarly.
		
		\noindent Case (iii):
		In this case, $p =a$ is a focal point for $abd,ace,abe,acd$, and we are in situation \eqref{eq:single}.  The other cases obtained by cyclically relabeling $b,c,d,e$ work similarly.
		
		\noindent Case (iv): Two equivalence classes $\{b\}$ and $\{c,d,e\}$.
		In this case, $p = p_b = a$ is a focal point for $abc,abd,abe$.  To find the vertex $q$, set $F = F(a,c) = F(a,d) = F(a,e)$ and follow $F$ repeatedly until we reach a vertex $q$ where $c,d,e$ are not all parallel.  Then $q$ is a focal point for at least two of $acd,ace,ade$, and it satisfies \eqref{eq:dprop}.
		
		\noindent Case (v): One equivalence class $\{b,c,d,e\}$.  Set $F = F(a,b)=F(a,c)=F(a,d) =F(a,e)$.  Follow $F$ repeatedly until we reach a vertex $p$ where not all four of $b,c,d,e$ are parallel.  If we have three or more equivalence classes, or we have two equivalence classes of size two, then $p$ is a single focal point satisfying \eqref{eq:single} and we are done.  If we have two equivalence classes, one of size three and one of size one, we follow the recipe in Case (iv) to find a second focal point $q$.  We check that $p$ and $q$ satisfy \eqref{eq:dprop}, noting that the path we have followed is a combinatorial geodesic between $a$ and $q$ passing through $p$.
		
		This completes the proof of \cref{thm:main}(2).
		
		\section{Simple and normal CAT(0) planar graphs}\label{sec:simplenormal}		
		\subsection{Simplicity}

		\begin{proposition}\label{prop:sum}
			Suppose that $(Q',\z')$ and $(Q'',\z'')$ are two labeled CAT(0) planar graphs with boundary vertices $\z' = (z'_1,\ldots,z'_n)$ and $\z'' = (z''_1,\ldots,z''_n)$.  Suppose that vertex $v' \in V(Q')$ is labeled by $z'_p,z'_{p+1},\ldots,z'_q$ and vertex $v'' \in V(Q'')$ is labeled by $z''_r, z''_{r+1}, \ldots, z''_s$, so that the cyclic intervals $[p,q]$ and $[r,s]$ are proper and satisfy $[p,q] \cup [r,s] = [n]$.  Let $Q$ be obtained from $Q'$ and $Q''$ by identifying the vertices $v'$ and $v''$, and using the boundary vertices $z_i = z'_i$ if $i \in [r,s]$ and $z_i = z''_i$ if $i \in [p,q]$.  Then
			$$
			\tropP_\bullet(Q) = \tropP_\bullet(Q') + \tropP_\bullet(Q'').
			$$
		\end{proposition}
		\begin{proof}
			Suppose $i,j,k \in [p,q]$.  Then $z_i,z_j,z_k$ are all vertices of $Q''$.  So $\tropP_{i,j,k}(Q') = 0$ and we have $\tropP_{i,j,k}(Q) = \tropP_{i,j,k}(Q'')$, as expected.  Similarly for $(i,j,k) \in [r,s]$.
			
			Now suppose that $i,j \in [p,q]$ and $k \in [r,s] \setminus [p,q]$.  (The case $i,j \in [r,s]$ and $k \in [p,q] \setminus [r,s]$ is identical.)  Then there exists a focal point for $i,j,k$ among the vertices $V(Q'') \subset V(Q)$.  It follows from \cref{thm:focal} that
			$$
			\tropP_{i,j,k}(Q) = \Sigma_{Q''}(z''_i,z''_j,z''_k) + \dist_{Q'}(v',z'_k) =\tropP_{i,j,k}(Q'')  + \tropP_{i,j,k}(Q'), 
			$$
			where we note that $\Sigma_{Q'}(v',v',b'_k) = \dist_{Q'}(v',z'_k)$ since $v'$ is always a focal point for $(v',v',z'_k)$.
		\end{proof}
		
		\cref{prop:sum} holds even if the boundary curve is not an oriented cycle.
		
		Recall that a vertex $v \in \partial Q$ is called \defn{acute} if it has degree $2$, and is a vertex of exactly one triangular face of $Q$.  Acute vertices are necessarily on the boundary of $Q$.  Recall the definition of simple CAT(0) planar graph from \cref{def:simple}.
		\begin{lemma}\label{lem:acute}
			A simple CAT(0) planar graph $Q$ has at least three acute vertices.
		\end{lemma}
		\begin{proof}
			Boundary vertices of a simple CAT(0) planar graph $Q$ are either acute (and have degree $2$) or have degree $\geq 4$.  We now apply \cref{lem:graph}.
		\end{proof}

		\begin{lemma}\label{lem:ubound}
			Let $Q$ be a simple CAT(0) planar graph and $v_0 \to v_1 \to \cdots \to v_r$ be the boundary path between two acute vertices $v_0$ and $v_r$, with $v_1, v_2,\ldots,v_{r-1}$ non-acute.  Let $u$ be another vertex of $Q$, not equal to any of the $v_i$.  Consider the neighbors $u'_i := f_{v_i}(u)$ of $u$.  Then exactly one of the following holds:
			\begin{enumerate}
				\item All the $u'_i$ are the same.
				\item For some $a \in [0,r-1]$ we have $u'_0 = u'_1 = \cdots = u'_a \neq u'_{a+1} = u'_{a+2} = u'_r$.  The two vertices $u'_a$ and $u'_{a+1}$ are cyclically adjacent around $u$.
				\item For some $a \in [1,r-2]$ we have $u'_0 = \cdots = u'_{a-1} \neq u'_a \neq u'_{a+1} = \cdots = u'_r$.  The three vertices $u'_{a-1}, u'_a, u'_{a+1}$ are cyclically adjacent around $u$, and furthermore we have $u'_{a\pm 1} \to u$ and $u \to u'_a$.
			\end{enumerate}
		\end{lemma}
		\begin{proof}
			By \cref{lem:strandstripF}, we have $u'_i = u'_{i+1}$ unless $u$ lies on the strand strip $S(i) := S(v_i \to v_{i+1})$. By \cref{prop:strandintersect}, the strand strips $S(i)$ and $S(j)$ have no triangles in common, and it follows that $u$ lies on zero, one, or two of the strand strips.  These three cases correspond to (1), (2), (3) respectively.
		\end{proof}
		
		
		\begin{lemma}\label{lem:glue}
			Let $Q$ be a CAT(0) planar graph with boundary oriented counterclockwise.  Then $Q$ is a union of simple CAT(0) planar graphs glued along cut vertices of $\partial Q$.  
		\end{lemma}
		
		\subsection{Normality}
		
		\begin{lemma}\label{lem:fboundary}
			Suppose that $Q$ is a CAT(0) planar graph with boundary oriented counterclockwise.  Let $v = v_0 \to v_1 \to \cdots \to v_{r-1} \to v_r = w$ be a portion of the boundary such that none of the vertices $v_1,\ldots,v_{r-1}$ are acute.  Then this path is a path of minimum distance from $v$ to $w$ and (traversed backwards) a path of minimal distance from $w$ to $v$.  Furthermore, we have $f_w(v) = v_1$ and $f_v(w) = v_{r-1}$.
		\end{lemma}
		
		\begin{proof}
			First suppose none of the vertices $v_i$ are cut vertices.  Then using \cref{cor:localdistance} we may prove by induction on $r$ that $\delta(v_i, v) = 2i$ and $\delta(v_{r-i},w) = i$, and the other claims also follow.
			
			By \cref{lem:acute}, the vertices $v_i$ are distinct.  If one of the vertices, say $v_i$, is a cut vertex, then \emph{any} path from $v$ to $w$ will pass through $v_i$.  The claim then reduces to the same claims for the paths $v_0 \to \cdots \to v_i$ and $v_i \to \cdots \to v_r$.  We are done by induction on the number of cut vertices.
		\end{proof}
		
		\begin{proposition}\label{prop:normal}
			Let $(Q,\z)$ be a normal CAT(0) planar graph.  Then every vertex $v \in V(Q)$ is normal.
		\end{proposition}
		\begin{proof}
			First suppose that $Q$ is simple.  Let $v \in V(Q)$.  First consider the case that $v$ is an interior vertex.  Let $z_{i_1},z_{i_2},\ldots$ be the acute vertices in cyclic order.  Consider the sets $F(v,z_{i_j})$, which are weakly cyclically ordered around $v$ by \cref{lem:cyclic}.  By \cref{lem:ubound}, the consecutive sets $F(v,z_{i_j})$ and $F(v,z_{i_{j+1}})$ must intersect.  It follows that $\bigcup_j F(v, z_{i_j}) = F(v)$, and by \cref{lem:acute} it follows that $v$ is a focal point for some triple among the acute vertices.  By \cref{thm:focal}, $v$ is normal.
			
			Next consider the case that $v$ is a boundary vertex.  Let $u$ (resp. $w$) be the next (resp. previous) labeled vertex along the boundary from $v$.  Note that $(Q,\z)$ has at least three acute vertices by \cref{lem:graph}, and thus by normality the three vertices $w,v,u$ are distinct.  By \cref{lem:fboundary}, we have $f_u(v)$ (resp. $f_w(v)$) is the next (resp. previous) vertices along the boundary from $v$.  If $v$ is acute it follows immediately that $v$ is normal.  By \cref{lem:ubound}, if $v$ is not acute then as $z$ varies over $\z$, we have that $\bigcup_{z \in \z} F(v,z) = F(v)$, and again $v$ is normal.  
			
			This proves the result for $Q$ simple.  For a general normal $(Q,\z)$, by \cref{lem:glue}, $Q = \bigcup_j Q_j$ where the $Q_j$ are simple CAT(0) planar graphs.   Let $v \in V(Q)$.  If $v$ belongs to component $Q_j$ and is not a cut vertex, the normality of $v$ follows from the result for the simple normal $(Q_j, \z^{(j)})$ where we keep all the original labels of $Q_j$, and additionally label all the cut vertices of $Q_j$.  (\cref{lem:acute} guarantees that there are labeled vertices on both sides of every cut vertex.)
			
			Finally, if $v$ is a cut vertex and $Q = \bigcup Q_k$ where the $Q_k$ are glued at $v$, then the above argument applied to one of the $Q_k$ suffices to show that $v$ is normal.
		\end{proof}		
		
		\part{Strands and planar basis}\label{part:webs}
		
		\section{Non-elliptic strand combinatorics}\label{sec:strands}
		
		A non-elliptic web $W$ is a \emph{reduced} plabic graph in the sense of Postnikov \cite{Pos}.  It follows from general results on reduced plabic graphs that all of the strands of $W$ go from boundary to boundary and have no self-intersections.  Furthermore, two strands cannot have a \emph{bad double-crossing}: that is $\eta$ and $\eta'$ cannot both use edges $e$ and $e'$ so that these edges are in the same order on both strands.  
		
		\subsection{Proof of \cref{prop:strandintersect}}
		Let $W$ be a non-elliptic web and $F$ be an interior face of $W$, so that $F$ is a $m$-gon for even $m \geq 6$.  Let $\str$ and $\str'$ be two strands with edges $e \in \str$ (resp. $e' \in \str'$) on the perimeter of $F$.  (The intersection of $\str$ with the perimeter of $F$ consists of two edges, and it is possible or not for $e = e'$.)
		
		Suppose that $\str$ and $\str'$ intersect at some other edge $f$, not on the perimeter of $F$.  
		
		Consider the curve $\gamma$ that follows $\str$ from $e$ to $f$, then follows $\str'$ from $f$ to $e$, and then goes along the perimeter of $F$ from $e$ to $e'$ in such a way that $F$ is contained in the region $R$ bounded by $\gamma$; see \cref{fig:FF'}.
		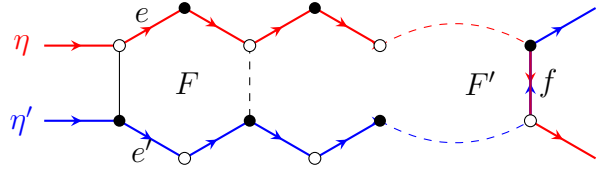
\begin{figure}
		\begin{center}
		$$
		\begin{tikzpicture}[
			bvert/.style={circle, fill=black, inner sep=0, minimum size=4.5pt},
			wvert/.style={circle, draw=black, fill=white, inner sep=0, minimum size=4.5pt},
			qvert/.style={circle, fill=blue!60!black, inner sep=0, minimum size=6pt},
			midarrow/.style={decoration={markings, mark=at position 0.5 with {\arrow{stealth}}}, postaction={decorate}},
			]
			
		\coordinate (a) at (0,0);
		\coordinate (b) at ($(a) + (30:1)$);
		\coordinate (c) at ($(b) + (-30:1)$);
		\coordinate (d) at ($(c) + (30:1)$);
		\coordinate (e) at ($(d) + (-30:1)$);
		\coordinate (f) at ($(e) + (2,0)$);
		
		\coordinate (a1) at (0,-1);
		\coordinate (b1) at ($(a1) + (-30:1)$);
		\coordinate (c1) at ($(b1) + (30:1)$);
		\coordinate (d1) at ($(c1) + (-30:1)$);
		\coordinate (e1) at ($(d1) + (30:1)$);
		\coordinate (f1) at ($(e1) + (2,0)$);
		
		\draw[midarrow,red,thick] (a) to (b);
		\draw[midarrow,red,thick] (b) to (c);
		\draw[red,thick,midarrow] ($(a)+(-1,0)$) to (a);
		\draw[midarrow,red,thick] (c) to (d);
		\draw[midarrow,red,thick] (d) to (e);
		\draw[midarrow,red,dashed] (e) edge[bend left] (f);
		\draw[blue,thick,midarrow] (f) to ($(f)+(30:1)$);
		
		\draw[midarrow,blue,thick] (a1) to (b1);
		\draw[midarrow,blue,thick] (b1) to (c1);
		\draw[blue,thick,midarrow] ($(a1)+(-1,0)$) to (a1);
		\draw[midarrow,blue,thick] (c1) to (d1);
		\draw[midarrow,blue,thick] (d1) to (e1);
		\draw[midarrow,blue,dashed] (e1) edge[bend right] (f1);
		\draw[red,thick,midarrow] (f1) to ($(f1)+(-30:1)$);
		
		\node[text=red] at ($(a) +(-1.3,0)$) {$\eta$};
		\node[text=blue] at ($(a1) +(-1.3,0)$) {$\eta'$};
		
		\draw (a)--(a1);

		\draw[red,thick,midarrow] (f) -- (f1);
		\draw[blue,thick,midarrow] (f1) -- (f);
		\draw[purple,thick] (f) --(f1);
		\draw[dashed] (c) --(c1);

		\node at (0.9,-0.5) {$F$};
		\node at (4.8, -0.5) {$F'$};
		\node at (0.3,0.4) {$e$};
		\node at (0.3,-1.4) {$e'$};
		\node at (5.7,-0.5) {$f$};
			
		\node[wvert] at (a) {};
		\node[bvert] at (b) {};
		\node[wvert] at (c) {};
		\node[bvert] at (d) {};
		\node[wvert] at (e) {};
		\node[bvert] at (f) {};
		
			\node[bvert] at (a1) {};
		\node[wvert] at (b1) {};
		\node[bvert] at (c1) {};
		\node[wvert] at (d1) {};
		\node[bvert] at (e1) {};
		\node[wvert] at (f1) {};
		\end{tikzpicture}
		$$
		\end{center}
		\caption{A picture of two strands $\eta$ and $\eta'$ adjacent to the same face $F$, and later intersecting.}
		\label{fig:FF'}
		\end{figure}
		
		We assume that $\gamma$ is a simple closed curve.  Let $F'$ be the face in $W(R)$ with $f$ on its perimeter and suppose that $F'$ is a $m'$-gon.  We apply the planar dual of \cref{lem:graph} to $W(R)$.  The conclusion is as follows.  Let $\alpha$ (resp. $\alpha'$) be the number of edges of $F$ (resp. $F'$) that are on the boundary of $R$.  Then $\alpha+\alpha' \geq m+m' - 2$ and we have $\alpha \geq 3$ (resp. $\alpha' \geq 3$) with equality if and only if $\str$ and $\str'$ intersect along the perimeter of $F$ (resp. $F'$).

		Since we have chosen $\gamma$ to be a simple closed curve, we also have $m - \alpha \geq 1$ (resp. $m' -\alpha' \geq 1$).  Substituting, we get the inequality $2 \geq (m-\alpha) + (m'-\alpha') \geq 2$, so the only solutions are when $m - \alpha = m'-\alpha' = 1$.  Since $m,m' \geq 6$, this forces $\alpha,\alpha' \geq 5$, and in particular $\str$ and $\str'$ cannot intersect along $\gamma$.  Thus the picture in \cref{fig:FF'} can never occur.

		The cases $m - \alpha = m'-\alpha' = 1$ give the only solutions, pictured in \cref{fig:parallel}, where $\str,\str'$ never intersect.
		
	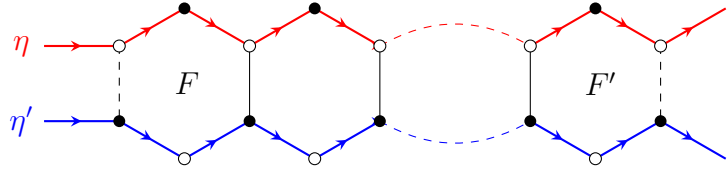
\begin{figure}
		\begin{center}
		$$
		\begin{tikzpicture}[
			bvert/.style={circle, fill=black, inner sep=0, minimum size=4.5pt},
			wvert/.style={circle, draw=black, fill=white, inner sep=0, minimum size=4.5pt},
			qvert/.style={circle, fill=blue!60!black, inner sep=0, minimum size=6pt},
			midarrow/.style={decoration={markings, mark=at position 0.5 with {\arrow{stealth}}}, postaction={decorate}},
			]
			
		\coordinate (a) at (0,0);
		\coordinate (b) at ($(a) + (30:1)$);
		\coordinate (c) at ($(b) + (-30:1)$);
		\coordinate (d) at ($(c) + (30:1)$);
		\coordinate (e) at ($(d) + (-30:1)$);
		\coordinate (f) at ($(e) + (2,0)$);
		\coordinate (g) at ($(f) + (30:1)$);
		\coordinate (h) at ($(g) + (-30:1)$);
		
		\coordinate (a1) at (0,-1);
		\coordinate (b1) at ($(a1) + (-30:1)$);
		\coordinate (c1) at ($(b1) + (30:1)$);
		\coordinate (d1) at ($(c1) + (-30:1)$);
		\coordinate (e1) at ($(d1) + (30:1)$);
		\coordinate (f1) at ($(e1) + (2,0)$);
		\coordinate (g1) at ($(f1) + (-30:1)$);
		\coordinate (h1) at ($(g1) + (30:1)$);
		
		\draw[midarrow,red,thick] (a) to (b);
		\draw[midarrow,red,thick] (b) to (c);
		\draw[red,thick,midarrow] ($(a)+(-1,0)$) to (a);
		\draw[midarrow,red,thick] (c) to (d);
		\draw[midarrow,red,thick] (d) to (e);
		\draw[midarrow,red,dashed] (e) edge[bend left] (f);
		\draw[midarrow,red,thick] (f) to (g);
		\draw[midarrow,red,thick] (g) to (h);
		\draw[red,thick,midarrow] (h) to ($(h)+(30:1)$);
		
		\draw[midarrow,blue,thick] (a1) to (b1);
		\draw[midarrow,blue,thick] (b1) to (c1);
		\draw[blue,thick,midarrow] ($(a1)+(-1,0)$) to (a1);
		\draw[midarrow,blue,thick] (c1) to (d1);
		\draw[midarrow,blue,thick] (d1) to (e1);
		\draw[midarrow,blue,dashed] (e1) edge[bend right] (f1);
		\draw[midarrow,blue,thick] (f1) to (g1);
		\draw[midarrow,blue,thick] (g1) to (h1);
		\draw[thick,blue,midarrow] (h1) to ($(h1)+(-30:1)$);
		
		\node[text=red] at ($(a) +(-1.3,0)$) {$\eta$};
		\node[text=blue] at ($(a1) +(-1.3,0)$) {$\eta'$};
		
		\draw[dashed] (a)--(a1);
		\draw (c) --(c1);
		\draw (e)--(e1);
		\draw(f)--(f1);
		\draw[dashed] (h)--(h1);

		\node at (0.9,-0.5) {$F$};
		\node at (6.4, -0.5) {$F'$};

		\node[wvert] at (a) {};
		\node[bvert] at (b) {};
		\node[wvert] at (c) {};
		\node[bvert] at (d) {};
		\node[wvert] at (e) {};
		\node[wvert] at (f) {};
		\node[bvert] at (g) {};
		\node[wvert] at (h) {};
		
			\node[bvert] at (a1) {};
		\node[wvert] at (b1) {};
		\node[bvert] at (c1) {};
		\node[wvert] at (d1) {};
		\node[bvert] at (e1) {};
		\node[bvert] at (f1) {};
		\node[wvert] at (g1) {};
		\node[bvert] at (h1) {};
		\end{tikzpicture}
		$$
		\end{center}
		\caption{Two strands $\eta$ and $\eta'$ can both be adjacent to many faces, as long as the two strands never intersect.  The dashed lines can be multiple edges: the starting and ending faces $F$ and $F'$ can have arbitrarily long perimeter.  The middle faces are all hexagons.}
		\label{fig:parallel}
		\end{figure}

		\subsection{Strand strips}
		We assume in this section that $W$ has black boundary vertices and thus the dual $(Q,\z)$ is a CCW-labeled CAT(0) planar graph.
		
		Let $\str=\str_i$ be a strand ending at boundary vertex $i$ of $W$.  The \defn{strand strip} $S_i$ is the union of all triangles (and the vertices and edges of those triangles) in $Q(W)$ corresponding to vertices visited by $\str$.  Let $z_i \to z'_i$ be the boundary edge of $Q$ out of $z_i$.  The vertex $z'_i$ may or may not be among the $\z$.
		
		\begin{lemma}\label{lem:strandcomb}
			The two long parallel sides of the strand strip are combinatorial geodesics, and are the canonical geodesics towards $z_i$ or $z'_i$ for all vertices on the same side.  See \cref{fig:strandstripparallel}.
		\end{lemma}
		\begin{proof}
			This follows from \cref{lem:combdist} by induction.
		\end{proof}
		\begin{figure}[h!]
			\centering
			\includegraphics[width=0.49\linewidth]{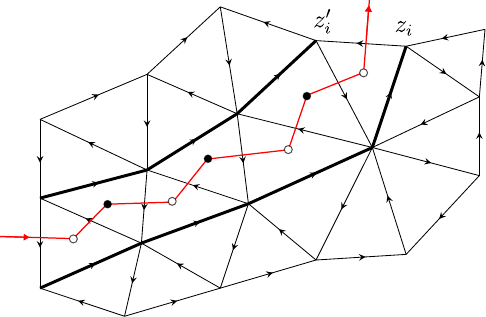} \includegraphics[width=0.49\linewidth]{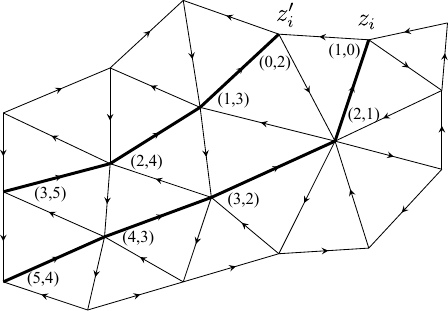}
			\caption{Left: strand strips are bounded by combinatorial geodesics.  See \cref{lem:strandcomb}.  Right: Distances $(\delta(\cdot ,z'_i),\delta(\cdot ,z_i))$ in \cref{lem:stranddist}.}
			\label{fig:strandstripparallel}
		\end{figure}
		
		\begin{lemma}\label{lem:stranddist}
			The distance functions $\delta(\cdot, z'_i)$ and $\delta(\cdot, z_{i})$ restricted to the vertices on $S_i$ are as illustrated in \cref{fig:strandstripparallel}.
		\end{lemma}		
		\begin{proof}
			Follows from \cref{lem:strandcomb} and \cref{prop:short}.
		\end{proof}
				
		\begin{lemma}\label{lem:strandstrip} \
			\begin{enumerate}
				\item
				Let $v$ be a vertex on the side of the strand $\str$ containing $z'_i$.  Then $\delta(v,z_{i})= \delta(v,z'_i) + 2$.  Furthermore, paths with shortest directed distance from $v$ to $z'_i$ (resp. $z_{i}$) do not cross the strand $\str$ (resp. cross the strand $\str$ once).
				\item
				Let $v$ be a vertex on the side of the strand containing $z_{i}$.  Then $\delta(v,z'_{i})= \delta(v,z_{i}) + 1$.  Furthermore, paths with shortest directed distance from $v$ to $z_{i}$ (resp. $z'_{i}$) do not cross the strand $\str$ (resp. cross the strand $\str$ once).
			\end{enumerate}
		\end{lemma}
		\begin{proof}
			Immediate from \cref{lem:stranddist}.
		\end{proof}
		
		\begin{lemma}\label{lem:strandstripF}\
			\begin{enumerate}
				\item
				Let $v$ be a vertex not on the strand strip $S_i$.  Then $F(v,z'_i) = F(v,z_{i})$.
				\item
				Let $v$ be on the strand strip, and on the side with $z'_i$.  Then $|F(v,z'_i)|=2$, $|F(v,z_{i})|=1$, and $F(v,z_{i}) \subset F(v,z'_i)$.
				\item
				Let $v$ be on the strand strip, and on the side with $z_{i}$.  Then $|F(v,z'_i)|=1$, $|F(v,z_{i})|=2$, and $F(v,z'_{i}) \subset F(v,z_{i})$.
			\end{enumerate}
		\end{lemma}
		\begin{proof}
			(2) and (3) follow immediately from \cref{lem:stranddist} and \cref{cor:localdistance}.  We claim that on the side of $\str$ containing $z'_i$, we have $\dist(\cdot,z_i) = \dist(\cdot, z'_i) + 2$ (and similarly on the side containing $z_i$).  This follows as \cref{cor:localdistance} states that the directed distance functions can be propagated using only local geometry.  Since the $F$-sets are determined by the distance function locally, (1) follows.
		\end{proof}
		
		\subsection{Normal and cyclic-less webs}\label{sec:normal}
		In this section we prove \cref{thm:normalnormal}.  Recall from \cref{def:normalweb} that a web is normal if and only if strands that intersect have different labels.
		
		\begin{proposition}\label{prop:acuteweb}
			Let $W$ be a non-elliptic web with black boundary that is not normal.  Then there are two adjacent boundary vertices with the same label connected to the same white interior vertex.
		\end{proposition}
		\begin{proof}
			Let $W$ be a non-elliptic web with black boundary and let $\str$ and $\str'$ be two strands that intersect and end at the same boundary vertex.  Since strands are continuous and go from boundary to boundary we may assume that the last edges $e$ and $e'$ used by $\str$ and $\str'$ are distinct and cyclically adjacent.  Suppose that $e$ and $e'$ are not incident to the same interior vertex.  Let $F$ be the (boundary) face with $e$ and $e'$ on its perimeter.  We would like to apply \cref{prop:strandintersect} to our situation.   To do this we split the boundary vertex and add edges so that $F$ becomes an interior face with at least six sides.  Applying \cref{prop:strandintersect} we see that the situation is impossible as $\str$ and $\str'$ are assumed to intersect somewhere else.  We conclude that $e$ and $e'$ are incident to the same interior vertex, as required.
		\end{proof}
		
		\begin{proof}[Proof of \cref{thm:normalnormal}]
			In $W$, we have two adjacent boundary vertices with the same label connected to the same white interior vertex if and only if we have an unlabeled acute vertex in $(Q(W),\z(W))$.  The equivalence of the two notions of normality follows from \cref{prop:acuteweb}.
			
			We now prove the equivalence of the two notions of cyclic-less.  Suppose that $(Q,\z)$ is cyclic, with a boundary path connecting $z_i$ and $z_{i+1}$ where all intermediate vertices have degree 4.  Then the white vertex in $W$ corresponding to the unique triangle with $z_i$ as a vertex has cyclic strand triple $(i,i+1,i+2)$.  Thus $(Q,\z)$ cyclic implies that $W$ is cyclic.
						
			Now suppose that $W$ has a strand triple $(i,i+1,i+2)$ at a vertex $a$.  Consider the strands $\eta_i, \eta_{i+1}, \eta_{i+2}$ passing through $v$.  Let $a = a_0, a_1,\ldots$ be the vertices of $W$ that $\eta_{i+1}$ passes through, starting at $a$.  We claim that $a_1$ also has strand triple $(i,i+1,i+2)$.  For two of the strands $\eta_i, \eta_{i+1}, \eta_{i+2}$ it is clear that they pass through $a_1$.  Let $\eta$ (resp. $\eta'$) be the third strand passing through $a_0$ (resp. $a_1$).  By \cref{prop:strandintersect}, $\eta$ and $\eta'$ cannot intersect.  It follows that the end point of $\eta'$ must be labeled either $i+1$, or the same as the endpoint of $\eta$.  Since $W$ is normal, we see that $a_1$ has strand triple $(i,i+1,i+2)$.  Repeating this argument, we may now assume that $a$ is connected to the boundary vertex of $W$ labeled $i+1$.				
			
			The union of all triangles on or to the left of the strand $\eta_i$ cuts off a CAT(0) planar graph $Q' \subset Q$.  Since $\eta_i$ goes from the boundary vertex $b_{i+1}$ to the boundary vertex $b_i$ of $W$, the part of $\partial Q$ that belongs to $Q'$ contains exactly one labeled vertex $z_i$.  By \cref{lem:acute}, $Q'$ contains at least three acute vertices -- the only possibilities are $z_i$ and the corners at the start and end of $\eta_i$.  Thus $Q'$ has exactly three acute vertices and the proof of \cref{lem:acute} implies that all boundary vertices of $Q'$ that are not acute have degree four.  Repeating the argument for $\eta_{i+2}$, we deduce that all intermediate vertices on the boundary path from $z_i$ to $z_{i+1}$ have degree four.  We conclude that $W$ cyclic implies that $(Q,\z)$ is cyclic.			
\end{proof}
		
		\subsection{Uniqueness of Strand Triples}
		
		\begin{lemma}\label{lem:strand-bijection}
			Let $Q$ be a standard CAT(0) planar graph. Then the interior triangles of $Q$ are 
			uniquely determined by their strand triples. That is, the map
			\[
			T \mapsto (i_T, j_T, k_T)
			\]
			from interior triangles to strand triples is injective.
		\end{lemma}
		
		\begin{proof}  
			Let $a$ be a vertex of a non-elliptic web $W$ with strand triple $(i,j,k)$.   Two strands $\str_i$ and $\str_j$ that belong to the strand triple of $a$ intersect along one of the edges, say $e = (a,b)$, incident to $a$.  By \cref{prop:strandintersect}, the two strands $\str_i$ and $\str_j$ do not intersect elsewhere in $W$.  So the only way that $(i,j,k)$ can be a strand triple of another vertex is if it is also the strand triple of the vertex $b$.  
			
			Suppose that $\str_k$ passes through both $a$ and $b$ (but not along the edge $e$).  Noting the direction that $\str_k$ turns at $a$ and at $b$ we see that this is impossible because $\str_k$ has no self-intersections.
		\end{proof}

		\section{Planar basis and planar cross-ratios}\label{sec: planar basis}
		
		\begin{definition}[Planar basis]\label{def:planar-basis}
			Define the \emph{directed distance} $d(e_I, e_J)$ on the vertices $e^I$ of the hypersimplex $\Delta(k,n)$ to be the minimal number of steps of the form $e_i - e_{i+1}$ (indices mod $n$) required to travel from $e_I$ to $e_J$ along edges of $\Delta(k,n)$.  The \emph{planar basis element} $\h_J = (\h_J)_\bullet \in \R^{\binom{[n]}{k}}$ is
			\[
			\mathfrak{h}_J := \frac{1}{n} \sum_{I \in \binom{[n]}{k}} d(e_J, e_I) \, e^I.
			\]
		\end{definition}
		
		For example, with $n = 6$ and $k = 3$, we have 
		\begin{align*}
		6\h_{124} &= e^{123} + e^{246} + 2e^{236} + 2e^{146}+2e^{245}+3e^{136}+3e^{235}+3e^{145}+4e^{126}+4e^{135}+4e^{234}+4e^{456}\\
		&+5e^{125}+5e^{134}+5e^{356}+6e^{256}+6e^{346}+7e^{156}+7e^{345}.
		\end{align*}

		\begin{definition}[(Tropical) Planar cross-ratio]\label{def:tropical-cross-ratio}
			For $J = \{j_1 < \cdots < j_k\} \in \binom{[n]}{k}^{\ncyc}$, denote by $I'_J = \{j \in J : j+1 \notin J\}$ the set of \emph{cyclic endpoints} of $J$.  The \emph{cubical array} is
			\[
			D_J := \left\{ e_J + \sum_{m \in M} (e_{m+1} - e_m) : M \subseteq I'_J \right\}.
			\]
			For $\tropP_\bullet \in \R^{\binom{[n]}{k}}$, the \emph{tropical cross-ratio} (or \emph{tropical $u$-variable}) is
			\[
			u^t_J(\tropP_\bullet) := \sum_{e_M \in D_J} (-1)^{|J \cap M| - k - 1} \tropP_M,
			\]
			viewed as a linear function on $\R^{\binom{[n]}{k}}$.  
		\end{definition}
		
		For example, with $k = 3$ and $n = 7$, we have
		$$
			u^t_{236}(\tropP_\bullet) = -\tropP_{236} + \tropP_{246} + \tropP_{237} - \tropP_{247}.
		$$

		The set $\{\mathfrak{h}_J : J \in \binom{[n]}{k}\}$ is a basis for $\R^{\binom{[n]}{k}}$; the elements $\{\mathfrak{h}_J : J \text{ cyclic}\}$ span the lineality space $L_{k,n}$ (defined in \eqref{eq:lin}); and for non-cyclic $J$, the ray $\R_{\ge 0} \cdot \mathfrak{h}_J$ is a ray of $\Trop_{> 0} X(k,n)$ \cite{E22}.  See also \cite[Theorem 3.1]{ELnc}.
		
		\begin{theorem}[\cite{E20}]\label{thm:udual}
		For $J \in  \binom{[n]}{k}^{\ncyc}$, we have that $u^t_J$ vanishes on $L_{k,n}$.  For $I, J \in \binom{[n]}{k}^{\ncyc}$, we have the duality $u^t_J(\h_I) = \delta_{I,J}$.
		\end{theorem}

		Thus, for any $\tropP_\bullet \in \R^{\binom{[n]}{k}}$, we have the planar basis expansion 
		$$
		\tropP_\bullet \equiv \sum_{J \in \binom{[n]}{k}^{\ncyc}} u^t_J(\tropP_\bullet) \, \mathfrak{h}_J \pmod{L_{k,n}}.
		$$
		We may view this as giving sections of the quotient maps $\R^{\binom{[n]}{k}} \to \R^{\binom{[n]}{k}}/L_{k,n}$ and $\Trop \, \Gr(k,n) \to \Trop\, X(k,n)$.

				Let $\T^{k,n}:= \R^{(k-1)\times(n-k)}$ with natural basis $e_{i,j}$ for $(i,j) \in [k-1] \times [n-k]$.  For $J = \{j_1 < j_2 < \cdots < j_k\} \in \binom{[n]}{k}$, define the integer points 
	$$
	\vn_J := \sum_{i=1}^{k-1} \sum_{s = j_i - (i-1)}^{j_{i+1} - (i+1)} e_{i,s} \in \T^{k,n}.
	$$
Define the projection $\Psi: \R^{\binom{[n]}{k}} \to \T^{k,n}$ by $\mathfrak{h}_J \mapsto \mathfrak{t}_J$.  It is shown in \cite{ELnc} that the $\t_J$ for $J$ noncyclic can be interpreted as generators of the rays of the \emph{noncrossing fan}.  The map $\Psi$ inverts the positive parameterization of $\Trop_{>0}X(k,n)$.  In the following result we use the notion of noncrossing tableau.  We will review this definition in the case of $k = 3$ in \cref{sec:noncrossing}.

\begin{thm}[\cite{ELnc}]\label{thm:ELnc}
	The projection $\Psi$ descends to $\R^{\binom{[n]}{k}}/L_{k,n}$ and the restriction to $\Trop_{>0}X(k,n)$ is a bijection.  This induces a bijection between integer points of $\Trop_{>0}X(k,n)$ and the set of cyclic-less noncrossing tableaux: for an integer point $\tropP_\bullet \in \Trop_{>0}X(k,n)$ we have
	 $$
	 \Psi(\tropP_\bullet) = \sum_{J \in \J} \t_J
	 $$
	 for a unique cyclic-less noncrossing tableau $\J = (J_1,J_2,\ldots,J_r)$.
\end{thm}

In the rest of the paper we will use these results in the case $k = 3$.
		
		\section{Planar basis expansion}
		\def\tn{{\tilde n}}
		
		The aim of this section is to prove \cref{thm:pbexpansion}.
		
		\begin{example}\label{example: standardization 310 to 321}
			This example illustrates that the expansion of \cref{thm:pbexpansion} can have multiplicities when $Q$ is not standard.  Consider the non-standard noncrossing tableau
			$$\mathcal{J} = \{(1,6,7),(1,6,8),(2,3,10),(2,3,10),(2,5,9),(3,4,9),(7,9,10)\}$$
			with $\t$-vector
			$$\t_\mathcal{J} = \t_{1,6,7}+\t_{1,6,8}+2 \t_{2,3,10}+\t_{2,5,9}+\t_{3,4,9}+\t_{7,9,10}$$
			The planar basis expansion of the Pl\"ucker function is 
			\begin{eqnarray}\label{eq: 310 ray}
				\tropP_\bullet(W(\J)) & = & -2 \h_{1,3,9}+2 \h_{1,3,10}-\h_{1,4,8}+\h_{1,4,9}-\h_{1,5,7}+\h_{1,5,9}+\h_{1,6,7}+\h_{1,6,8}\\
				& & +2 \h_{2,3,9}-\h_{2,4,7}+\h_{2,4,8}+\h_{2,5,7}+\h_{3,4,7}-\h_{6,8,10}+\h_{6,9,10}+\h_{7,8,10}. \nonumber
			\end{eqnarray}
			Looking ahead to \cref{sec:noncrossing} and applying \cref{prop: standardization}, the standardization of $\mathcal{J}$ is given by 
			$$\{(1,12,13),(2,11,15),(3,10,16),(4,7,20),(5,6,21),(8,9,17),(14,18,19)\}.$$	
		\end{example}
		
		\subsection{Proof of \cref{thm:pbexpansion}}
		Our proof of \cref{thm:pbexpansion} is based on the following computation of tropical cross-ratios.
		\begin{proposition}\label{prop:cross-ratio-formula}
			Let $(Q,\z)$ be a normal CAT(0) planar graph. For any triple $(i, j, k)$ of boundary indices, we have
			\begin{equation}\label{eq:diff}
			u^t_{ijk}(\tropP_\bullet(Q,\z)) = \#  \text{white strand triples equal to $i,j,k$} -   \#  \text{black strand triples equal to $i,j,k$}.		\end{equation}
		\end{proposition}
		
					By \cref{thm:udual}, the planar basis elements $\{\mathfrak{h}_J : J \in \binom{[n]}{3}^{\text{ncyc}}\}$ 
			satisfy $u^t_I(\mathfrak{h}_J) = \delta_{I,J}$. Therefore, the coefficient of $\mathfrak{h}_{ijk}$ 
			in the planar basis expansion of $\tropP_\bullet(Q,\z)$ equals $u^t_{ijk}(\tropP_\bullet(Q,\z))$. By 
			\cref{prop:cross-ratio-formula}, this coefficient is given by \eqref{eq:diff}, and  \cref{thm:pbexpansion} follows.

		\subsection{Reduction to standard case}
		
		We will reduce \cref{prop:cross-ratio-formula} to the statement for a standard CAT(0) planar graph.
		
		\begin{proposition}\label{prop:cross-ratio-formula-standard}
			Let $Q$ be a standard CAT(0) planar embedded graph with boundary vertices 
			$z_1, \ldots, z_n$. For any triple $(i, j, k)$ of boundary indices:
			\[
			u^t_{ijk}(\tropP_\bullet(Q)) = 
			\begin{cases}
				+1 & \text{if } (i,j,k) \text{ is the strand triple of a white triangle}, \\
				-1 & \text{if } (i,j,k) \text{ is the strand triple of a black triangle}, \\
				0 & \text{if } (i,j,k) \text{ is not a strand triple of any triangle}.
			\end{cases}
			\]
		\end{proposition}

		\begin{proof}
			By \cref{lem:strand-bijection}, for any triple $(i, j, k)$, there is 
			at most one interior triangle $T$ with strand triple $(i, j, k)$.
			If no such triangle exists, then $(i, j, k)$ is not a strand triple, and by 
			\cref{prop:non-strand-vanishing}, $u^t_{ijk}(\tropP_\bullet(Q)) = 0$.
			If such a triangle $T$ exists, it is unique, and by Corollary~\ref{cor:cross-ratio}:
			\[
			u^t_{ijk}(\tropP_\bullet(Q)) = 
			\begin{cases}
				+1 & \text{if } T \text{ is white (counterclockwise-oriented)}, \\
				-1 & \text{if } T \text{ is black (clockwise-oriented)}.
			\end{cases} \qedhere
			\]
		\end{proof}
		\cref{cor:cross-ratio} and \cref{prop:non-strand-vanishing} are proved later in this section.
		
	\subsection{Proof of \cref{prop:cross-ratio-formula}}

		Let $W$ be a normal web.  Let $W'$ be obtained from $W$ by splitting high degree boundary vertices into degree one vertices, and let $W'' = \std(W)$ be obtained from $W'$ by removing boundary vertices with degree $0$.  Thus $W''$ is the standardization of $W$.  Let $(Q,\z) = (Q(W),\z(W))$, $(Q',\z') = (Q(W'),\z(W'))$, and $(Q'',\z'') = (Q(W''),\z(W''))$.   
		
		The proof of \cref{prop:cross-ratio-formula-standard} generalizes to webs where all boundary vertices have degree $\leq 1$.  We need only note that by \cref{def:tropical-cross-ratio}, $u_{i,j,k}^t(\tropP_\bullet(Q))$ vanishes unless $z_i \neq z_{i+1}$, and $z_j \neq z_{j+1}$, and $z_k \neq z_{k+1}$.  Thus \cref{prop:cross-ratio-formula-standard} holds as stated for $Q'$.
		
		Let $\iotaFInc: [n] \hookrightarrow [n']$ denote the inclusion such that $z_i = z'_{\iotaFInc(i)}$.
		Then the Pl\"ucker variables of $Q$ and $Q'$ are related by the formula
		$$
		\tropP_{I}(Q) = \tropP_{\iotaFInc(I)}(Q'), \qquad
		\text{or alternatively}, \qquad
		\tropP_\bullet(Q) = \iotaFInc(\tropP_\bullet(Q'))
		$$
		where $\iotaFInc : \R^{\binom{[n']}{3}} \to \R^{\binom{[n]}{3}}$ is given by
		$$
		\iotaFInc(e^I) = \begin{cases} 0 &\mbox{if $I$ is not contained in $\iotaFInc([n])$,} \\
			e^{\iotaFInc^{-1}(I)} & \mbox{otherwise.}
		\end{cases}
		$$ 
		Finally, by \cite[Section 3.4]{ELnc}, we have
		$$
		\iotaFInc(\h_{i,j,k}) = \h_{\iotaFInc^{-1}(i'),\iotaFInc^{-1}(j'),\iotaFInc^{-1}(k')}
		$$
		where $i' = i$ if $i \in \iotaFInc([n])$, otherwise $i'$ is the cyclic predecessor of $i$ inside $\iotaFInc([n])$, and similarly for $j'$ and $k'$.  Note that by normality, $\iotaFInc^{-1}(i'),\iotaFInc^{-1}(j'),\iotaFInc^{-1}(k')$ are distinct.  Applying this to \cref{prop:cross-ratio-formula-standard}, we obtain \cref{prop:cross-ratio-formula}.

		\subsection{Strand triple computation}
		
		\begin{proposition}\label{prop:focalT}
			Suppose that $T$ is a triangle of $Q(W)$ with strand triple $i,j,k$.  Then for all triples $(a, b, c) \in C_{ijk}$ in the cubical array, at least one vertex of $T$ is a minimizer for $(z_a, z_b, z_c)$.
		\end{proposition}
		
		\begin{proof}
			If $T$ is a corner triangle, so two of $\{i,j,k\}$ are cyclically adjacent, then the statement is easy.  Suppose that $T$ is not a corner triangle, so $\{i,i+1,j,j+1,k,k+1\}$ are all distinct.  By \cref{thm:focal}, it suffices to show that at least one vertex of $T$ is a focal point.  The focal points can be determined using \cref{lem:strandstripF}.  For $T$ a white triangle, the choices of focal points are shown in Figure \ref{fig:uijk+1}.  For $T$ a black triangle, the choices of focal points are shown in Figure \ref{fig:uijk-1}.  
		\end{proof}
		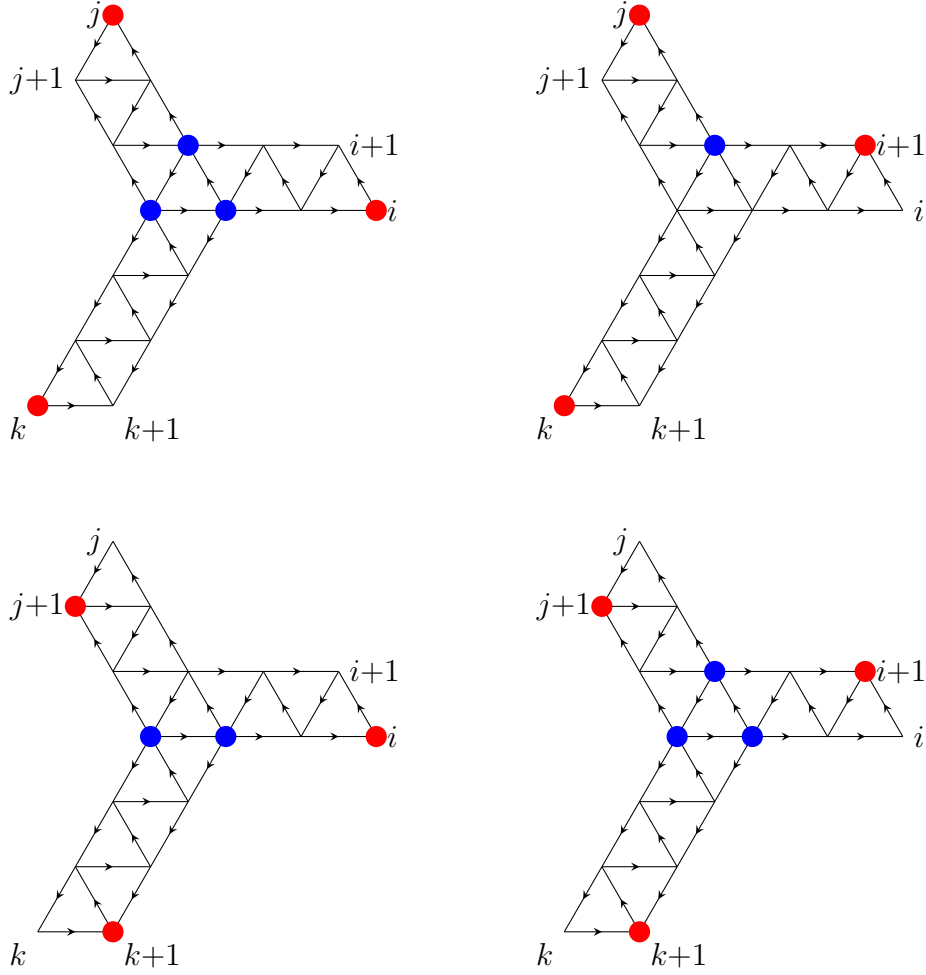
\begin{figure}[!h]
			\begin{tikzpicture}[scale=1.0]
				\tikzset{->-/.style={decoration={markings, mark=at position #1 with
							{\arrow[scale=1.0]{stealth}}},postaction={decorate}}, ->-/.default=0.5}
				
				\newcommand{\V}[2]{({(#1)+(#2)*0.5},{(#2)*sqrt(3)/2})}
				
				\newcommand{\drawYshape}[2]{
					\begin{scope}[shift={(#1,#2)}]
						\foreach \b in {0,...,8} {
							\pgfmathtruncatemacro{\amax}{9-1-\b}
							\foreach \a in {0,...,\amax} {
								\pgfmathtruncatemacro{\sab}{\a+\b}
								\pgfmathtruncatemacro{\Vab}{%
									((\a==2 || \a==3) && \b>=1 && \b<=5)
									|| ((\b==4 || \b==5) && \a>=2 && \a<=9-\b)
									|| ((\sab==6 || \sab==7) && \a>=0 && \a<=3)
									? 1 : 0}
								\pgfmathtruncatemacro{\spb}{\a+1+\b}
								\pgfmathtruncatemacro{\Vpb}{%
									((\a+1==2 || \a+1==3) && \b>=1 && \b<=5)
									|| ((\b==4 || \b==5) && \a+1>=2 && \a+1<=9-\b)
									|| ((\spb==6 || \spb==7) && \a+1>=0 && \a+1<=3)
									? 1 : 0}
								\pgfmathtruncatemacro{\saq}{\a+\b+1}
								\pgfmathtruncatemacro{\Vaq}{%
									((\a==2 || \a==3) && \b+1>=1 && \b+1<=5)
									|| ((\b+1==4 || \b+1==5) && \a>=2 && \a<=9-\b-1)
									|| ((\saq==6 || \saq==7) && \a>=0 && \a<=3)
									? 1 : 0}
								\pgfmathtruncatemacro{\drawH}{\Vab*\Vpb}
								\ifnum\drawH=1
								\draw[->-] \V{\a}{\b} -- \V{\a+1}{\b};
								\fi
								\pgfmathtruncatemacro{\drawL}{\Vab*\Vaq}
								\ifnum\drawL=1
								\draw[->-] \V{\a}{\b+1} -- \V{\a}{\b};
								\fi
								\pgfmathtruncatemacro{\drawD}{\Vpb*\Vaq}
								\ifnum\drawD=1
								\draw[->-] \V{\a+1}{\b} -- \V{\a}{\b+1};
								\fi
							}
						}
						\node[left]  at \V{0}{7} {$j$};
						\node[left]  at \V{0}{6} {$j{+}1$};
						\node[below left] at \V{2}{1} {$k$};
						\node[below right] at \V{3}{1} {$k{+}1$};
						\node[right] at \V{5}{4} {$i$};
						\node[right] at \V{4}{5} {$i{+}1$};
					\end{scope}
				}
				
				\def\dotrad{4pt}
				
				\drawYshape{0}{0}
				\begin{scope}[shift={(0,0)}]
					\fill[red]  \V{0}{7} circle (\dotrad);  
					\fill[red]  \V{2}{1} circle (\dotrad);  
					\fill[red]  \V{5}{4} circle (\dotrad);  
					\fill[blue] \V{2}{5} circle (\dotrad);  
					\fill[blue] \V{2}{4} circle (\dotrad);  
					\fill[blue] \V{3}{4} circle (\dotrad);  
				\end{scope}
				
				\drawYshape{7}{0}
				\begin{scope}[shift={(7,0)}]
					\fill[red]  \V{0}{7} circle (\dotrad);  
					\fill[red]  \V{2}{1} circle (\dotrad);  
					\fill[red]  \V{4}{5} circle (\dotrad);  
					\fill[blue] \V{2}{5} circle (\dotrad);  
				\end{scope}
				
				\drawYshape{0}{-7}
				\begin{scope}[shift={(0,-7)}]
					\fill[red]  \V{5}{4} circle (\dotrad);  
					\fill[red]  \V{0}{6} circle (\dotrad);  
					\fill[red]  \V{3}{1} circle (\dotrad);  
					\fill[blue] \V{2}{4} circle (\dotrad);  
					\fill[blue] \V{3}{4} circle (\dotrad);  
				\end{scope}
				
				\drawYshape{7}{-7}
				\begin{scope}[shift={(7,-7)}]
					\fill[red]  \V{0}{6} circle (\dotrad);  
					\fill[red]  \V{3}{1} circle (\dotrad);  
					\fill[red]  \V{4}{5} circle (\dotrad);  
					\fill[blue] \V{2}{5} circle (\dotrad);  
					\fill[blue] \V{2}{4} circle (\dotrad);  
					\fill[blue] \V{3}{4} circle (\dotrad);  
				\end{scope}
				
			\end{tikzpicture}
			\caption{Representative minimizers used in the calculation of $u^t_{i,j,k}=1$.  The triple of boundary red dots $(a,b,c) \in \{i,i+1\} \times \{j,j+1\} \times \{k,k+1\}$ indicates that we are computing $\Sigma_{a,b,c}$, while blue dots on the interior triangle are distance minimizers for $\Sigma_{a,b,c}$.  Note that minimizers can always be found on the triangle, but upper right and lower left minimizer sets do not overlap.  The distances are $\Sigma_{i,j,k} = D,\ \Sigma_{i+1,j,k} = D-2,\ \Sigma_{i,j+1,k+1} = D-1,\ \Sigma_{i+1,j+1,k+1} = D.$}	
			\label{fig:uijk+1}
		\end{figure}
		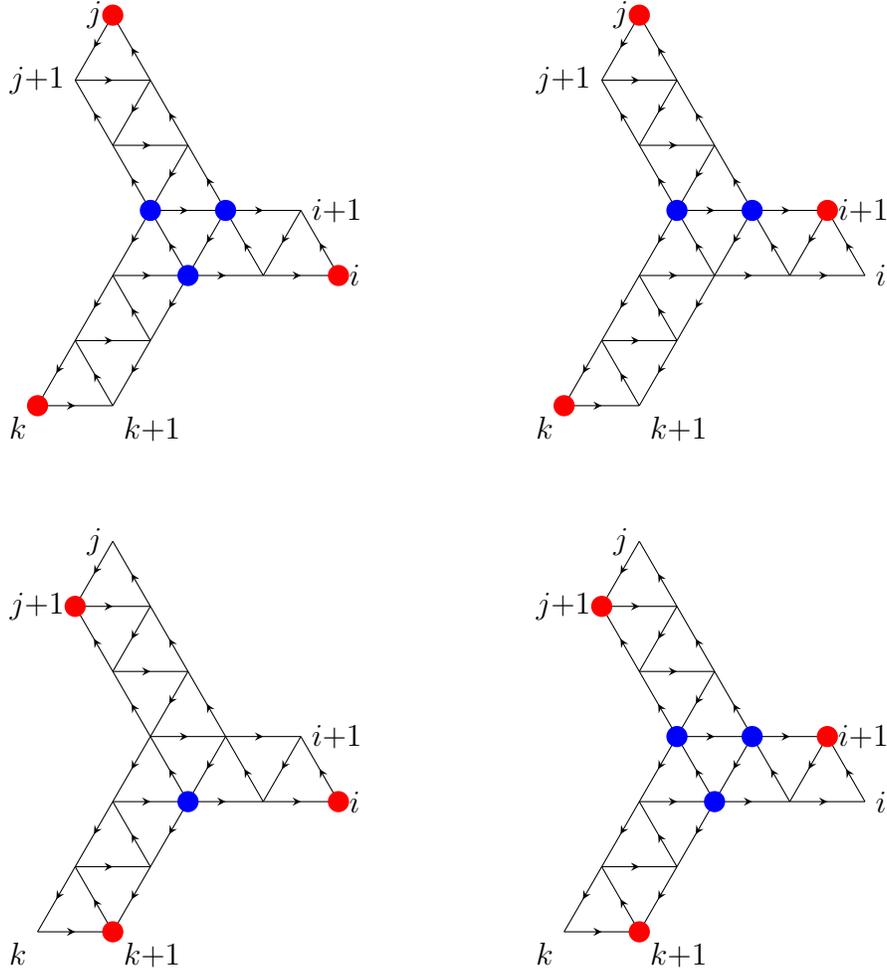
\begin{figure}[!h]
			\begin{tikzpicture}[scale=1]
				\tikzset{->-/.style={decoration={markings, mark=at position #1 with
							{\arrow[scale=1]{stealth}}},postaction={decorate}}, ->-/.default=0.5}
				
				\newcommand{\V}[2]{({(#1)+(#2)*0.5},{(#2)*sqrt(3)/2})}
				
				\newcommand{\drawYshapeCW}[2]{
					\begin{scope}[shift={(#1,#2)}]
						\foreach \b in {0,...,8} {
							\pgfmathtruncatemacro{\amax}{9-1-\b}
							\foreach \a in {0,...,\amax} {
								\pgfmathtruncatemacro{\sab}{\a+\b}
								\pgfmathtruncatemacro{\Vab}{%
									((\a==2 || \a==3) && \b>=2 && \b<=5)
									|| ((\b==4 || \b==5) && \a>=3 && \a<=9-\b)
									|| ((\sab==7 || \sab==8) && \a>=0 && \a<=2)
									? 1 : 0}
								\pgfmathtruncatemacro{\spb}{\a+1+\b}
								\pgfmathtruncatemacro{\Vpb}{%
									((\a+1==2 || \a+1==3) && \b>=2 && \b<=5)
									|| ((\b==4 || \b==5) && \a+1>=3 && \a+1<=9-\b)
									|| ((\spb==7 || \spb==8) && \a+1>=0 && \a+1<=2)
									? 1 : 0}
								\pgfmathtruncatemacro{\saq}{\a+\b+1}
								\pgfmathtruncatemacro{\Vaq}{%
									((\a==2 || \a==3) && \b+1>=2 && \b+1<=5)
									|| ((\b+1==4 || \b+1==5) && \a>=3 && \a<=9-\b-1)
									|| ((\saq==7 || \saq==8) && \a>=0 && \a<=2)
									? 1 : 0}
								\pgfmathtruncatemacro{\drawH}{\Vab*\Vpb}
								\ifnum\drawH=1
								\draw[->-] \V{\a}{\b} -- \V{\a+1}{\b};
								\fi
								\pgfmathtruncatemacro{\drawL}{\Vab*\Vaq}
								\ifnum\drawL=1
								\draw[->-] \V{\a}{\b+1} -- \V{\a}{\b};
								\fi
								\pgfmathtruncatemacro{\drawD}{\Vpb*\Vaq}
								\ifnum\drawD=1
								\draw[->-] \V{\a+1}{\b} -- \V{\a}{\b+1};
								\fi
							}
						}
						\node[left]  at \V{0}{8} {$j$};
						\node[left]  at \V{0}{7} {$j{+}1$};
						\node[below left]  at \V{2}{2} {$k$};
						\node[below right] at \V{3}{2} {$k{+}1$};
						\node[right] at \V{5}{4} {$i$};
						\node[right] at \V{4}{5} {$i{+}1$};
					\end{scope}
				}
				
				\def\dotrad{4pt}
				
				\drawYshapeCW{0}{0}
				\begin{scope}[shift={(0,0)}]
					\fill[red]  \V{0}{8} circle (\dotrad);  
					\fill[red]  \V{2}{2} circle (\dotrad);  
					\fill[red]  \V{5}{4} circle (\dotrad);  
					\fill[blue] \V{2}{5} circle (\dotrad);  
					\fill[blue] \V{3}{5} circle (\dotrad);  
					\fill[blue] \V{3}{4} circle (\dotrad);  
				\end{scope}
				
				\drawYshapeCW{7}{0}
				\begin{scope}[shift={(7,0)}]
					\fill[red]  \V{0}{8} circle (\dotrad);  
					\fill[red]  \V{2}{2} circle (\dotrad);  
					\fill[red]  \V{4}{5} circle (\dotrad);  
					\fill[blue] \V{2}{5} circle (\dotrad);  
					\fill[blue] \V{3}{5} circle (\dotrad);  
				\end{scope}
				
				\drawYshapeCW{0}{-7}
				\begin{scope}[shift={(0,-7)}]
					\fill[red]  \V{0}{7} circle (\dotrad);  
					\fill[red]  \V{3}{2} circle (\dotrad);  
					\fill[red]  \V{5}{4} circle (\dotrad);  
					\fill[blue] \V{3}{4} circle (\dotrad);  
				\end{scope}
				
				\drawYshapeCW{7}{-7}
				\begin{scope}[shift={(7,-7)}]
					\fill[red]  \V{0}{7} circle (\dotrad);  
					\fill[red]  \V{3}{2} circle (\dotrad);  
					\fill[red]  \V{4}{5} circle (\dotrad);  
					\fill[blue] \V{2}{5} circle (\dotrad);  
					\fill[blue] \V{3}{5} circle (\dotrad);  
					\fill[blue] \V{3}{4} circle (\dotrad);  
				\end{scope}
				
			\end{tikzpicture}
			\caption{Representative minimizers used in the calculation of $u^t_{i,j,k}=-1$ as in Figure \ref{fig:uijk+1}.  The distances are $\Sigma_{i,j,k} = D,\ \Sigma_{i+1,j,k} = D-2,\ \Sigma_{i,j+1,k+1} = D-4,\ \Sigma_{i+1,j+1,k+1} = D-3$.}
			\label{fig:uijk-1}
		\end{figure}
		
		\begin{lemma}
			\label{lem:main}
			Let $Q$ be a standard CCW-labeled CAT(0) planar graph with boundary vertices 
			$z_1, \ldots, z_n$. Let $T$ be a non-corner interior 
			triangular face of $Q$ with strand triple $(i, j, k)$, where $i, j, k$ are 
			pairwise non-adjacent in cyclic order. Let $D = \Sigma(i, j, k)$.
			
			\noindent
			{(1) If $T$ is black:}
			\begin{align*}
				\Sigma(i+1, j, k) &= \Sigma(i, j+1, k) = \Sigma(i, j, k+1) = D - 2, \\
				\Sigma(i+1, j+1, k) &= \Sigma(i+1, j, k+1) = \Sigma(i, j+1, k+1) = D - 4, \\
				\Sigma(i+1, j+1, k+1) &= D - 3.
			\end{align*}
			
			\noindent
			{(2) If $T$ is white:}
			\begin{align*}
				\Sigma(i+1, j, k) &= \Sigma(i, j+1, k) = \Sigma(i, j, k+1) = D - 2, \\
				\Sigma(i+1, j+1, k) &= \Sigma(i+1, j, k+1) = \Sigma(i, j+1, k+1) = D - 1, \\
				\Sigma(i+1, j+1, k+1) &= D.
			\end{align*}
		\end{lemma}
		\begin{proof}
			Direct computation using \cref{prop:focalT}.
		\end{proof}
		
		\begin{corollary}[Tropical Cross-Ratio for Strand Triples]
			\label{cor:cross-ratio}
			For a triangle $T$ with strand triple $(i, j, k)$:
			\[
			u^t_{ijk}(\tropP_\bullet(Q)) = 
			\begin{cases}
				-1 & \text{if } T \text{ is black (clockwise)}, \\
				+1 & \text{if } T \text{ is white (counterclockwise)}.
			\end{cases}
			\]
		\end{corollary}
		\begin{proof}
			For a corner triangle $T$, the result is straightforward.  We assume that $T$ is non-corner.
			Expanding the definition of $u^t_{ijk}$ and using \cref{lem:main}, we have
			\begin{align*}
				u^t_{ijk} 
				&= -\frac{1}{3}\bigl[-\Sigma_{ijk} + \Sigma_{i+1,j,k} + \Sigma_{i,j+1,k} + \Sigma_{i,j,k+1} \\
				&\quad\qquad - \Sigma_{i+1,j+1,k} - \Sigma_{i+1,j,k+1} - \Sigma_{i,j+1,k+1} + \Sigma_{i+1,j+1,k+1}\bigr].
			\end{align*}			
			For a black triangle, this gives			
			\begin{align*}
				u^t_{ijk} &= -\frac{1}{3}\bigl[-D + 3(D-2) - 3(D-4) + (D-3)\bigr] = - 1
			\end{align*}
			For a white triangle, this gives				
			\begin{align*}
				u^t_{ijk} &= -\frac{1}{3}\bigl[-D + 3(D-2) - 3(D-1) + D\bigr] = +1. \qedhere
			\end{align*}
		\end{proof}

		\subsection{Non-Strand Triples: Vanishing of the Tropical Cross-Ratio}

		\begin{lemma}
			\label{lem:linearity-vanishing}
			If $\Sigma$ is tropically linear over $\{i, i+1\} \times \{j, j+1\} \times \{k, k+1\}$, that is, $\Sigma(a,b,c) = D + (a-i)\delta_1 + (b-j)\delta_2 + (c-k)\delta_3$ for $(a,b,c)$ in the cube, then $u^t_{ijk}(\tropP_\bullet(Q)) = 0$.
		\end{lemma}
		
		\begin{proof} 
			Clear from definition.
		\end{proof}
		
\begin{proposition}
	\label{prop:non-strand-vanishing}
	Let $Q$ be a standard CAT(0) planar graph with boundary vertices $z_1, \ldots, z_n$. Let $(i,j,k)$ be any triple of boundary indices that is \emph{not} a strand triple of any interior triangle of $Q$. Then $u^t_{ijk}(\tropP_\bullet(Q)) = 0$.
\end{proposition}

\begin{proof}
Let $C:= D_{ijk}$ as defined in \cref{def:tropical-cross-ratio}.  In the following, we will generally make the assumption that $i,j,k$ are pairwise non-adjacent cyclically and $C = \{i,i+1\}\times\{j,j+1\}\times\{k,k+1\}$ has size eight.   (The case $j = i+1$ is simpler as there are only four terms in $u_{ijk}^t$ and for the purposes of the below case analysis it behaves as if $x^*$ never lies on the strand $\eta_i$.)

Let $x^*$ be a focal point for $(i,j,k)$, and let $F_a:= F(x^*,z_a)$.  By following a path towards a witness while staying within the set of focal points for $x^*$, we may assume that 
\begin{equation}\label{eq:Fsame}
\mbox{no two of $F_i,F_j,F_k$ coincide.}
\end{equation}

Case (1).
Suppose that $x^*$ does not lie on any of the strand strips $S_i, S_j, S_k$.  Then by \cref{lem:strandstripF}(1) it follows that $F(x^*,z_i) = F(x^*,z_{i+1})$ and similarly for $j$ and $k$.  Using \eqref{eq:Fsame}, it follows that $x^*$ is a focal point for all eight triples $(z_a,z_b,z_c)$ with $(a,b,c)\in\{i,i+1\}\times\{j,j+1\}\times\{k,k+1\}$.  The claim then follows from \cref{lem:linearity-vanishing}. 

\medskip

Henceforth, we assume that $x^*$ lies on at least one of the strand strips.

For $a \in \{i,j,k\}$, the \emph{$a$-line} (resp. \emph{$(a+1)$-line}) is the path on the boundary of the strand strip $S_a$ that ends at $z_a$ (resp. $z_{a+1}$).
For each $a \in \{i,j,k\}$, we have that $x^*$ either lies on the $a$-side or $a+1$-side of the strand strip $S_a$.  If $x^*$ lies on a strand strip $S_a$ then one of the two sets $\{F(x^*,z_a),F(x^*,z_{a+1})\}$ has size one, denoted $F_1(a)$, and the other has size two, denoted $F_2(a)$.  We have $F_1(a) = \{v\} \subsetneq F_2(a)$ where $e = (v \to x^*)$ is an incoming edge called the \emph{spine}.  Note that the spine always goes from one side of the strand strip to the other.  

Let $\T$ denote the set of triangles with $x^*$ as a vertex.  If $x^*$ lies on $S_a$, then $\T \cap S_a$ consists of exactly three triangles, unless $x^*$ lies on the boundary $\partial Q$.  In the following, for brevity we will assume that $x^*$ does not lie on $\partial Q$; the case analysis we do still works even if $x^* \in \partial Q$.  If $x^*$ lies on both $S_a$ and $S_b$, we say that the strips are opposed if $x^*$ lies (on the $a$-line and $(b+1)$-line) or (on the $(a+1)$-line and the $b$-line).

If $x^*$ lies on two strand strips $S_a$ and $S_b$, and the two strips are opposed, then $\T \cap S_a \cap S_b$ consists of either $0$ or $2$ triangles.  In the case of $2$ triangles, we have that the spines of $S_a$ and $S_b$ (at $x^*$) coincide.  If $x^*$ lies on two strand strips $S_a$ and $S_b$ that are not opposed, then $\T \cap S_a \cap S_b$ consists of either $0$ or $1$ triangles.  In the case of $1$ triangle, the spines of $S_a$ and $S_b$ are cyclically adjacent among the incoming edges to $x^*$.

Case (2).  Suppose that $x^*$ belongs to two strips, say $S_i$ and $S_j$, with the same spine $e = (v \to x^*)$.  After possibly swapping $i$ and $j$, we may assume that $x^*$ lies on the $i$-line and on the $(j+1)$-line.  Using the assumption that $(i,j,k)$ is not a strand triple, we may check directly that $x^*$ is a focal point for the six triples in $C \setminus \{(i+1,j,k),(i+1,j,k+1)\}$ and that $v$ is a focal point for the two triples $\{(i+1,j,k),(i+1,j,k+1)\}$.  Furthermore, $\dist(x^*,z_k)-\dist(x^*,z_{k+1}) = \dist(v,z_k) - \dist(v, z_{k+1})$.  A direct calculation similar to \cref{cor:cross-ratio} shows that $u^t_{ijk} = 0$.  

\medskip

Henceforth, we assume that spines of strand strips that $x^*$ lies on are distinct.  

Case (3). If $x^*$ lies on all three strand strips with distinct spines, we deduce that $x^*$ is a focal point for all triples in $C$.  The claim then follows from \cref{lem:linearity-vanishing}. 
Note that it could be the case that two (out of six) of the $F$-sets coincide; this can only happen if, for example, $j = i+1$.  In this case, we find that we can choose $x^* = j$ and it is a focal point for all triples in $C$.

\medskip
Henceforth, we assume that $x^*$ lies on one or two strand strips.

Case (4a).  
Suppose that $x^*$ lies on $S_i$ and $S_j$ which may or may not be opposed, and does not lie on $S_k$.  Suppose first that $|\T \cap S_i \cap S_j| = 0$.  Then $x^*$ is a focal point for all triples in $C$ unless $F(x^*,z_k)$ coincides with one of the spine edges $e = (v \to x^*)$, say that of $S_i$.  In the latter case, $v$ is a focal point for all triples in $C$ (even if $v$ lies on $S_k$).  The claim then follows from \cref{lem:linearity-vanishing}. 

Case (4b).  
Now suppose that $|\T \cap S_i \cap S_j| = 1$ (which implies that $S_i$ and $S_j$ are not opposed).  Then in addition to the above case, we can also have that $F(x^*,z_k)$ consists of \emph{both} spine edges of $S_i$ and $S_j$.  Let $(v^* \to w)$ be the outgoing edge between the two spine edges.  Then $w$ is a focal point for all triples in $C$.  The claim then follows from \cref{lem:linearity-vanishing}. 

\medskip
Henceforth, we assume that $x^*$ lies on exactly one of the three strand strips.  Suppose that $x^*$ lies on $S_i$ but not $S_j$ or $S_k$.

Case (5a).  Suppose that $x^*$ lies on the $(i+1)$-line.  It is straightforward to see that we have a SDR at $x^*$ for any of the triples in $C$.  We need to consider the possibility that $F(x^*,z_{i+1}) = F(x^*,z)$ has size two and $z_{i+1},z$ are not parallel at $x^*$, where $z \in \{z_j,z_{j+1},z_k,z_{k+1}\}$.  But if we follow a path to the witness we conclude that $z = z_{i+1}$ (so $i$ and one of $\{j,k\}$ is cyclically adjacent).  In this case, $z_{i+1}$ is a focal point for all triples in $C$.  The claim then follows from \cref{lem:linearity-vanishing}. 

Case (5b).  Suppose that $x^*$ lies on the $i$-line.  If $F(x^*,z_i) = F(x^*,z)$ has size two and $z_i, z$ are not parallel at $x^*$, the situation is similar to Case 5(a).  Otherwise, either $x^*$ is a focal point for all triples in $C$, or one of $\{F(x^*, z_j), F(x^*,z_k)\}$ is equal to $v$ where $e = (v \to x^*)$ is the spine of $S_i$.  In this last case, $v$ is a focal point for all triples in $C$.  The claim then follows from \cref{lem:linearity-vanishing}. 

This completes the proof.
\end{proof}
\begin{figure}[h!]
	\centering
	\includegraphics[width=1\linewidth]{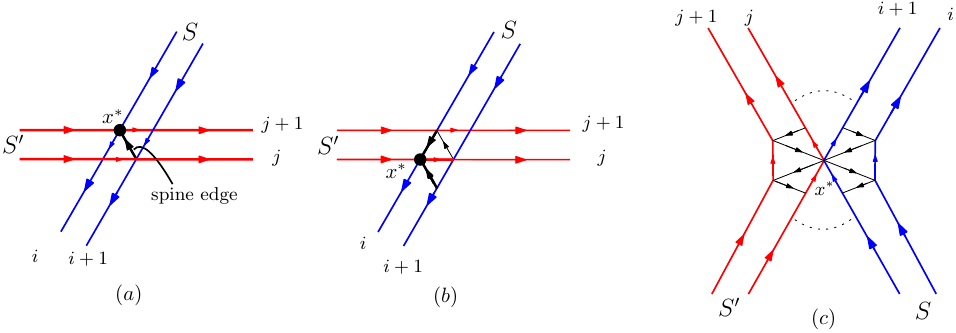}
	\caption{Three of the cases appearing in the proof of \cref{prop:non-strand-vanishing}.  In (a), $S$ and $S'$ are opposed and share the same spine edge; we have $\vert \mathbb{T}\cap S\cap S'\vert=2$.  In (b), $S$ and $S'$ are not opposed; the spine edges are cyclically adjacent; we have $\vert \mathbb{T}\cap S\cap S'\vert = 1$.  In (c), one of several possibilities where $\vert \mathbb{T}\cap S\cap S'\vert =0$.}
	\label{fig:prop16}
\end{figure}

		\part{Noncrossing tableaux}\label{part:noncrossing}
		\section{Noncrossing tableau}\label{sec:noncrossing}
		
		Two pairs $(a < b), (a'< b') \in \binom{[n]}{2}$ are \emph{noncrossing} if the interiors of the arcs $a \to b$ and $a' \to b'$ are non-intersecting when the numbers $1,2,\ldots,n$ are arranged in order on the boundary of a circle.
		Two triples $I = \{a < b <c\}, I' = \{a'<b'<c'\} \in \binom{[n]}{3}$ are \emph{noncrossing} if 
		\begin{enumerate}
			\item $(a,b)$ and $(a',b')$ are noncrossing, and
			\item $(b,c)$ and $(b',c')$ are noncrossing, and
			\item if $b = b'$ then $(a,c)$ and $(a',c')$ are noncrossing.
		\end{enumerate}
		
		\noindent A \emph{noncrossing tableau} $\J = (J_1,J_2,\ldots,J_r)$ is a pairwise noncrossing collection of triples $J_i \in \binom{[n]}{3}$.  The tableau $\J$ is standard if $[n] = \bigsqcup_i J_i$.  In particular, if $\J$ is standard then $3r = n$.  The \emph{content} $c(\J) = (c_1,c_2,\ldots)$ of a noncrossing tableau is the composition of $3r$ where $c_i$ counts the number of occurrences of $i$ in $\J$.  A noncrossing tableau $\J = (J_1,J_2,\ldots,J_r)$ is \defn{cyclic-less} if none of the $J_i$ is a cyclic triple.
		
		\subsection{Standardization}
		We define the \emph{standardization} of a noncrossing tableau $\J = (J_1,J_2,\ldots,J_r)$.  Let $J = \{a< b<c\}$.  We call $a$ the \emph{left endpoint}, $b$ the \emph{middle}, and $c$ the \emph{right endpoint}.  If $a \in [n]$ is a letter used in multiple $J_i$, we first move all left endpoints at $a$ slightly to the right, and all right endpoints at $a$ slightly to the left, in such a way that locally does not create crossings.  For example, if we have $(a,b,c)$ and $(a,b',c')$ where $b < b'$ then we would change the triples to $(a+s\epsilon, b, c)$ and $(a+t\epsilon, b',c')$ for $\epsilon > 0$ where $t < s$.
		
		After doing this, only middle points can coincide.  Among all the triples with a fixed middle $b$, we take the left endpoints and order them $a_1 < a_2 < \cdots < a_r$ and the right endpoints $c_1 < c_2 < \cdots < c_r$.  The standardization consists of the triples 
		$$
		(a_1,b_1,c_1), (a_2,b_2,c_2), \ldots, (a_r,b_r,c_r),
		$$
		where $b_i$ are all close to $b$ and satisfy $b_1 > b_2 > \cdots > b_r$.  Finally, to complete the standardization, we move the points to the integer locations $1,2,\ldots,3r$, keeping their relative order.  The following result is a direct check.
		
		\begin{proposition}\label{prop: standardization}
			The standardization $\std(\J)$ is a noncrossing tableau.
		\end{proposition}
		
		\begin{example}
		Let $\J = (134,135,234,234,356)$.  Then $\std(\J) = (18a,27b,36c,45d,9ef)$, where we use the letters $\{a,b,c,d,e,f\}$ instead of $\{10,11,12,13,14,15\}$.		
		\end{example}
		
		\begin{proposition}
			A noncrossing tableau $\J$ can be recovered from the pair $(\std(\J),c(\J))$ consisting of its standardization and content.
		\end{proposition}
		\begin{proof}
		First note that standardization does not change whether a letter is a left endpoint, middle, or right endpoint.  
		Suppose that $\J$ uses letters in $[n]$ and $\std(\J)$ uses the numbers in $[n']$.  We have a map $\iotaStd:[n'] \to [n]$, given by $\iotaStd([c_1+\cdots+c_{i-1} + 1, c_1 + \cdots + c_i]) = i$.  For each $b \in [n]$, take the triples $(a_1,b_1,c_1),\ldots,(a_r,b_r,c_r) \in \std(\J)$ satisfying $\iotaStd(b_j) = b$ and $b_1 > b_2 > \cdots > b_r$ and therefore $a_1 < a_2 < \cdots$ and $c_1 < c_2 < \cdots$.  Then $\J$ contains the triples 
		$$
		(\iotaStd(a_1),b,\iotaStd(c_r)), \ldots, (\iotaStd(a_r),b, \iotaStd(c_1)).
		$$
Taking the union over $b \in [n]$ we recover $\J$.
		\end{proof}
		
		\subsection{Bijection with semistandard Young tableaux}
		We draw semistandard Young tableaux (SSYT) in English notation, so that columns are strictly increasing as we go down, and rows are weakly increasing as we go from left to right.
		
		We describe a bijection $\Phi$ between SSYT $T$ of shape $3 \times r$ with letters in $[n]$ and noncrossing tableaux $\J = (J_1,\ldots,J_r)$ with $J_i \in \binom{[n]}{3}$.  For each row $R$ that is not the first row, starting from the left we connect each letter $\ell$ to the rightmost letter in the row above it that is $< \ell$ and that has not been connected to a letter in row $R$.  By semistandardness, at every step such a letter always exists.  The connected triples $J_1,J_2,\ldots,J_r$ form a noncrossing tableau $\J = \Phi(T)$.
		
		\begin{proposition}
			The map $\Phi$ is a bijection between $3 \times r$ SSYT with letters in $[n]$ and noncrossing tableaux $(J_1,\ldots,J_r)$ with $J_i \in \binom{[n]}{3}$.
		\end{proposition}
		\begin{proof}
			Given a noncrossing tableau $\J$, define rows $R_1,R_2,R_3$ obtained by taking the left, middle, right endpoints, respectively, of the $J_i$ and sorting.  Since left endpoints are smaller than middle endpoints are smaller than right endpoints, it is clear that $\Psi(\J):= (R_1,R_2,R_3)$ is a semistandard Young tableau, and that $\Psi\circ \Phi(T) = T$.  We need to show that $\Phi \circ \Psi(\J) = \J$.  Let $J_1 \in \J$ be the triple with the smallest right endpoint, and if there is a tie, the largest left endpoint, and if there is another tie, the smallest middle endpoint.  By induction on $r$, it is enough to show that $J_1 = (a < b < c)$ can be unambiguously recovered from $\Psi(\J) = (R_1,R_2,R_3)$.  Indeed, we claim that $J_1$ is the triple connected to the leftmost letter in $R_3$ in the construction of $\Phi \circ \Psi(\J)$.
			
			Let $J' = (a' < b' < c')$ be another triple in $\J$.  Since $c' \geq c$, by the noncrossing condition and our assumptions, we have $b' \leq b$ or $b' \geq c$.  We conclude that $b$ is largest among the letters in $R_2$ that are less than $c$.  Again by the noncrossing condition, we deduce that
			\begin{itemize}
				\item
				if $b' \leq b$ then either $b' < a$ or $a \leq a' < b'$,   
				\item
				if $c \leq b' < c'$ then either $a' < a$ or $a' > b$.
			\end{itemize}
			We compare this to the construction of $\Phi \circ \Psi(\J)$ and we see that $J_1$ must be the triple connected to the leftmost letter, that is, $c$, in $R_3$.  This completes the proof.
		\end{proof}
		
		\section{Noncrossing tableaux and webs}
		\subsection{Tymoczko's construction}
		Tymoczko \cite{Tym} constructs a non-elliptic web $W(\J)$ from a standard noncrossing tableau $\J$.  Let $\J$ be a standard noncrossing tableau.  We replace each triple $J = (a<b<c)$ by a $\m$ drawn as a pair of semicircle arcs connecting $a$ and $b$, and connecting $b$ and $c$ on a horizontal line.  Any pair of $\m$-s are either non-intersecting, or in one of the following two positions, see \cite[Section 2.2]{Tym}.  This follows easily from the noncrossing condition.  
		\begin{equation}\label{eq:intersect}
		\begin{tikzpicture}[baseline=20pt]
		
		\draw (-2,0) arc (180:0:0.5) circle;
		\draw (-1,0) arc (180:0:1.5) circle;
		\draw[dashed] (-3,0) arc (180:0:1.5) circle;
		\draw[dashed] (0,0) arc (180:0:0.5) circle;
		\draw (-4,0)--(10,0);
		
		\draw[dashed] (4,0) arc (180:0:0.5) circle;
		\draw[dashed] (5,0) arc (180:0:1) circle;
		\draw (6,0) arc (180:0:1) circle;
		\draw (8,0) arc (180:0:0.5) circle;
		
		\end{tikzpicture}
		\end{equation}
		
		Note that any intersection between two arcs is an intersection between a left arc of a $\m$ with a right arc of another $\m$.  Thus there do not exist three arcs that all pairwise intersect, and consequently the $\m$-diagram is well-defined up to isotopy of arcs, as long as we insist that arcs intersect minimally.
		
		Next we replace each $\m$ by a $Y$, where the white vertex is placed very close to the horizon, 
		$$
		\begin{tikzpicture}
		
		\draw (0,0) arc (180:0:1) circle;
		\draw (2,0) arc (180:0:1) circle;
		\draw (-1,0)--(5,0);
		
		\draw (6,0)--(13,0);
		\draw[thick] (10,0) -- (10,0.3);
		
		\draw[thick] (8,0) arc (180:20:1) circle;
		\draw[thick] (12,0) arc (0:160:1) circle;
		\filldraw[white] (10,0.3) circle (2pt);
		\draw (10,0.3) circle (2pt);
		\end{tikzpicture}
		$$
		and then we resolve all crossings in the following way:
		$$
		\begin{tikzpicture}
		\draw (0,0)--(2,2);
		\draw (2,0)--(0,2);
		
		\draw[->] (3,1)--(5,1);
		
		\begin{scope}[shift={(6,0)}]
		\draw[thick] (0,0)-- (1,0.5)--(2,0);
		\draw[thick] (0,2) --(1,1.5)--(2,2);
		\draw[thick] (1,0.5)--(1,1.5);
				\filldraw[white] (1,1.5) circle (2pt);
		\draw (1,1.5) circle (2pt);
				\filldraw[black] (1,0.5) circle (2pt);
		\end{scope}
		\node at (9,1) {or};
		\begin{scope}[shift={(10,0)}]
		\draw[thick] (0,0)-- (1,0.5)--(2,0);
		\draw[thick] (0,2) --(1,1.5)--(2,2);
		\draw[thick] (1,0.5)--(1,1.5);
				\filldraw[white] (1,0.5) circle (2pt);
		\draw (1,0.5) circle (2pt);
				\filldraw[black] (1,1.5) circle (2pt);
		\end{scope}
		\end{tikzpicture}
		$$
		depending on whether the corresponding arcs are as in the left or the right configuration in \eqref{eq:intersect}.

		The following result is obtained from \cite{Tym} by composing with a bijection between standard tableaux and standard noncrossing tableaux; see \cite{PPS,SSW17}. 
		\begin{theorem}\label{thm:WJ}
			This procedure produces a non-elliptic web $W(\J)$, and is a bijection between standard noncrossing tableaux and standard webs.
		\end{theorem}
		
		\subsection{Semistandardization}\label{sec: semistandarization}
		We now semistandardize \cref{thm:WJ}.  (We note that our construction is distinct from Russell's construction \cite{Rus} for SSYT of shape $n \times 3$.)  Let $\K= \std(\J)$ where $\J =  (J_1,J_2,\ldots,J_r)$ is a noncrossing tableau.  Let $\iotaStd: [3r] \to [n]$ be the order-preserving map such that $c_i(\J) = |\iotaStd^{-1}(i)|$.  We define $W(\J)$ by applying $\iotaStd$ to the labels of the boundary vertices of $W(\K)$.
		
		\begin{theorem}\label{thm:WJnormal}
			When applied to a (possibly not standard) noncrossing tableau, the resulting web is normal.  Furthermore, we obtain a bijection between noncrossing tableaux and normal webs.
		\end{theorem}
		\begin{proof}
			First, we can trivially assume that we work with noncrossing tableaux that use all the indices of $[n]$, and normal webs with boundary labels $[n]$ such that boundary vertices have degree $\geq 1$.  Henceforth, we make this assumption.

			Let $W$ be a non-elliptic web with black boundary.  Then by \cref{prop:acuteweb}, $W$ is normal if and only if for any white vertex $w$ connected to two boundary vertices with labels $\ell, \ell'$, we have $\ell \neq \ell'$.  
			
			Let $W(\K)$ be the web constructed from a standard noncrossing tableau $\K$.  Suppose that $W(\K)$ has a white vertex $w$ connected to two adjacent boundary vertices.  Then there are two possibilities for $w$: (i) it came from the trivalent vertex of a $Y$ of a triple $(a,b,c)$, or (ii) it came from resolving the intersection between the strand $b-c$ and $a'-b'$ between triples $(a,b,c)$ and $(a',b',c')$ where $a < b < a' < c < b' < c'$; the two adjacent boundary vertices are $a'$ and $c$.
			
			Now, suppose that $\K$ arose from standardization of a (nonstandard) noncrossing tableau $\J$.  In case (i), we have that  $\iotaStd(a,b,c)$ is one of the triples of $\J$, and so all three labels $\iotaStd(a), \iotaStd(b), \iotaStd(c)$ are distinct.  In case (ii), we note that we cannot have $\iotaStd(a') = \iotaStd(c)$.  This is because $a'$ is a left endpoint while $c$ is a right endpoint and this would contradict the definition of standardization.  We conclude that the construction $\J \mapsto W(\J)$ has image in normal webs.
			
			To show that this map is bijective, we construct an inverse following Khovanov--Kuperberg \cite{KK} and Petersen--Pylyavskyy--Rhoades \cite{PPR}.  Using the \emph{depth map} (\cite[Section 2]{PPR}), we produce from the web $W(\K)$ a word in $\{1,0,-1\}$ using each letter $r$ times.  We obtain a SSYT $T$ by assigning the corresponding labels of $W(\J)$ to the first, second, or third rows of $T$ respectively.  Composed with the bijection $T \mapsto \Phi(T) =: \J$ we have the inverse map.
		\end{proof}
		
\subsection{Noncrossing ray identities}\label{sec:tJ-def}

Recall some definitions from \cref{sec: planar basis}.  Let $\T^{3,n}$ denote the space of $2 \times (n-3)$ real matrices, with standard basis vectors $e_{i,s}$ for $i \in \{1,2\}$ and $s \in \{1, \ldots, n-3\}$.  For a triple $J = \{j_1, j_2, j_3\}$ with $j_1 < j_2 < j_3$, we have the vectors $\t_J \in \T^{3,n}$ given by
	\[
	\t_J = \sum_{s=j_1}^{j_2-2} e_{1,s} + \sum_{s=j_2-1}^{j_3-3} e_{2,s}.
	\]
	Row~1 of $\t_J$ is the indicator of the interval $[j_1,\, j_2 - 2]$ and Row~2 is the indicator of $[j_2 - 1,\, j_3 - 3]$.  Note that $\t_J = 0$ when $J$ is a cyclic triple.  Recall that the linear projection $\Psi: \R^{\binom{n}{3}} \to \T^{3,n}$ is defined by $\Psi(\h_J) = \t_J$ for each planar basis element $\h_J$.  

\begin{lemma}\label{lem:4-term}
	The $\t_J$ vectors satisfy the following exchange identities.
	\begin{enumerate}
		\item If $a < a' \leq b-1$ and $b+1 \leq c < c'$, then
		\[
		\t_{a,b,c} + \t_{a',b,c'} = \t_{a',b,c} + \t_{a,b,c'}.
		\]
		\item If $a + 1 \leq b < b'$ and $b' + 1 \leq c < c'$, then
		\[
		\t_{a,b,c} + \t_{a,b',c'} = \t_{a,b,c'} + \t_{a,b',c}.
		\]
		\item If $a < a' \leq b-1$ and $ b < b' \leq c-1$, then
		\[
		\t_{a,b,c} + \t_{a',b',c} = \t_{a,b',c} + \t_{a',b,c}.
		\]
		\item If $a<a'<b<b'<c<c'$, then 
		\[
		\t_{a,b,c} + \t_{a',b',c'} = \t_{a,b,c'} + \t_{a',b',c} = \t_{a',b,c} + \t_{a,b',c'}= \t_{a,b',c} + \t_{a',b,c'},
		\]
		where only the rightmost term is noncrossing. 
	\end{enumerate}
\end{lemma}

\begin{corollary}\label{cor:difference-invariant}
	Let $a < b < c$ and $a' < b' < c'$ be triples in $\binom{[n]}{3}$.
	\begin{enumerate}
		\item If $b = b'$, then $\t_{a,b,c} - \t_{a,b,c'}$ depends only on $b$, $c$, and $c'$; that is, it is independent of $a$, within
		the range $a, a' \le b - 1$.
		\item If $a = a'$, then whenever $a + 1 \le b \le b'$ and $b' + 1 \le c, c'$, we have
		\[
		\t_{a,b,c} - \t_{a,b,c'} = \t_{a,b',c} - \t_{a,b',c'};
		\]
		that is, $\t_{a,b,c} - \t_{a,b,c'}$ is independent of the middle index.
		\item If $c = c'$, then whenever $a, a' \le b - 1$ and $b \le b' \le c - 1$, we have
		\[
		\t_{a,b,c} - \t_{a',b,c} = \t_{a,b',c} - \t_{a',b',c};
		\]
		that is, $\t_{a,b,c} - \t_{a',b,c}$ is independent of the middle index.
	\end{enumerate}
\end{corollary}

\subsection{Strand triples of resolution vertices}\label{sec:strand-triples-resolution}
Let $\J = (J_1,\ldots,J_r)$ be a noncrossing tableau without repeated indices and let $J =(a<b<c) \in \J$ be the triple with the largest right endpoint $c$.  Let $W' = W(\J\setminus J)$.  We consider the construction of $W(\J)$ by adding the $\m$-diagram $\m_J$ of $J$ to $W'$.

Two types of crossings between $\m_J$ and $W'$ arise:
\begin{itemize}
	\item \emph{Left-arc crossings:} the left arc of $\m_J$ crosses edges of $W'$ originally coming from right arcs of existing $\m$-diagrams.  These right arcs come from triples $(a',b',c')$ satisfying $a' < b' < a < c' < b < c$. 
	\item \emph{Right-arc crossings:} the right arc of $\m_J$ crosses edges of $W'$ originally coming from left arcs of existing $\m$-diagrams.  These left arcs come from triples $(a',b',c')$ satisfying $a'<a<b<b'<c'<c$.
\end{itemize}

In the following computation, the assumption that $c$ is the largest letter is crucial.

\begin{lemma}\label{lem:strand-triples}
	Suppose the left arc of $\m_J$ crosses $r$ edges of $W'$, where the $i$-th crossed edge is adjacent to a white vertex with strand triple $(a_i, b_i, c_i)$, with $a_1 < a_2 < \cdots < a_r$.  Suppose the right arc of $\m_J$ crosses $\ell$ edges of $W'$, where the $i$-th crossed edge is adjacent to a white vertex with strand triple $(a'_i, b'_i, c'_i)$, with $c'_1 < c'_2 < \cdots < c'_\ell$.  Then the resolution of all $r + \ell$ crossings creates the following vertices.
	\begin{enumerate}
		\item From left-arc crossings ($r+1$ white vertices, $r$ black vertices), all sharing right endpoint index $\gamma$:
		\begin{itemize}
			\item White strand triples:
			\[
			(a, c_r, \gamma), (b_r,c_{r-1}, \gamma), (b_{r-1},c_{r-2},\gamma), \ldots, (b_1,b,\gamma) 
			\]
			where $\gamma = c$ if $\ell = 0$, and $\gamma = c'_1$ if $\ell \geq 1$.  If $r = 0$ we instead get a single triple $(a,b,\gamma)$.
			\item Black strand triples: 
			\[
			(b_r, c_r, \gamma),\; (b_{r-1},c_{r-1}, \gamma), \ldots, (b_1, c_1, \gamma).
			\]
		\end{itemize}
		\item From right-arc crossings ($\ell$ white vertices, $\ell$ black vertices), all sharing middle $b$:
		\begin{itemize}
			\item White strand triples:
			\[
			(a'_1, b, c'_2), (a'_2, b, c'_3), \ldots, (a'_\ell, b, c).
			\]
			\item Black strand triples: \[(a'_1, b, c'_1), (a'_2, b, c'_2),\ldots,(a'_\ell, b, c'_\ell).\]
		\end{itemize}
	The strand triples of existing vertices are unchanged.
	\end{enumerate}
\end{lemma}

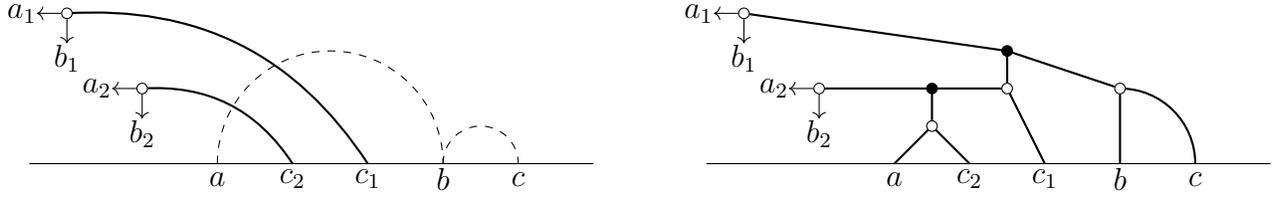
\begin{figure}
\begin{center}
\begin{tikzpicture}
\draw (3.5,0) -- (11,0);
\draw[dashed] (6,0) arc (180:0:1.5) circle;
\draw[dashed] (9,0) arc (180:0:0.5) circle;

\draw[thick] (5,1) to [bend left = 30](7,0);
\draw[thick] (4,2) to [bend left = 30](8,0);
\draw[->] (5,1) -- (5,0.6);
\draw[->] (5,1) -- (4.6,1);
\draw[->] (4,2) -- (4,1.6);
\draw[->] (4,2) -- (3.6,2);

\filldraw[white] (5,1) circle (2pt);
\draw (5,1) circle (2pt);
\filldraw[white] (4,2) circle (2pt);
\draw (4,2) circle (2pt);

\node at (5,0.4) {$b_2$};
\node at (4.4,1) {$a_2$};
\node at (4,1.4) {$b_1$};
\node at (3.4,2) {$a_1$};

\node at (6,-0.2) {$a$};
\node at (7,-0.2) {$c_2$};
\node at (8,-0.2) {$c_1$};
\node at (9,-0.2) {$b$};
\node at (10,-0.2) {$c$};

\begin{scope}[shift={(9,0)}]
\draw (3.5,0) -- (11,0);

\draw[->] (5,1) -- (5,0.6);
\draw[->] (5,1) -- (4.6,1);
\draw[->] (4,2) -- (4,1.6);
\draw[->] (4,2) -- (3.6,2);

\draw[thick] (5,1) -- (6.5,1) -- (6.5,0.5) -- (6,0);
\draw[thick] (6.5,0.5)--(7,0);
\draw[thick] (6.5,1)--(7.5,1) -- (8,0);
\draw[thick] (7.5,1)--(7.5,1.5)--(4,2);
\draw[thick] (7.5,1.5)--(9,1)--(9,0);
\draw[thick] (9,1) to [bend left = 45] (10,0);

\filldraw[white] (5,1) circle (2pt);
\draw (5,1) circle (2pt);
\filldraw[white] (4,2) circle (2pt);
\draw (4,2) circle (2pt);
\filldraw[white] (6.5,0.5) circle (2pt);
\draw (6.5,0.5) circle (2pt);
\filldraw[white] (7.5,1) circle (2pt);
\draw (7.5,1) circle (2pt);
\filldraw[white] (9,1) circle (2pt);
\draw (9,1) circle (2pt);

\filldraw[black] (6.5,1) circle (2pt);
\filldraw[black] (7.5,1.5) circle (2pt);

\node at (5,0.4) {$b_2$};
\node at (4.4,1) {$a_2$};
\node at (4,1.4) {$b_1$};
\node at (3.4,2) {$a_1$};

\node at (6,-0.2) {$a$};
\node at (7,-0.2) {$c_2$};
\node at (8,-0.2) {$c_1$};
\node at (9,-0.2) {$b$};
\node at (10,-0.2) {$c$};
\end{scope}
\end{tikzpicture}
\end{center}
\caption{Left: The left arc of $\m_J$ (drawn dashed) intersects a number of edges of the web $W'$ (drawn solid).  The strand triples of the white interior vertices on these edges are denoted $(a_i,b_i,c_i)$.  Right: After resolving we create $r+1$ new white vertices and $r$ new black vertices with strand triples described by \cref{lem:strand-triples}.  Here $r = 2$.}
\label{fig:resolve}
\end{figure}

\begin{proof}
Figure 1 shows the left arc crossings (in the representative case of $r=2$), from which the strand triples in (1) can be computed.  The right arc crossings can be recovered in a similar way.
\end{proof}

Let $\tropP_\bullet(W):= \tropP_\bullet(Q(W),\z(W))$ denote the tropical Pl\"ucker vector.

\begin{lemma}\label{lem:telescope}
In the above situation, we have 
	\[
	\Psi(\tropP_\bullet(W)) - \Psi(\tropP_\bullet(W')) = \t_{a,b,c}.
	\]
\end{lemma}

\begin{proof}
We apply \cref{thm:pbexpansion} and use the projection $\Psi(\h_J) = \t_J$.  We have 
$$
\Psi(\tropP_\bullet(W)) - \Psi(\tropP_\bullet(W')) = L + R
$$
where with $c_0 = b$ the contribution from \cref{lem:strand-triples}(1) is
$$
L = \t_{a,c_r,\gamma} + \sum_{i=1}^{r} \t_{ b_i,c_{i-1}, \gamma} - \sum_{i=1}^r \t_{b_i,c_i,\gamma}
$$
and with $c'_{\ell+1} = c$ the contribution from \cref{lem:strand-triples}(2) is
$$
R = \sum_{i=1}^\ell \t_{a'_i,b,c'_{i+1}} - \sum_{i=1}^\ell \t_{a'_i,b,c'_i}.
$$
For simplicity, we have assumed that $r > 0$ and $\ell > 0$.  When either $r=0$ or $\ell =0$ the computation is simpler.
By \cref{cor:difference-invariant}, 
$$
R = \sum_{i=1}^\ell \left( \t_{a,b,c'_{i+1}}- \t_{a,b,c'_{i}}\right) = \t_{a,b,c}-\t_{a,b,\gamma}.
$$
By \cref{lem:4-term}(3), and setting $b_0 = a$,
$$
L =  \t_{a,c_r,\gamma} + \sum_{i=1}^{r} \t_{ b_i,c_{i-1}, \gamma} - \sum_{i=1}^r \t_{b_i,c_i,\gamma} = \t_{a,c_r,\gamma} + \sum_{i=1}^r\left(-\t_{a,c_i,\gamma} +\t_{a,c_{i-1},\gamma}\right) = \t_{a,b,\gamma}.
$$
The claim follows.
\end{proof}

\subsection{The planar basis expansion of the web of a noncrossing tableau}\label{sec:nc-planar-basis}

\begin{theorem}\label{thm:nc-to-web}
	Let $\J = (J_1,J_2,\ldots,J_r)$ be a (possibly not standard) noncrossing tableau and $W(\J)$ be the normal web from \cref{thm:WJnormal}.  Then
	\[
	\Psi(\tropP_\bullet(W(\J))) = \sum_i \t_{J_i}.
	\]
\end{theorem}

\begin{proof}
Repeatedly applying \cref{lem:telescope}, we see that the statement holds for $\J$ a standard noncrossing tableau.

Now let $\J = (J_1,\ldots,J_r)$ be arbitrary and let $\K = (K_1,\ldots,K_r)$ be the standardization of $\J$, and let $\iotaStd: [n'] \to [n]$ be the corresponding map on labels.  Since $W(\J)$ is obtained from $W(\K)$ by applying $\iotaStd$ to the boundary labels, we have 
$$
\tropP_\bullet(W(\J)) = \iotaStd(\tropP_\bullet(W(\K)))
$$
where $\iotaStd(\h_{a,b,c}) = \h_{\iotaStd(a),\iotaStd(b),\iotaStd(c)}$.  Thus
$$
\Psi(\tropP_\bullet(W(\J)))  = \sum_i \t_{\iotaStd(K_i)}.
$$
It is not always the case that $\J = (\iotaStd(K_1),\iotaStd(K_2),\ldots,\iotaStd(K_r))$, but we do have $\sum_i \t_{\iotaStd(K_i)} = \sum_i \t_{J_i}$, using \cref{lem:4-term}.
\end{proof}

\begin{example}\label{example: 39}
	We illustrate \cref{thm:nc-to-web} in the case of a positive tropical Pl\"ucker vector $\tropP_\bullet$ arising from a standard CAT(0) planar graph.  Let
$$\tropP_\bullet = -\h_{136}+\h_{137}+\h_{146}+\h_{236}-\h_{379}+\h_{389}-\h_{469}+\h_{479}+\h_{569} \in \Trop_{>0}\Gr(3,9),$$
with normal model given in \cref{fig: 39 normal}.  We have
\begin{eqnarray*}
	\Psi(\tropP_\bullet) & = & \t_{146}+\t_{237}+\t_{589}.
\end{eqnarray*}
Compare to \cref{fig:resolve} with $(a_1,b_1,c_1) = (1,4,6)$ and $(a_2,b_2,c_2) = (2,3,7)$ and $(a,b,c) = (5,8,9)$.
	\begin{figure}[h!]
	\begin{tikzpicture}[scale=0.5]
		\tikzset{->-/.style={decoration={markings, mark=at position 0.55 with {\arrow{stealth}}},postaction={decorate}}}
		
		\pgfmathsetmacro{\h}{1.732}
		
		\coordinate (1) at (3, 3*\h);
		
		\coordinate (y) at (4, 2*\h);
		\coordinate (2) at (2, 2*\h);
		
		\coordinate (6) at (5, \h);
		\coordinate (w) at (3, \h);
		\coordinate (z) at (1, \h);
		
		\coordinate (5) at (6, 0);
		\coordinate (x) at (4, 0);
		\coordinate (4) at (2, 0);
		\coordinate (3) at (0, 0);
		
		\draw[->-] (5) -- (6);
		\draw[->-] (x) -- (5);
		\draw[->-] (4) -- (x);
		\draw[->-] (3) -- (4);
		\draw[->-] (6) -- (y);
		\draw[->-] (y) -- (1);
		\draw[->-] (1) -- (2);
		\draw[->-] (2) -- (z);
		\draw[->-] (z) -- (3);
		
		\draw[->-] (6) -- (x);
		\draw[->-] (x) -- (w);
		\draw[->-] (w) -- (4);
		\draw[->-] (4) -- (z);
		\draw[->-] (z) -- (w);
		\draw[->-] (w) -- (6);
		\draw[->-] (w) -- (2);
		\draw[->-] (y) -- (w);
		\draw[->-] (2) -- (y);
		
		\foreach \v in {1,2,3,4,5,6,x,y,w,z} {
			\fill[black!60] (\v) circle (2.5pt);
		}
		
		\node at ($(1)+(0,0.3)$) {\small $3$};
		\node at ($(2)+(-0.3,0.1)$) {\small $4$};
		\node at ($(3)+(-0.25,-0.15)$) {\small $6$};
		\node at ($(4)+(0,-0.25)$) {\small $7$};
		\node at ($(5)+(0.1,-0.25)$) {\small $9$};
		\node at ($(6)+(0.3,0)$) {\small $1$};
		\node at ($(x)+(0,-0.25)$) {\small $8$};
		\node at ($(y)+(0.3,0.1)$) {\small $2$};
		\node at ($(w)+(-0.15,-0.25)$) {\small $w$};
		\node at ($(z)+(-0.3,0)$) {\small $5$};
		
	\end{tikzpicture}
	\caption{The normal CAT(0) planar graph used in \cref{example: 39}.}
	\label{fig: 39 normal}
\end{figure}
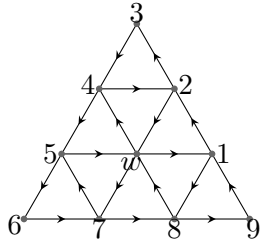
\end{example}

The two examples which follow illustrate various perspectives on planar basis expansions and normal models.  In the first example, we present the normal CAT(0) planar graph and web for the ray of $\Trop_{>0}\Gr(3,9)$ whose planar basis expansion was computed in \cite[Figure 1]{E21}.  The normal model is given in \cref{fig:bigray39} below.  In \cref{example: 310 to 321 continued}, we continue with the ray of $\Trop_{>0}G(3,10)$ from \cref{example: standardization 310 to 321}; this is the simplest ray for which its planar basis expansion has a coefficient which is not in $\{-1,0,1\}$, whose discovery provided a clue that some finer structure should be sought for; in \cref{fig:g310ppbexpansion} we present schematically its planar basis expansion via its normal model, where the multiplicities are resolved into $\pm1$.

\begin{figure}[h!]
\centering
\includegraphics[width=.9\linewidth]{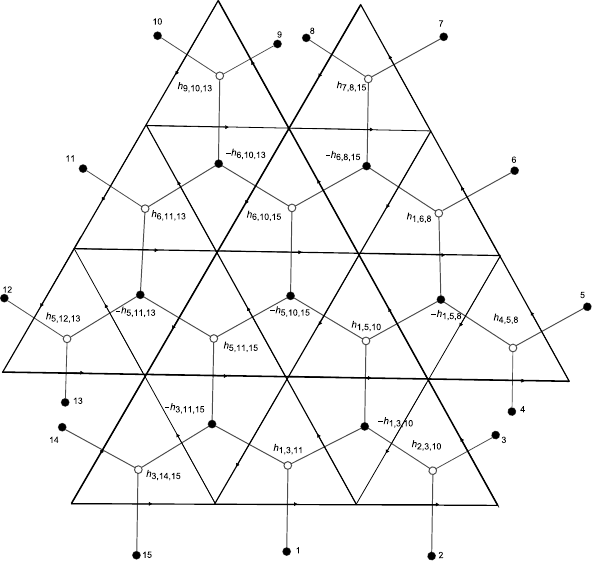}
\caption{Normal web, scaffold, and the planar basis expansion of the positive tropical Pl\"ucker vector $\tropP_\bullet(W(\K))$ corresponding to the (standard) noncrossing tableau $\K= \{(1,6,8),(2,3,11),(4,5,10),(7,14,15),(9,12,13)\}.$}
\label{fig:bigray39}
\end{figure}

\begin{example}\label{example: big315}
	Let us illustrate the whole section (and \cref{thm:WJ} and \cref{thm:WJnormal} in particular) with a ray for $\Trop_{>0}\Gr(3,9)$, standardized to a ray of $\Trop_{>0}\Gr(3,15)$.  See the bottom right entry in \cite[Figure 1]{E20} for the planar basis expansion modulo relabeling.
	
	The noncrossing tableau
	$$\J = \{(1,2,7),(1,4,5),(2,3,6),(4,8,9),(5,7,8)\},$$
	has standardization 
	$$\K= \{(1,6,8),(2,3,11),(4,5,10),(7,14,15),(9,12,13)\}.$$
	The planar basis expansion of $\tropP_\bullet(W(\K))$, the standard web, and its dual CAT(0) planar graph is given in \cref{fig:bigray39}.  Applying $\iotaStd$ as defined in the beginning of \cref{sec: semistandarization} and making the replacements
	$$1\to 1,\ 2\to 1,\ 3\to 2,\ 4\to 2,\ 5\to 3,\ 6\to 4,\ 7\to 4,8\to 5,$$
	$$9\to 5,\ 10\to 6,\ 11\to 7,\ 12\to 7,\ 13\to 8,\ 14\to 8,\ 15\to 9$$	
	we can recover the original positive tropical Pl\"ucker vector.  After many cancellations, we have
	\begin{eqnarray*}
		\tropP_\bullet(W(\J)) & = & -(\h_{135}+\h_{279}+\h_{369}+\h_{468})\\
		& &+  \h_{127}+\h_{136}+\h_{145}+\h_{235}+\h_{289}+\h_{379}+\h_{469}+\h_{478}+\h_{568},
	\end{eqnarray*}
	and we compute its projection to derive the noncrossing tableau:
	$$\Psi(\tropP_\bullet(W(\J))) = \mathfrak{t}_{127} + \mathfrak{t}_{145} + \mathfrak{t}_{236} + \mathfrak{t}_{489} + \mathfrak{t}_{578}.$$
\end{example}

\begin{example}\label{example: 310 to 321 continued}
We continue \cref{example: standardization 310 to 321}, specifically with the Pl\"ucker vector given in \eqref{eq: 310 ray}, whose planar basis expansion is depicted schematically in \cref{fig:g310ppbexpansion}.  Note that the map $\iotaStd: [21] \rightarrow[10]$ is rather nontrivial, resulting in significant redundancy in the planar basis expansion.  Note, for example, that $\pm \h_{1,3,9}$ appears four times, resolving the seemingly mysterious coefficient $-2$ \eqref{eq: 310 ray}, whose origin is now explained by the unique normal model representation. 

\end{example}

\begin{figure}[h!]
	\centering
	\includegraphics[width=0.8\linewidth]{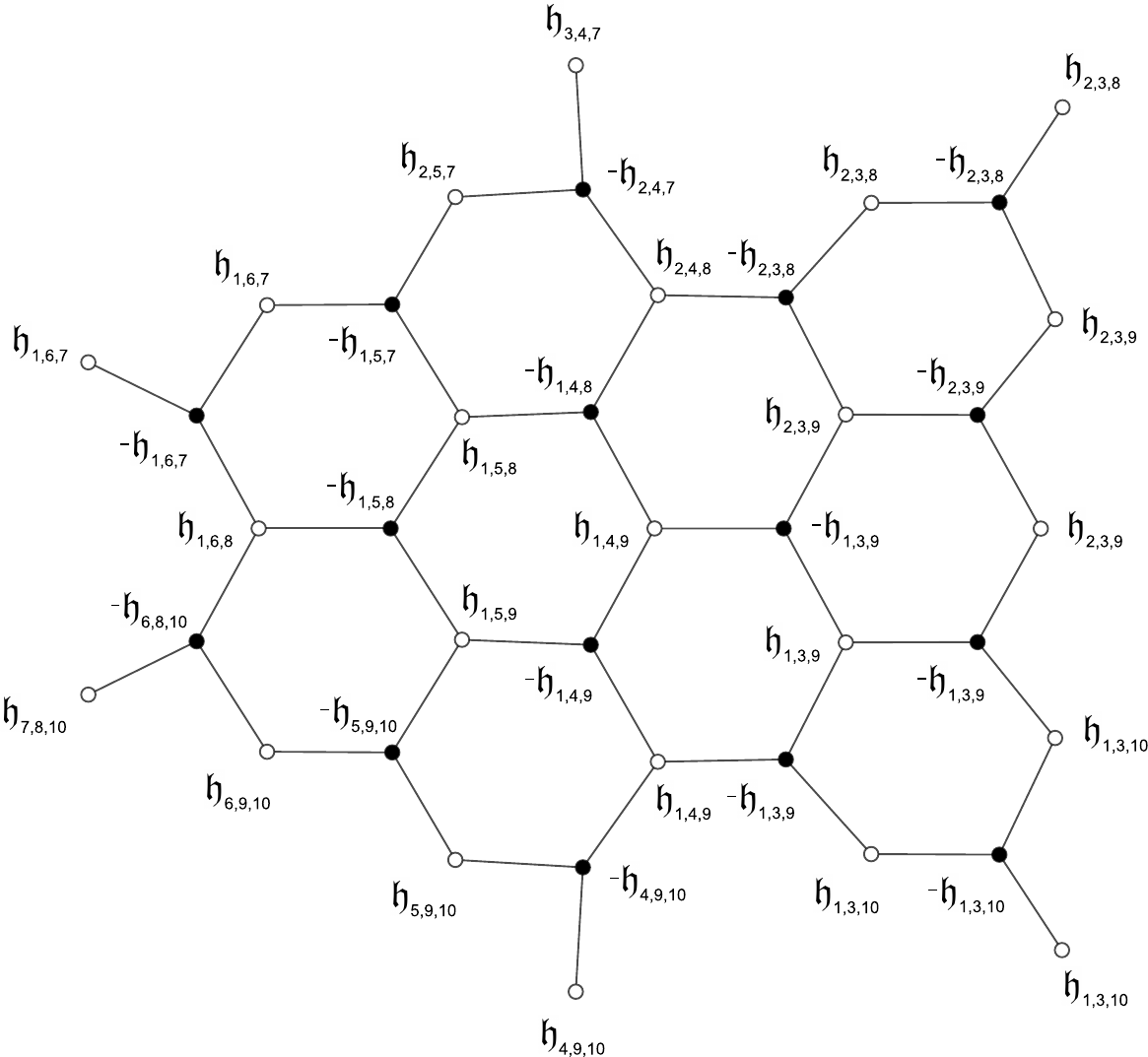}
	\caption{The planar basis expansion for the ray of $\Trop_{>0}\Gr(3,10)$ from \cref{example: 310 to 321 continued} corresponding to a (non-standard) noncrossing tableau, constructed by standardizing, computing the standard web for $\Trop_{>0}\Gr(3,21)$, taking the planar basis expansion and then reindexing.  Summing the nodes gives \cref{eq: 310 ray}, i.e. $\tropP_\bullet(W(\J)) = -2 \h_{1,3,9}+2 \h_{1,3,10}-\h_{1,4,8}+\h_{1,4,9}-\h_{1,5,7}+\h_{1,5,9}+\h_{1,6,7}+\h_{1,6,8}
	 +2 \h_{2,3,9}-\h_{2,4,7}+\h_{2,4,8}+\h_{2,5,7}+\h_{3,4,7}-\h_{6,8,10}+\h_{6,9,10}+\h_{7,8,10}.$}
	\label{fig:g310ppbexpansion}
\end{figure}

\begin{proposition}\label{prop:cyclic-less}
Suppose that $\J$ is a noncrossing tableau and $W(\J)$ is the normal web from \cref{thm:WJnormal}.  Then $\J$ is cyclic-less if and only if $W(\J)$ is cyclic-less.  Thus the bijection of \cref{thm:WJnormal} restricts to a bijection between cyclic-less noncrossing tableaux and cyclic-less normal webs.
\end{proposition}
\begin{proof}
We have to analyze the recursive construction of $W(\J)$, and we let $\K$ be the standardization of $\J$.  We will consider the situation of \cref{lem:strand-triples}, and we will refer to a strand triple as \emph{cyclic} if it is cyclic after the destandardization projection $\iotaStd$.  Suppose that $\J$ contains a cyclic triple $(a^*,b^*,c^*)$.  Let us order the triples of $\K$ in increasing order of the right endpoint, corresponding to the order in which the triples are added in \cref{lem:strand-triples}.  Find the first triple of the form $(a,b,c)$ where $\iotaStd(a) = a^*$ and $\iotaStd(b) = b^*$.  We claim that we will create a cyclic strand triple when we add $(a,b,c)$.  If $\iotaStd(c)= c^*$, this is easy to see.  If $\iotaStd(c) > c^*$, the definition of standardization implies that a triple $(a',b',c')$ satisfying $\iotaStd(b') = b^*$ and $\iotaStd(c') = c^*$ had already been added earlier.  The white strand triple $(a, c_r,\gamma)$ satisfies $c_r \leq b$ and $\gamma \leq c'$, and is therefore cyclic.

For the converse, suppose that $\J$ contains no cyclic triples, and consider one instance of \cref{lem:strand-triples} successively applied to triples in $\K$.  Suppose that $W'$ has no cyclic strand triples and we are adding $(a,b,c)$, a triple in $\K$.  Let us set $(a^*,b^*,c^*):= (\iotaStd(a),\iotaStd(b),\iotaStd(c))$.  Consider the new strand triples of the form $(b_i,c_{i-1}, \gamma)$.  If $(b_i,c_{i-1},\gamma)$ is cyclic then $\iotaStd(b_i),\iotaStd(c_{i-1}), \iotaStd(\gamma)$ are adjacent, which can only happen if $\iotaStd(b_i),\iotaStd(c_i),\iotaStd(\gamma)$ are adjacent.  But we have $b_i < a < c_i <b < \gamma$, so this would imply that $(a,b,c)$ is cyclic which is not the case.  A similar argument holds for triples of the form $(b_i,c_i,\gamma)$ or $(a'_i,b,c'_{i+1})$ or $(a'_i,b,c'_i)$.  

Let us now consider the new strand triple $(a,c_r,\gamma)$.  We have $a < c_r < b < \gamma \leq c$.  If $\gamma = c$, then $(a,c_r,\gamma)$ cyclic would force $(a,b,c)$ to be cyclic, which is not the case.  We may thus suppose that $\gamma =c'_1 < c$.  We have $a < c_r < b < b'_1 < c'_1 < c$ and $\iotaStd(b'_1) \neq \iotaStd(c'_1)$.  So $(a,c_r,c'_1)$ cyclic would force $\iotaStd(c_r) = \iotaStd(b) = \iotaStd(b'_1) = b^*$.  If $(a,c_r,\gamma)$ is cyclic then we would have that $a^*$ and $b^*$ are adjacent, and also $\iotaStd(b'_1)$ and $\iotaStd(c'_1)$ is adjacent.  When the strand triple $(a'_1,b'_1,c'_1)$ in $W'$ was first created (in the process of \cref{lem:strand-triples}), we must have added a triple of the form $(\bar a, \bar b, \bar c)$ where $\iotaStd(\bar b) = b^*= \iotaStd(b'_1)$ and $\iotaStd(\bar c) = \iotaStd(c'_1)$ are adjacent to $b^*$.  It follows that $\J$ contains the cyclic triple $(a^*,b^*,\iotaStd(\bar c))$.  This completes the proof of the converse.
\end{proof}

\section{Proof of \cref{thm:main}(3)}\label{sec:proof-main-3}

We now prove \cref{thm:main}(3).

Suppose that $(Q,\z)$ is normal.  By \cref{prop:normal}, every vertex of $Q$ is normal.  By \cref{prop:Mv-positroid}, the set $M_v$ is a loopless positroid.  By \cref{prop:embedding}, the map $v \mapsto -\bb_v$ sends $V(Q)$ into the integer points of $L(\bp_\bullet)$, and adjacent vertices map to edges of the simplicial complex on $L_\Z(\bp_\bullet)$.  We claim that the map $\v \mapsto -\bb_v$ is injective.  Suppose first that $Q = Q(W)$ for a standard web $W$.  In this case, the difference $\bb_v - \bb_w = \b_v - \b_w$ is given by the strand labels of the edges of some path connecting $v$ and $w$; equivalently, $\b_v - \b_w$ is determined by the set of strands separating $v$ from $w$.  It is known \cite{Pos,LP} that the face labels of a reduced plabic graph, or equivalently, the vertices of the membrane, are distinct in this situation. 

In the general case, $W$ may not be standard but it will be normal, so strands with the same label do not intersect and are parallel.  So for any two distinct vertices $v$ and $w$, there exists some index $i$ such that $v$ and $w$ are separated by some nonzero number of strands labeled $i$, all of which are traveling in the same direction.  This implies that the map $\v \mapsto -\bb_v$ is injective.   It follows that $(Q,\z)$ is a strong scaffold for $\tropP_\bullet(Q,\z)$ in the sense of \cref{def:scaffold}. 

We now prove existence and uniqueness of the cyclic-less normal model.  Let $\tropP_\bullet \in \Trop_{>0} X(3,n)(\Z)$.  By \cref{thm:ELnc}, $\Psi(\tropP_\bullet)$ admits a unique noncrossing tableau decomposition $\Psi(\tropP_\bullet) = \sum_{J \in \J} \t_J$.  By \cref{thm:WJnormal} and \cref{prop:cyclic-less}, the resolution algorithm produces a normal, cyclic-less, non-elliptic web $W(\J)$, and by \cref{thm:nc-to-web}, its planar basis expansion satisfies $\Psi(\tropP_\bullet(W(\J))) = \sum_{J \in \J} \t_J = \Psi(\tropP_\bullet)$.  Since $\Psi$ is injective modulo lineality, $\tropP_\bullet(W(\J)) \equiv \tropP_\bullet$ modulo lineality.

Let $(Q,\z)$ be the dual CAT(0) planar graph of $W(\J)$.  Combining with \cref{thm:pbexpansion}, the Pl\"ucker vector $\tropP_\bullet(Q,\z)$ satisfies $\tropP_\bullet(Q,\z) \equiv \tropP_\bullet$ modulo lineality, so $(Q,\z)$ is a normal, cyclic-less model for $\tropP_\bullet$.  Since each step in the construction $(\tropP_\bullet \text { modulo lineality}) \mapsto \J \mapsto W(\J) \mapsto (Q,\z)$ is a bijection (using \cref{thm:ELnc}, the resolution bijection \cref{thm:WJnormal}, and the web-to-dual-graph correspondence), the normal model is unique.  \qed

		\part{Buildings and membranes}\label{part:buildings}

		\section{Buildings}\label{sec:buildings}
		
\subsection{Preliminaries}
Let $\K = \C((t))$ and $\O = \C[[t]]$ denote the rings of formal Laurent series and formal power series respectively.  The field $\K$ is equipped with a valuation $\val: \K^\times \to \Z$ sending a formal Laurent series to the exponent of the leading (lowest degree) term.  A \emph{lattice} $\Lambda$ in $\K^k$ is a $\O$-submodule of $\K^k$ generated by $k$ linearly independent vectors in $\K^k$.  Two lattices $\Lambda_1$ and $\Lambda_2$ are \emph{equivalent} if $\Lambda_1 = c \Lambda_2$ for $c \in \K^\times$.  We let $[\Lambda]$ denote the equivalence class of lattices containing $\Lambda$.  Two equivalence classes $[\Lambda_1]$ and $[\Lambda_2]$ are \emph{adjacent} if there are representatives satisfying $t\Lambda_1 \subsetneq \Lambda_2 \subsetneq \Lambda_1$.
		
		Let $\B$ denote the affine building of $\PGL_k(\K)$.  This is the flag simplicial complex with vertex set given by the set of equivalence classes $[\Lambda]$ of lattices, and edges given by adjacent equivalence classes.  The simplices of $\B$ are given by tuples of equivalence classes that are pairwise adjacent.		
		
The affine building $\B$ is covered by \emph{apartments}.  An apartment $A$ consists of all equivalence classes of lattices of the form
\begin{equation}\label{eq:lattice}
\Lambda= \sp_\O( t^{-a_1} e_1, t^{-a_2} e_2, \ldots, t^{-a_k} e_k),
\end{equation}
where $e_1,\ldots,e_k$ is a basis of $\K^k$.  We may identify the lattice \eqref{eq:lattice} with the integer point $(a_1,\ldots,a_k) \in \Z^k$.  Restricted to the apartment the equivalence relation is that two lattice points in $\Z^k$ are equivalent if their difference is a multiple of $\one=(1,1,\ldots,1)$.  Thus the vertices of an apartment can be identified with $\Z^k/\one$.

The \emph{fundamental coweights} $\omega_1,\ldots,\omega_{k-1}$ are the vectors $$\omega_i := (1,1,\ldots,1,0,0,\ldots,0) \in \Z^k/\one$$
that form a basis of $\Z^k/\one$.  We call $\Z^k/\one$ the \emph{coweight lattice}.  An element of the coweight lattice is \emph{dominant} if and only if it is a nonnegative sum of fundamental coweights.  The Weyl group $W = S_k$ acts on $\Z^k/\one$ and each element of the Weyl group orbit in the coweight lattice has a unique dominant coweight.  The coweight-valued distance between two points $\a,\a’ \in \Z^k/\one$ is $d(\a,\a’) = |\a’-\a|$, where $|\a|$ denotes the dominant coweight belonging to the Weyl group orbit $W \cdot \a$.  

Let $[\Lambda_1]$ and $[\Lambda_2]$ be any two vertices of $\B$.  Then there exists an apartment $A$ that contains both $[\Lambda_1]$ and $[\Lambda_2]$.  The distance $d([\Lambda_1],[\Lambda_2])$ is defined to be the dominant coweight computed using the distance function of $A$.  This does not depend on the choice of $A$.  Two equivalence classes $[\Lambda_1]$ and $[\Lambda_2]$ are adjacent if and only if the distance $d([\Lambda_1],[\Lambda_2])$ is equal to $\omega_i$ for some $i = 1,2,\ldots,k-1$.  Note that $d([\Lambda_1],[\Lambda_2])$ is not symmetric.  Instead, we have $d([\Lambda_2],[\Lambda_1]) = - w_0 \cdot d([\Lambda_1],[\Lambda_2])$ where the longest permutation $w_0 \in W$ acts by $w_0 \cdot \omega_i = - \omega_{k-i}$.

We shall use the $\omega_1$-distance function
$$
d_1(\a,\a') :=  \langle \omega'_1, d(\a,\a') \rangle, \qquad \omega'_1 = \left(\frac{k-1}{k},-\frac{1}{k},\ldots,-\frac{1}{k}\right).
$$
Here, $\omega'_1$ is the first fundamental weight.  Thus if $\a$ and $\a'$ are connected by an edge in the building, we have
$$
d_1(\a,\a') = \langle \omega'_1,\omega_i \rangle = \frac{k-i}{k}.
$$
Let us convert the one skeleton of the building $\B$ into a directed graph $G(\B)$ where we keep only the edges $\a \to \a'$ where $d_1(\a,\a') = 1/k$.  Equivalently, $d(\a,\a') = \omega_{k-1}$ or $\a-\a' \in W \omega_1$.

\begin{proposition}
The directed graph distance function on $G(\B)$ agrees with $k d_1$.
\end{proposition}
\begin{proof}
First, note that the distance $d([\Lambda_1],[\Lambda_2])$ is the distance of the shortest path from $[\Lambda_1]$ to $[\Lambda_2]$, where the distance of a path is the sum of the distances of the edges of the path.

Clearly the distance function in $G(\B)$ is greater than or equal to that in $\B$, multiplied by $k$.  Let $[\Lambda_1],[\Lambda_2]$ be two vertices connected by an edge in $\B$.  We may assume that $d([\Lambda_1],[\Lambda_2]) = \omega_i$.  Translating, we suppose that $[\Lambda]$ is the vertex $\0$ and $[\Lambda']$ is the vertex $\omega_i$ in some apartment.  These two vertices are part of a simplex with the vertices $\0, \omega_1,\omega_2,\ldots,\omega_{k-1}$.  We have
$$
d_1(\0,\omega_{k-1}) = d(\omega_{k-1},\omega_{k-2}) = \cdots = d(\omega_{i+1},\omega_i) = 1/k.
$$
It follows that the graph distance from $[\Lambda_1]$ to $[\Lambda_2]$ in $G(\B)$ is at most $k-i$.  Since this is $k$ times the $d_1$-distance from $[\Lambda_1]$ to $[\Lambda_2]$ in $G(\B)$, the distance functions agree.  
\end{proof}

Given $[\Lambda] \in \B$, define the color of $[\Lambda]$ to be
$$
c([\Lambda]) := \val \det(\Lambda) \in \Z/k\Z.
$$
Note that $\val \det(c \Lambda) = k \cdot \val(c) + \val \det(\Lambda)$, so the color is well-defined.  Viewed as a function $c(\cdot)$ on the vertices of $G(\B)$, the color function satisfies \cref{def:color}.

\subsection{Le's distance function}
\def\tDelta{\tilde \Delta}

We assume that a basis of $\K^k$ has been fixed.    Then given $v_1,\ldots,v_k \in \K^k$, we have an element $\det(v_1,\ldots,v_k) \in \K$.  We define $\val \det(\Lambda)$ to be the valuation of $\det(v_1,\ldots,v_k)$ for a $\O$-basis $v_1,\ldots,v_k$ of $\Lambda$.  While the determinant itself may not be well-defined, the valuation is well-defined.

Given lattices $\Lambda_1,\ldots,\Lambda_k$, define
\begin{equation}\label{eq:tDelta}
\tDelta(\Lambda_1,\ldots,\Lambda_k):= \min \{\val \det(v_1,\ldots,v_k) \mid v_i \in \Lambda_i\} \in \Z.
\end{equation}
The minimum is achieved by picking generic $v_i \in \Lambda_i$.  Given any vector $v \in \K^k$ and lattice $\Lambda$, we define
$$
c(v, \Lambda):= \min \{\lambda \in \Z \mid t^\lambda v \in \Lambda\}.
$$
We say that $v \in \K^k$ is a \emph{generator} for $\Lambda'$ if it is part of a basis for $\Lambda'$ as a $\O$-module.  Given two lattices $\Lambda, \Lambda'$, we define
$$
c(\Lambda',\Lambda) = \min \{ c(v,\Lambda) \mid v \text{ is a generator for } \Lambda'\}.
$$
We say that $v$ is a \emph{tight generator} for $\Lambda'$ with respect to $\Lambda$ if $c(v, \Lambda) = c(\Lambda',\Lambda)$.

\begin{theorem}[\cite{Le}]\label{thm:Le1}
Let $\Lambda_1,\ldots,\Lambda_k$ be lattices.  There exists a lattice $\Lambda$ such that
$$
\val(\det(\Lambda)) + \sum_{i=1}^k c(\Lambda,\Lambda_i) = \tDelta(\Lambda_1,\ldots,\Lambda_k).
$$
\end{theorem}

For equivalence classes $[\Lambda_1],\ldots,[\Lambda_k] \in \B$, define
$$
\Delta([\Lambda_1],\ldots,[\Lambda_k]):= \frac{1}{k}\left(\sum_{i=1}^k \val \det(\Lambda_i)\right) -\tDelta(\Lambda_1,\ldots,\Lambda_k).
$$
It is straightforward to verify that this function does not depend on the representatives of equivalence classes chosen.

\begin{theorem}[\cite{Le}]\label{thm:Le2}
Let $[\Lambda_1],\ldots,[\Lambda_k]$ be vertices of $\B$.  Then
$$
\Delta([\Lambda_1],\ldots,[\Lambda_k]) = \min_{[\Lambda] \in \B} \left(d_1([\Lambda],[\Lambda_1]) + \cdots + d_1([\Lambda],[\Lambda_k]) \right).
$$
The minimum is attained at a vertex $[\Lambda]$ that satisfies \cref{thm:Le1}.
\end{theorem}

\begin{proof}[Proof of \cref{thm:modelexists}]
Let $X \in \Gr(k,n)(\K)$ be a point in the Grassmannian over $\K$ with tropical Pl\"ucker vector $\tropP_\bullet = \tropP_\bullet(X) \in \R^{\binom{[n]}{k}}$.  Let $x_1,x_2,\ldots,x_n \in \K^k$ be the columns of $X$.  Let $\Lambda^a_i$ be the lattice with basis $t^{-a} x_i, w_2^{(i)},\ldots, w_k^{(i)}$ where the $w$-s are fixed vectors not depending on $a$.  For fixed $J = \{j_1,\ldots,j_k\}$ and sufficiently large $a$, the function
$$
\val(\det(v_1,\ldots,v_k)), \qquad v_i \in \Lambda_{j_i}
$$
takes minimum value at the choice $v_i = t^{-a} x_{j_i} \in \Lambda_{j_i}$.  This is because a generic choice of $v_i$ would give
$$
\det(v_1,\ldots,v_k) = t^{-ka} c \det(x_{j_1},\ldots,x_{j_k}) + \text{ higher order terms},
$$
and by assumption we have $ \det(x_{j_1},\ldots,x_{j_k}) \neq 0$.  Here, $c \in \C$ is a scalar that does not affect the valuation.  Now, choose $\Lambda_i := \Lambda^a_i$ where $a$ is sufficiently large in the above sense for all choices of $J$.  Then our choice gives
$$
\tDelta(\Lambda_{j_1},\ldots,\Lambda_{j_k}) = \tropP_J(X) - ak
$$
and
$$
\Delta([\Lambda_{j_1}],\ldots,[\Lambda_{j_k}]) = -\tropP_J(X) + ak + \frac{1}{k}\sum_{i=1}^k \val \det(\Lambda_{j_i}).
$$
for all $J$.  

We apply \cref{thm:Le2} to $k$-element subsets of $(\Lambda_1,\ldots,\Lambda_n)$.  For each $J \in \binom{[n]}{k}$, there is a distance minimizing lattice $[\Lambda_J]\in \B$.  Consider the $n + \binom{[n]}{k}$ vertices $\{[\Lambda_1],\ldots,[\Lambda_n]\} \cup \{[\Lambda_J] \mid J \in \binom{[n]}{k}\}$.  Let $Q \subset G(\B)$ be a directed subgraph such that the directed distance functions $\dist_{G(\B)}(\cdot,\cdot)$ and $\dist_Q(\cdot,\cdot)$ agree on any pair from these $n + \binom{[n]}{k}$ vertices.  As each $[\Lambda_J]$ was a distance minimizer in $G(\B)$, it is still a distance minimizer in $Q$, with identical Fermat-Le distance functions.  Then $(Q, \z =([\Lambda_1],\ldots,[\Lambda_n]))$ is a model for $\tropP_\bullet(X)$, up to the lineality action by the vector $(\frac{1}{k}\val \det(\Lambda_1)+a,\frac{1}{k}\val \det(\Lambda_2)+a,\ldots, \frac{1}{k}\val \det(\Lambda_n)+a)$.  
\end{proof}

\subsection{Positive configurations}
We review the notion of \emph{positive configuration} from \cite{Le}.
An $n$-tuple of lattices $\Lambda_1,\ldots,\Lambda_n$ is called a \emph{positive configuration}
if we have a collection of ordered bases $v_{i1},\ldots,v_{ik}$ for $\Lambda_i$
such that for each triple of integers $1 \leq p < q<r\leq n$, and each triple of
nonnegative integers $i, j, \ell$ such that $i + j + \ell = k$, 
\begin{enumerate}
\item we have $f^t_{ij\ell}(\Lambda_p,\Lambda_q,\Lambda_r) = - \val \det(v_{p1},\ldots,v_{pi},v_{q1},\ldots,v_{qj},v_{r1},\ldots,v_{r\ell})$, and
\item the leading coefficient of $\det(v_{p1},\ldots,v_{pi},v_{q1},\ldots,v_{qj},v_{r1},\ldots,v_{r\ell})$ is positive.
\end{enumerate}
The tropical functions $f^t_{ij\ell}$ are defined similarly to \eqref{eq:tDelta} by taking $i$ (resp. $j$ and $\ell$) vectors from $\Lambda_p$ (resp. $\Lambda_q$ and $\Lambda_r$).

\section{Diskoids and convex sets}\label{sec:diskoid}
\def\blambda{\mathbf{\lambda}}
\def\maxconv{{\rm maxconv}}
\def\minconv{{\rm minconv}}
In this section, we specialize to $k = 3$.  For a sequence $(\lambda_1,\lambda_2,\ldots,\lambda_n)$ where $\lambda_i \in \{\omega_1,\omega_2\}$, the space of polygons is
$$
P(\blambda) = \{([\Lambda_1],[\Lambda_2],\ldots,[\Lambda_{n}],[\Lambda_{n+1}]=[\Lambda_1]) \mid d([\Lambda_i],[\Lambda_{i+1}])= \lambda_i\}.
$$
The space of polygons $P(\blambda)$ has the structure of an algebraic variety (see \cite{FKK,Akh}) and we say that $P \in P(\blambda)$ is \emph{generic} if it is a generic point in an irreducible component of the polygon space.

Given a generic polygon $P \in P(\blambda)$, we obtain a diskoid $D(P)$ with boundary $P$; the one-skeleton of $D(P)$ (after orienting the edges) can be identified with a CAT(0) planar graph $Q(P)$.  The vertices of $D(P)$ were characterized by Akhmejanov \cite{Akh} as a convex hull.  The max-convex hull of a set $S$ of vertices of $\B$ is the smallest set $\maxconv(S)$ such that $[\Lambda], [\Lambda'] \in \maxconv(S)$ implies $[\Lambda+\Lambda'] \in \maxconv(S)$.  The min-convex hull of a set $S$ of vertices of $\B$ is the smallest set $\minconv(S)$ such that $[\Lambda], [\Lambda'] \in \minconv(S)$ implies $[\Lambda \cap \Lambda'] \in \minconv(S)$.  The convex hull is the intersection $\conv(S) := \maxconv(S) \cap \minconv(S)$.

\begin{theorem}[\cite{Akh}]
For a generic polygon $P$, the set of vertices of $D(P)$ (or of $Q(P)$) is the convex hull $\conv(P)$.
\end{theorem}

\begin{proposition}\label{prop:LeAkh}
A polygon $P \in P(\blambda)$ is a positive configuration of lattices in $\B$ in the sense of \cite{Le} if and only if it is a generic polygon in the sense of \cite{Akh}.  
\end{proposition}
\begin{proof}
By \cite[Theorem 5.25]{Le1}, a configuration of lattices is positive if certain tropical functions $f^t_{i,j,k}$ satisfy the \emph{hive inequalities}.  By \cite{GS}, hives index the top-dimensional irreducible components of polygon spaces, and the functions $f^t_{i,j,k}$ are constant on generic points of each such component.  This is exactly the notion of generic in \cite{Akh}.
\end{proof}

\begin{remark}
In general, the positive configurations $([\Lambda_1],[\Lambda_2],\ldots,[\Lambda_{n}],[\Lambda_{n+1}]=[\Lambda_1])$ of \cite{Le} do not necessarily satisfy $d([\Lambda_i],[\Lambda_{i+1}]) \in \{\omega_1,\omega_2\}$.  However, \cite[Lemma 5.20]{Le1} allows one to fill in the gaps in such positive configurations.
\end{remark}

\begin{proposition}\label{prop:dagree}
For $[\Lambda],[\Lambda'] \in Q(P)$, we have $\dist_{Q(P)}([\Lambda],[\Lambda']) = k \cdot d_1([\Lambda],[\Lambda'])$.
\end{proposition}
\begin{proof}
This follows from \cite[Lemma 5.3 and Theorem 5.5]{FKK}.
\end{proof}

We now discuss our notion of focal point (\cref{def:focal}) in the context of buildings.  Recall that \cref{thm:focal} states that the set of focal points is exactly the set of distance minimizers.  For a lattice, define $V(\Lambda) = \Lambda/t\Lambda$, and let $V(\Lambda,\Lambda') \subset V(\Lambda)$ denote the subspace spanned by the images of tight generators with respect to $\Lambda'$.  Note that any non-zero vector in $V(\Lambda,\Lambda')$ is the image of a tight generator.

\begin{proposition}
Suppose that $[\Lambda]$ is a distance-minimizer for $([\Lambda_a],[\Lambda_b],[\Lambda_c])$ in $Q(P)$.  Then $[\Lambda]$ is also a distance-minimizer in the sense of \cref{thm:Le2}.  In particular, distance-minimizers can be found within $\conv(P)$.
\end{proposition}
\begin{proof}
By \cref{prop:dagree}, the notions of distance in $Q(P)$ and in $\B$ agree.  However, in $\B$ we are minimizing over the infinitely many vertices of $\B$ which could potentially provide for lower distance sums. 

By \cref{thm:focal}, $[\Lambda]$ is a focal point.  We shall show that $\Lambda$ has an $\O$-basis $\{v_a,v_b,v_c\}$ such that $v_a$ is a tight generator for $\Lambda$ with respect to $\Lambda_a$, and similarly for $\Lambda_b$ and $\Lambda_c$.  Equivalently, we find vectors $v_a,v_b,v_c$ from $V(\Lambda,\Lambda_a),V(\Lambda,\Lambda_b),V(\Lambda,\Lambda_c)$ respectively that span the three-dimensional vector space $V(\Lambda)$.  We first note that if $[\Lambda]$ and $[\Lambda']$ are adjacent, then
$$
\dim V(\Lambda,\Lambda') = \begin{cases} 1 & \mbox{if $[\Lambda] \leftarrow [\Lambda']$,} \\
2 & \mbox{if $[\Lambda] \rightarrow [\Lambda']$.}
\end{cases}
$$
Furthermore, \cref{prop:dagree} implies that any three distinct incoming edges $[\Lambda] \leftarrow [\Lambda']$, $[\Lambda] \leftarrow [\Lambda'']$, and $[\Lambda] \leftarrow [\Lambda''']$ satisfy that $V(\Lambda,\Lambda'),  V(\Lambda,\Lambda''), V(\Lambda,\Lambda''')$ span $V(\Lambda)$, and any two such are linearly independent.  If we have a triangle $[\Lambda] \to [\Lambda'] \to [\Lambda'']\to[\Lambda]$ then $V(\Lambda,\Lambda') \supset V(\Lambda,\Lambda'')$.

Consider the standard geodesic $\gamma$ in $Q(P)$ from $[\Lambda]$ to $[\Lambda_a]$, which is also a combinatorial geodesic in $\B$.  Suppose that $\gamma$ has $e$ forward edges and $d$ backwards edges.  Then $d_{\B}([\Lambda],[\Lambda_a]) = d \omega_1 + e \omega_2$.  Let $[\Lambda']$ be the first vertex on the path $\gamma$ after $[\Lambda]$.

Suppose that $e = d = 0$.  Then $V(\Lambda,\Lambda_a) = V(\Lambda)$ is three-dimensional.

Suppose that $e = 0$ and $d > 0$.  Then $d_{\B}([\Lambda],[\Lambda_a]) = d \omega_1$, and after changing bases, we may assume that
$$
\Lambda = \sp_\O(e_1,e_2,e_3), \qquad \Lambda' = \sp_\O(t^{-1} e_1,  e_2, e_3), \qquad \Lambda_a = \sp_\O(t^{-d} e_1,e_2, e_3).
$$
In this case, $V(\Lambda,\Lambda') = V(\Lambda,\Lambda_a) = \sp(e_1)$.

Next suppose that $e > 0$ and $d = 0$.  The first step of $\gamma$ is a forwards edge $[\Lambda] \rightarrow [\Lambda']$.  Then $[\Lambda],[\Lambda'],[\Lambda_a]$ belong to a common apartment.  We may assume that
$$
\Lambda = \sp_\O(e_1,e_2,e_3), \qquad \Lambda' = \sp_\O(t^{-1} e_1, t^{-1} e_2, e_3), \qquad \Lambda_a = \sp_\O(t^{-e} e_1, t^{-e} e_2, e_3).
$$
In this case, $V(\Lambda,\Lambda') = V(\Lambda,\Lambda_a) = \sp(e_1, e_2)$.

Finally, if $e > 0$ and $d> 0$, typically the first step of $\gamma$ is a backwards edge, and
$$
\Lambda = \sp_\O(e_1,e_2,e_3), \qquad \Lambda' = \sp_\O(t^{-1} e_1, e_2, e_3), \qquad \Lambda_a = \sp_\O(t^{-e-d} e_1, t^{-e} e_2, e_3).
$$
In this case, $V(\Lambda,\Lambda') = V(\Lambda,\Lambda_a) = \sp(e_1)$.

However, it is possible for the first step of $\gamma$ to be a forwards edge, and
$$
\Lambda = \sp_\O(e_1,e_2,e_3), \qquad \Lambda' = \sp_\O(t^{-1} e_1, t^{-1} e_2, e_3), \qquad \Lambda_a = \sp_\O(t^{-e-d} e_1, t^{-e} e_2, e_3).
$$
Here, $\sp_\O(t^{-1} e_1, e_2, e_3)$ does not belong to $Q(P)$ (which is possible).  We then have $V(\Lambda,\Lambda') = \sp(e_1, e_2)$ but $V(\Lambda,\Lambda_a) = \sp(e_1)$.  We check that the definition of $[\Lambda_a]$ and $[\Lambda_b]$ parallel at $[\Lambda]$ implies the following: if $[\Lambda_a]$ and $[\Lambda_b]$ are not parallel then $V(\Lambda,\Lambda_a)$ and $V(\Lambda,\Lambda_b)$ span $V(\Lambda,\Lambda')$.

Since $[\Lambda]$ is a focal point, it has a system of distinct representatives.  Using the above case analysis, we see that this in turn gives a basis $\{v_a,v_b,v_c\}$ of $V(\Lambda)$.  It follows that $\Lambda$ is a lattice that satisfies the conclusion of \cref{thm:Le1} and therefore is a distance minimizer in the sense of \cref{thm:Le2}.
\end{proof}

\begin{figure}
	\centering
	\includegraphics[width=0.55\linewidth]{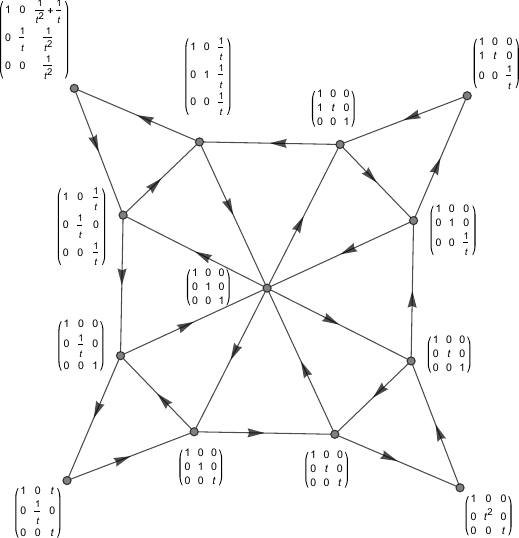}
	\caption{Embedding of the CAT(0) graph $Q$ into the $\emph{PGL}_3(\mathcal{K})$ affine building, with vertices written as matrices in $\text{PGL}_3(\mathcal{K})$.  The curvature is negative at the center vertex $v := (e_1,e_2,e_3)$, making it impossible to find a common basis which simultaneously diagonalizes $v$ and its eight neighbors.  This reflects the fact that $Q$ cannot be embedded into a single apartment.}
	\label{fig:312buildingex}
\end{figure}

	\section{Keel-Tevelev isomorphism}\label{sec:KT}
	Let $M = \{v_1,v_2,\ldots,v_n\} \subset \K^k$ be $n$ vectors that span $\K^k$.  The \emph{Keel-Tevelev membrane} is defined to be
	$$
	[M]:= \{\Lambda = \sp_\O(t^{-u_1} v_1,\ldots,t^{-u_n}v_n) \mid (u_1,u_2,\ldots,u_n) \in \Z^n\} \subset \B.
	$$	
	In \cite{KT}, lower rank lattices (belonging to a compactification $\overline \B$) are allowed; we will only need the part of the membrane consisting of lattices that belong to $\B$.   With $M$ fixed, we define
	$$u_j(\Lambda):= \max(u \mid t^{-u} v_j \in \Lambda) \in \Z$$
	for a lattice $\Lambda$.
	
	Viewing $M$ as a point in $\Gr(k,n)(\K)$, we obtain a tropical Pl\"ucker vector $\tropP_\bullet(M):= \val(\pluck_\bullet(M))$.
	\begin{theorem}[\cite{KT}]\label{thm:KT}
	The map 
	$$
	\iota: \Z^n \to \B, \qquad (u_1,u_2,\ldots,u_n) \mapsto \sp_\O(t^{-u_1} v_1,\ldots,t^{-u_n}v_n)
	$$
	induces an isomorphism of simplicial complexes between $L(\tropP_\bullet(M))$ and $[M]$.
	\end{theorem}
	Note that the membrane $[M]$ depends only on $M$ up to lineality, or equivalently, as a point in $X(k,n)(\K)$.  See also \cite{JSY} for a discussion of the relation between tropical geometry and buildings.

	Suppose that we are given a polygon $P = ([\Lambda_1],[\Lambda_2],\ldots,[\Lambda_{n}],[\Lambda_{n+1}]=[\Lambda_1])$ in $\B$, such that $[\Lambda_i] \in [M]$, where $M \in \Gr(k,n)(\K)$ has tropical Pl\"ucker vector $\tropP_\bullet(M)$.  We assume that $P$ is a positive configuration in the sense of \cite{Le} and thus by \cref{prop:LeAkh} a generic configuration in the sense of \cite{Akh}.  We further assume that the tropical coordinates $f_{ijk}^t(P)$ of \cite{Le,FL} recover $\tropP_\bullet(M)$ up to lineality, and that the tropical coordinates satisfy the \emph{hive inequalities}.  Let $(Q,\z)$ denote the (standard) CAT(0) planar graph associated to $P$.  Then Akhmejanov \cite{Akh} and Fontaine--Kamnitzer--Kuperberg \cite{FKK} define an embedding $\theta: Q \to \B$ so that the boundary vertices $z_i \in \partial Q$ are sent to $\theta(z_i) = [\Lambda_i] \in P$.  
	
	Let $v_1,\ldots,v_n$ be the columns of $M$.  Pick representatives $\Lambda_i \in [\Lambda_i]$ and ordered bases of $\Lambda_1,\Lambda_2,\ldots$ so that the leading generator of $\Lambda_j$ is $v_j$.  	After replacing $\Lambda_i$ by $t^{a_i} \Lambda_i$ and $v_i$ by $t^{a_i} v_i$ for each $i$, we may assume that 
	$$
	c(z_i) := \val(\det(\Lambda_i)) \in \{0,1,2\} \qquad \text{for all } i.
	$$
The assumption that $P$ is a positive configuration implies that
	$$
	c(\Lambda_i,\Lambda_j) = c(v_j, \Lambda_j) = -u_j(\Lambda_i).
	$$
	We have the following equality \cite[Remark 5.3]{Le}:
	$$
	d_1([\Lambda_i],[\Lambda_j]) = -c(\Lambda_i,\Lambda_j) -\frac{1}{3} \val(\det(\Lambda_i)) + \frac{1}{3} \val(\det(\Lambda_j)).
	$$
	Rearranging, we get
	$$
	u_j (\Lambda_i) = d_1([\Lambda_i],[\Lambda_j]) + \frac{1}{3}\left( \val(\det(\Lambda_i)) - \val(\det(\Lambda_j))\right).
	$$
	
	Comparing with \eqref{eq:mu} and using \cref{prop:dagree} we have $\mu(z_i) = (\mu_1(z_i),\ldots,\mu_n(z_i))$ where
	\begin{equation}\label{eq:muu}
	\mu_j(z_i) = \frac{1}{3}\left(\delta_Q(z_i,z_j) + c(z_i) - c(z_j)\right) = d_1([\Lambda_i],[\Lambda_j]) + \frac{1}{3}\left( \val(\det(\Lambda_i)) - \val(\det(\Lambda_j))\right) = u_j(\Lambda_i).
	\end{equation}

	\begin{proposition}\label{prop:LPKT}
	In the above situation, the two embeddings $\theta$ and $\iota \circ \mu$ of $(Q,\z)$ are identical.  Consequently, the Lam-Postnikov membrane $\mu(D)$ associated to $Q$ is embedded into the Keel-Tevelev membrane $[M]$ via $\iota$.
	\end{proposition}
	\begin{proof}
	Since we have assumed that $\Lambda_i \in [M]$, using the equality \eqref{eq:muu} we have 
	$$
	\iota \circ \mu(z_i) = \sp_\O(t^{-u_1} v_1,\ldots,t^{-u_n}v_n) = \Lambda_i = \theta(z_i).
	$$
	It is known \cite{Akh,FKK} that the rest of the embedding $\theta$ is determined by metric considerations.  By \cite{Akh}, $\theta(Q) \subset \conv(P) \subset \maxconv(P) \subset [M]$.  The last inclusion follows from the fact that $[M]$ is closed under sums of lattices.  
	
	For each interior vertex $v \in Q$, and each boundary vertex $z_j \in \partial Q$, we can find a CAT(0) planar subgraph $Q'$ such that both $v$ and $z_j$ lie on the boundary.  The embedding $\theta(\partial Q')$ will be a generic polygon, and thus a positive configuration by \cref{prop:LeAkh}.  It follows that the computation \eqref{eq:muu} applies to $\theta(v)$, and we have $u_j(\theta(v)) = \mu_j(v)$.  We have observed that $\theta(v) \in [M]$ so we conclude that $\theta(v) = \iota(\mu(v))$.
	\end{proof}
	In particular, \cref{prop:LPKT} establishes \cref{thm:LPKT}.

		\section{Membranes and positroids}
		\label{sec:positroids}
		
		We describe the matroids $M_v$ for a normal CAT(0) planar graph $(Q,\z)$ explicitly.

		\begin{proposition}\label{prop:rank3}
			Let $M$ be a simple positroid of rank 3 on $[n]$.  Let $\A$ be the set of maximal (under containment) cyclic intervals $[a,b]$ in $[n]$ such that $M|_{[a,b]}$ has rank two, and such that $|[a,b]| \geq 3$.  Then 
			$\A$ is a collection of proper cyclic intervals of size at least three such that any two consecutive elements $a,a+1$ are contained in at most one interval of $\A$.  Conversely, any such collection $\A$ uniquely determines a simple positroid $M$ of rank 3.
		\end{proposition}
		\begin{proof}
			By definition the collection $\A$ is an antichain.  Denote by $(I_1,I_2,\ldots,I_n)$ the Grassmann necklace of $M$.  Then the condition that $M$ is simple is equivalent to $I_a = \{a,a+1,a'\}$ for all $a$.   Let $<_a$ denote the order $a <_a a+1 <_a a+2 < \cdots <_a a-1$.  We have that
			\begin{align}
				a' = \begin{cases}  \max_{<_a}(c \mid [a,a+1] \subset [b,c] \in \A)+1 & \mbox{if $[a,a+1] \subset [b,c]$ for some $[b,c] \in \A$,} \\
					a+2 &\mbox{if $[a,a+1]$ is not contained in any interval in $\A$.}
				\end{cases}
				\label{eq:Ia}
			\end{align}
			Now, if $a' \neq a+2$ it follows that from the definition of Grassmann necklace that we must have $I_{a+2} = \{a+1,a+2,a'\}$, so there cannot be an interval $[a+1,c'] \in \A$.  It follows that the two consecutive elements $a+1,a+2$ is contained in at most one interval of $\A$.  This proves the forward direction of the theorem.  
			
			For the converse direction, note that \eqref{eq:Ia} uniquely determines a Grassmann necklace as long as $\A$ satisfies the stated conditions.
		\end{proof}

		\begin{proof}[Proof of \cref{prop:Mv-positroid}]
		Since $(Q,\z)$ is normal, $M_v$ is non-empty by \cref{prop:normal}.
			By \cref{thm:focal}, $v$ is a distance minimizer for $z_i,z_j,z_k$ if and only if it is a focal point for $z_i,z_j,z_k$.  Consider the sets $F(v,z_1),F(v,z_2),\ldots$.  By \cref{lem:cyclic}, these sets are arranged (weakly) cyclically.  We define a positroid $M$ by imposing the following rank conditions: 
			\begin{enumerate}
				\item
				whenever $z_i$ and $z_{i+1}$ are parallel we declare that $i$ and $i+1$ are parallel;
				\item
				whenever $|\bigcup_{i \in [a,b]} F(v,i)| = 2$ we declare that $a,a+1,\ldots,b$ has rank two.
			\end{enumerate}
			After removing parallel elements, whenever we have $|\bigcup_{i \in [a,b]} F(v,i)| = 2$, we must have that $|F(v,i)|=2$ for all $i \in [a+1,b-1]$.  It follows from this that the condition in \cref{prop:rank3} holds, and indeed we have a well-defined rank three positroid $M$.  It is then straightforward to see from \cref{def:focal} that $M = M_v$.
			
			Since each $F(v,z_i)$ is nonempty, it follows that for each $i$, there must be a triple $z_i,z_j,z_k$ for which $v$ is a focal point.  It follows that $M_v$ is loopless.
		\end{proof}
		
				\begin{proof}[Proof of \cref{thm:membrane}]
	By \cref{thm:main}(3), $(Q,\z)$ is a strong scaffold and thus $\mu$ defines an embedding into $L_\Z(\bp_\bullet)$.  
	It remains to check the condition: if $v \to w$ is an edge of $Q$ then $\mu(w) - \mu(v) = e_{[i+1,j]}$.  Here, $\eta_i$ (resp. $\eta_j$) is the strand passing through the dual edge $e$ of $W$ with $v$ on the right (resp. $v$ on the left).   Let $\mu(w) - \mu(v) \equiv e_J$ for $J \subset [n]$.
	
	Suppose $v$ is not on the strand strip labeled $a$.  Then by \cref{lem:strandstripF} it follows that $a-1$ and $a$ are either both in $J$ or both not in $J$.  On the other hand, if $v$ is on one side of the strand strip and $w$ is on the other side then $|J \cap \{a-1,a\}| = 1$.  This happens exactly when $a = i$ or $j$.  
			
	For $a = j$: by definition, $w$ is on the right of $\eta_j$, meaning $w$ is on the $z'_j$ side of the strand strip.  By \cref{lem:strandstrip}, $\dist(w,z_j) - \dist(v,z_j) = +2$, so $j \in J$.  For $a = i$: by definition, $v$ is on the right of $\eta_i$, meaning $v$ is on the $z'_i$ side of the strand strip.  By \cref{lem:strandstrip}, $\dist(w,z_i) - \dist(v,z_i) = -1$, so $i \notin J$.  We conclude that $J = [i+1,j]$.
	\end{proof}

\begin{corollary}\label{cor: matroid Mv}
Let $(Q,\z)$ be a normal CAT(0) planar graph and $v \to w$ be an edge such that $M_v$ and $M_w$ are distinct connected matroids, and let $(i,j)$ be as in \cref{def:membrane}.  Then in the matroid subdivision $\Delta(\tropP_\bullet(Q,\z))$, the two matroid polytopes $P_{M_v}$ and $P_{M_w}$ share a facet on the hyperplane $x_{[i+1,j]}:=x_{i+1} + \cdots + x_j =2$ with $P_{M_v}$ on the side $x_{[i+1,j]} \leq 2$ and $P_{M_w}$ on the side $x_{[i+1,j]} \geq 2$.
\end{corollary}
\begin{proof}
Since $M_v$ and $M_w$ are connected, the two matroid polytopes $P_{M_v}$ and $P_{M_w}$ are full-dimensional, and thus $\mu(v),\mu(w)$ are vertices in the matroid complex structure on $L(\bp_\bullet)$.  The vector $e_{[i+1,j]} = \mu(w)-\mu(v)$ is normal to the hyperplane separating these two polytopes in the subdivision $\Delta(\tropP_\bullet(Q,\z))$.  The last statement is obtained by going through the analysis in the proof of \cref{thm:membrane}.
\end{proof}

\begin{figure}[h!]
	\centering
	\includegraphics[width=0.7\linewidth]{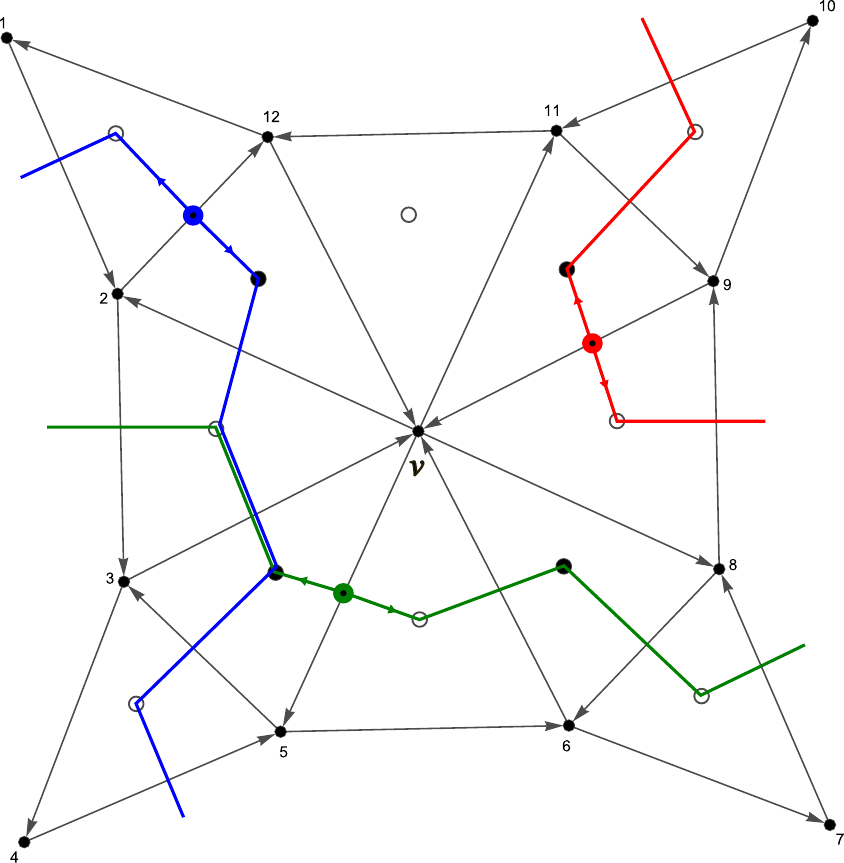}
	\caption{Illustration for \cref{example: octagon end}; the red, blue and green strands emanating from the edges lead us to compute the facets of the matroid polytope $P_{M_v}$ for the vertex $v$ at the center of $Q$.}
	\label{fig:bigexample312b}
\end{figure}

\begin{example}\label{example: octagon end}
	We illustrate \cref{cor: matroid Mv} by extracting the facet inequalities cutting out the matroid polytope $P_{M_v}$ for $M_v$ corresponding to the vertex $v$ at the center of the octagon in \cref{fig:bigexample312b}.  See also \cref{fig:bigexample312} from before. 
	
	
	Starting from the edge $(9,v)$ with the large red dot and following the red strand paths to the boundaries in both directions induces a partition into cyclic intervals $[11,8]\sqcup [9,10] = [12]$.  This defines a half-region $x_{[11,8]}\ge 2$ which is one of the facet inequalities for $P_{M_v}$.
	
	Similarly following the green strand paths from the edge $(v,5)$ with the large green dot to the two boundaries induces a partition into cyclic intervals $[8,2]\sqcup [3,7] = [12]$, and again we have defined one of the facet inequalities $x_{[8,2]}\ge 1$ cutting out $P_{M_v}$.
	
	Taking all index relabelings under $j\mapsto j+3$ mod $12$ gives all eight facet inequalities for $P_{M_v}$:
	\begin{eqnarray*}
	&x_{[2,11]} \ge 2,\ x_{[5,2]} \ge 2,\ x_{[8,5]} \ge 2,\ x_{[11,8]} \ge 2&\\
	&x_{[5,11]} \ge 1,\ x_{[8,2]} \ge 1,\ x_{[11,5]} \ge 1,\ x_{[2,8]} \ge 1&.
	\end{eqnarray*}

	Finally, following the blue strand paths from the edge $(2,12)$ with the large blue dot to the two boundaries induces a partition into cyclic intervals $[2,4]\sqcup [5,1] = [12]$.  This defines a half-region $x_{[5,2]}\ge 2$.
	
	Recall from \cref{fig:312canvas} the two additional fins occurring in the tropical linear space, corresponding to two additional connected positroid polytopes in the subdivision of $\Delta(3,12)$.  These cannot be immediately read off of $Q$.
\end{example}

\appendix

\section{Tabulation of Rays}\label{sec:appendix}
Here we tabulate normal models and noncrossing tableaux for all rays of $\Trop_{>0}X(3,n)$ for $n = 7,8$.  We identify rays with the primitive integer point $\r$ that generates the ray.  Using \cref{thm:main}(3) and \cref{thm:nc-to-web}, we identify such an integer point with a noncrossing tableau $\J = (J_1,\ldots,J_r)$.  The \emph{NC weight} of the ray is then equal to the size $r$ of the noncrossing tableau.

\subsection{$\Trop_{>0}X(3,7)$}
$\Trop_{>0}X(3,7)$ has $28$ rays of NC weight one, and 14 rays of NC weight two.  They have the following planar basis and noncrossing tableau expansions:
$$
\begin{array}{|c|c|c|}
\hline
	\h_{i_1i_2i_3} &-\h_{j_1j_3j_5}+\h_{j_2j_3j_5}+\h_{j_1j_4j_5}+\h_{j_1j_3j_6} & -\h_{j_2j_4j_6}+\h_{j_1j_2j_4}+\h_{j_2j_5j_6}+\h_{j_3j_4j_6} \\
	\t_{i_1i_2i_3} &\t_{j_1j_4j_5}+\t_{j_2j_3j_6} & \t_{j_1j_2j_4}+\t_{j_3j_5j_6}  \\
	\hline
\end{array}
$$
for all $i_1i_2i_3 \in \binom{[7]}{3}^{\text{ncyc}}$ and all subsets $\{j_1,\ldots, j_6\} \in \binom{\lbrack 7\rbrack}{6}$ with $j_1<\cdots<j_6$.

\subsection{$\Trop_{>0}X(3,8)$}
$\Trop_{>0}X(3,8)$ has the following rays:
\begin{itemize}
	\item $48$ rays with NC weight one (the scaffolds consist of a single CCW triangle with vertices decorated with ordered set partitions $([a+1,b],[b+1,c],[c+1,a])$ such that $a,b,c$ is not a cyclic interval),
	\item $56$ rays with NC weight two (the scaffolds are the one-skeleton of the second dilate of a triangle, as in the rightmost diagram in \cref{fig:examples36}), where the six vertices are variously labeled,
	\item $16$ rays of NC weight three.  There are two distinct orbits under cyclic relabeling, with representatives
	\begin{align}\label{equation: 38 rays}
	\begin{split}
		\pi_\bullet & =  -\h_{j_1 j_4 j_6}+\h_{j_2 j_4 j_6}+\h_{j_1 j_5 j_6}-\h_{j_1 j_3 j_7}+\h_{j_2 j_3 j_7}+\h_{j_1 j_4 j_7}+\h_{j_1 j_3 j_8}\\
		\pi'_\bullet & =  -\h_{j_1 j_4 j_6}+\h_{j_3 j_4 j_6}+\h_{j_1 j_5 j_6}-\h_{j_1 j_3 j_7}+\h_{j_2 j_3 j_7}+\h_{j_1 j_4 j_7}+\h_{j_1 j_3 j_8}.
	\end{split}
	\end{align}
\end{itemize}

\begin{figure}[h!]
	\centering
	\begin{tikzpicture}[scale=0.5]
		\tikzset{->-/.style={decoration={markings, mark=at position 0.55 with {\arrow{stealth}}},postaction={decorate}}}
		
		\pgfmathsetmacro{\h}{1.732}
		
		\coordinate (1) at (3, 3*\h);
		\coordinate (2) at (2, 2*\h);
		\coordinate (8) at (4, 2*\h);
		\coordinate (u) at (1, \h);
		\coordinate (v) at (3, \h);
		\coordinate (7) at (5, \h);
		\coordinate (3) at (0, 0);
		\coordinate (4) at (2, 0);
		\coordinate (5) at (4, 0);
		\coordinate (6) at (6, 0);
		
		\draw[->-] (1) -- (2);
		\draw[->-] (2) -- (u);
		\draw[->-] (u) -- (3);
		\draw[->-] (3) -- (4);
		\draw[->-] (4) -- (5);
		\draw[->-] (5) -- (6);
		\draw[->-] (6) -- (7);
		\draw[->-] (7) -- (8);
		\draw[->-] (8) -- (1); 
		
		\draw[->-] (8) -- (v);
		\draw[->-] (v) -- (4);
		\draw[->-] (7) -- (5);
		\draw[->-] (5) -- (v);
		\draw[->-] (4) -- (u);
		\draw[->-] (v) -- (2);
		\draw[->-] (2) -- (8);
		\draw[->-] (v) -- (7);
		\draw[->-] (u) -- (v);
		
		\foreach \vx in {1,2,3,4,5,6,7,8,u,v} {
			\fill[black!60] (\vx) circle (2.5pt);
		}
		
		\node at ($(1)+(0,0.35)$)       {\small $1$};
		\node at ($(2)+(-0.3,0.1)$)     {\small $2$};
		\node at ($(8)+(0.3,0.1)$)      {\small $8$};
		\node at ($(7)+(0.3,0)$)        {\small $7$};
		\node at ($(3)+(-0.25,-0.35)$)  {\small $3$};
		\node at ($(4)+(0,-0.35)$)      {\small $4$};
		\node at ($(5)+(0,-0.35)$)      {\small $5$};
		\node at ($(6)+(0.25,-0.35)$)   {\small $6$};
		\node at ($(v)+(-0.15,-0.25)$)  {\small $v$};
	\end{tikzpicture}
	\qquad
	\begin{tikzpicture}[scale=0.5]
		\tikzset{->-/.style={decoration={markings, mark=at position 0.55 with {\arrow{stealth}}},postaction={decorate}}}
		
		\pgfmathsetmacro{\h}{1.732}
		
		\coordinate (5) at (2, 2*\h);
		\coordinate (3) at (4, 2*\h);
		\coordinate (v) at (1, \h);
		\coordinate (4) at (3, \h);
		\coordinate (2) at (5, \h);
		\coordinate (6) at (0, 0);
		\coordinate (7) at (2, 0);
		\coordinate (8) at (4, 0);
		\coordinate (1) at (6, 0);
		
		\draw[->-] (6) -- (7);
		\draw[->-] (7) -- (8);
		\draw[->-] (8) -- (1);
		\draw[->-] (1) -- (2);
		\draw[->-] (2) -- (3);
		\draw[->-] (3) -- (4);
		\draw[->-] (4) -- (5);
		\draw[->-] (5) -- (v);
		\draw[->-] (v) -- (6);
		
		\draw[->-] (v) -- (4);
		\draw[->-] (4) -- (2);
		\draw[->-] (7) -- (v);
		\draw[->-] (4) -- (7);
		\draw[->-] (8) -- (4);
		\draw[->-] (2) -- (8);
		
		\foreach \vx in {1,2,3,4,5,6,7,8,v} {
			\fill[black!60] (\vx) circle (2.5pt);
		}
		
		\node at ($(5)+(-0.3,0.3)$)    {\small $4$};
		\node at ($(3)+(0.3,0.3)$)     {\small $3$};
		\node at ($(v)+(-0.3,0)$)      {\small $5$};
		\node at ($(4)+(-0.1,0.35)$)   {\small $v$};
		\node at ($(2)+(0.3,0)$)       {\small $2$};
		\node at ($(6)+(-0.25,-0.3)$)  {\small $6$};
		\node at ($(7)+(0,-0.35)$)     {\small $7$};
		\node at ($(8)+(0,-0.35)$)     {\small $8$};
		\node at ($(1)+(0.25,-0.3)$)   {\small $1$};
	\end{tikzpicture}
	\caption{The two normal models for $\pi_\bullet$ and $\pi'_\bullet$, respectively, in Equation \eqref{equation: 38 rays}.}
	\label{fig:38 example A}
\end{figure}
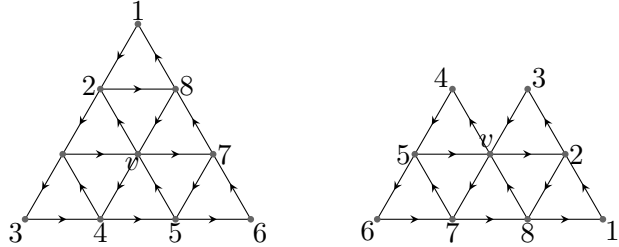

\subsection{$\Trop_{>0}X(3,9)$}
By \cite{ELnc}, the bounded complex of a positive tropical linear space $L(\pi_\bullet)$ embeds into the dilated standard alcove $\PK(\pi_\bullet)\cdot \Delta_{\std}$.

The positive tropical Grassmannian $\Trop_{>0}X(3,9)$ has $75$ planar basis rays (of NC weight one) and $168$ rays with NC weight two, in bijection with the noncrossing but not weakly separated pairs.  There are $156$ NC weight three rays, $69$ NC weight four rays and $3$ NC weight five rays.  More compactly:

$$
\begin{tabular}{|c|c|c|c|c|}
	\hline
	1 & 2 & 3 & 4 & 5 \\
	\hline
	75 & 168 & 156 & 69 & 3 \\
	\hline
\end{tabular}
$$
\subsection{$\Trop_{>0}X(3,10)$}
The tabulation of the 3080 rays of $\Trop_{>0}X(3,10)$ by NC weight is as follows:
$$
\begin{tabular}{|c|c|c|c|c|c|c|}
	\hline
	1 & 2 & 3 & 4 & 5 & 6 & 7 \\
	\hline
	110 & 420 & 840 & 970 & 560 & 170 & 10 \\
	\hline
\end{tabular}
$$

	\end{document}